\documentclass[manuscript,screen]{acmart}
\setcopyright{none}
\acmJournal{TALG}

\usepackage{algorithm}
\usepackage{algpseudocode}
\algrenewcomment[1]{\(\triangleright\) #1}
\usepackage{tikz}
\usetikzlibrary{calc}
\usetikzlibrary{patterns}
\usetikzlibrary{fit}
\usetikzlibrary{decorations.pathreplacing}
\usepackage{pgfplots}
\pgfplotsset{compat=1.18}
\pgfplotsset{pcfplot/.style={
  tick label style={font=\scriptsize}, label style={font=\small},
  legend style={font=\scriptsize, at={(0.5,1.03)}, anchor=south,
                legend columns=2, draw=none, fill=none,
                /tikz/every even column/.append style={column sep=8pt}},
  grid=major, grid style={gray!25}}}
\usepackage[capitalise]{cleveref}

\definecolor{greenback}{rgb}{0.75, 0.99, 0.75}
\definecolor{redback}{rgb}{0.99, 0.75, 0.75}
\definecolor{blueback}{rgb}{0.8, 0.8, 0.99}
\definecolor{darkgreen}{rgb}{0.0, 0.5, 0.0}
\definecolor{orangeback}{rgb}{1.0, 0.6, 0.4}
\definecolor{yellowback}{rgb}{0.98, 0.85, 0.37}

\DeclareOldFontCommand{\cal}{\normalfont\rmfamily}{\mathcal}

\theoremstyle{acmplain}
\newtheorem{observation}{Observation}
\theoremstyle{acmdefinition}
\newtheorem{remark}{Remark}

\begin{document}

\title[Parametric maximum closure on forests]{Parametric maximum closure on precedence forests}

\author{Valerio Dose}
\orcid{0000-0002-8186-0023}
\affiliation{%
  \institution{Sapienza Universit\`a di Roma}
  \department{Dipartimento di Ingegneria Informatica, Automatica e Gestionale}
  \city{Rome}\country{Italy}}
\email{valerio.dose@uniroma1.it}

\author{Fabio Furini}
\orcid{0000-0002-1839-5827}
\affiliation{%
  \institution{Sapienza Universit\`a di Roma}
  \department{Dipartimento di Ingegneria Informatica, Automatica e Gestionale}
  \city{Rome}\country{Italy}}
\email{fabio.furini@uniroma1.it}

\author{Marco Locatelli}
\orcid{0000-0001-7138-8653}
\affiliation{%
  \institution{Universit\`a di Parma}
  \department{Dipartimento di Ingegneria e Architettura}
  \city{Parma}\country{Italy}}
\email{marco.locatelli@unipr.it}

\renewcommand{\shortauthors}{Dose, Furini and Locatelli}

\begin{abstract}
The \textit{Maximum Closure Problem} asks for a maximum-weight closed subset of a precedence-constrained set of vertices; when vertex weights depend affinely on a scalar parameter $\lambda$, as in open-pit mine scheduling, the goal becomes computing the optimal closed set for every value of $\lambda$ at once. We address this problem when the precedence graph is a directed forest. We present two algorithms (\textit{Peel-and-Contract}, {\tt PaC}, and its dual {\tt DPaC}) that compute the full parametric solution -- 
a canonical sequence of disjoint closure layers inducing the nested optimal closures -- by iteratively aggregating vertices along precedence arcs, with time-complexity $\mathcal{O}(n^2)$ for general forests and space-complexity $\mathcal{O}(n)$; heap-based implementations ({\tt HPaC}, {\tt DHPaC}) have a worse time-complexity of $\mathcal{O}(n^2\log n)$, but on random instances their computing times empirically grow as $n\log n$. When the forest is restricted to in-trees or out-trees, specialized heap-based variants ({\tt HIPaC}, {\tt HOPaC}) achieve time-complexity $\mathcal{O}(n\log n)$. We then present a further aggregation-based algorithm (\textit{Rake-and-Compress}, {\tt RaC}) with time- and space-complexity $\mathcal{O}(n\log n)$ also for general forests, which matches the $\Omega(n\log n)$ lower bound for comparison-based algorithms. Computational experiments on random and structured forests with up to $10^5$ vertices characterize the practical behavior of the proposed algorithms, identify the topologies on which {\tt HPaC} and {\tt RaC} are preferable, and compare {\tt HPaC} with the public implementation of the fully parametric pseudoflow algorithm, the state of the art for arbitrary precedence graphs, which it outperforms on forests by one to more than three orders of magnitude.
\end{abstract}



\maketitle

\section{Introduction}
In this paper we are interested in the parametric \emph{Maximum Closure Problem} (MCP) with affine costs.
We first give a formal description of the problem.
Let $G=(V,A)$ be a directed acyclic graph, where the vertex set $V:=\{1,2,\ldots,n\}$ has $|V|=n$ vertices, the arc set $A$ has $|A|=m$ arcs, and a directed arc $(i,j)\in A$ establishes that $i$ precedes $j$, i.e., that inclusion of vertex $i$ implies inclusion of vertex
$j$. 
We denote with
\[
{\mathcal S}=\left\{S\subseteq V\ :\ i\in S\ \Rightarrow\ j\in S,\ (i,j)\in A\right\},
\]
the collection of subsets which fulfill the precedence constraints. 
It is easy to prove that ${\mathcal S}$ is a lattice when equipped with the intersection ($\cap$) and union ($\cup$) operations, i.e., $S_1,S_2\in {\mathcal S}$ implies $S_1\cap S_2, S_1\cup S_2\in {\mathcal S}$.
To each vertex $i\in V$ we associate an integer profit $p_i$, possibly negative, and an integer weight $w_i>0$.
For a given subset $S$ of vertices, we denote with $P(S)=\sum_{i\in S} p_i$ its overall profit, and with $W(S)=\sum_{i\in S} w_i$ its overall weight. 
\newline\newline\noindent
We define function $\varphi\colon(-\infty,+\infty)\rightarrow \mathbb R$ as follows
\begin{equation}
\label{eq:UPIT_parametric}
\varphi(\lambda):=\max_{S\in {\mathcal S}} \big\{P(S)-\lambda W(S)\big\} =
\max_{ {\boldsymbol{x}} \in \{0,1\}^V   } \Big\{ \sum_{i\in V} (p_i-\lambda\,w_i)\,x_i:~
  x_i-x_j \le 0,~  (i,j)\in A
 \Big\}.
\end{equation}
The parametric MCP is the problem of computing the set of optimal solutions $X^\star(\lambda)$ of problem (\ref{eq:UPIT_parametric}) as a function of $\lambda$.
Note that for each $\lambda$ also the set $X^\star(\lambda)$  is a lattice when equipped with the intersection ($\cap$) and union ($\cup$) operations, so that it admits a maximal member (the union of all optimal solutions) and a minimal member (the intersection of all optimal solutions).
Function $\varphi$ is a convex piece-wise linear function with breakpoints 
\[\lambda_0=+\infty>\lambda_1>\lambda_2>\ldots>\lambda_k>\lambda_{k+1}=-\infty,\]
decreasing for $\lambda< \lambda_1$ and constant (equal to 0) for $\lambda\geq \lambda_1$. 
It is known (see, e.g., \cite[Section 2.5]{HC00}) that for a monotone sequence of $\lambda$ values, the set of optimal solutions $X^\star(\lambda)$ gives rise to a nested sequence of subsets of vertices ${\mathcal M}_0\subset {\mathcal M}_1\subset {\mathcal M}_2 \subset \cdots \subset {\mathcal M}_k$, such that ${\mathcal M}_0=\emptyset$,  ${\mathcal M}_k=V$, and ${\mathcal M}_r\in X^\star(\lambda)$ for $\lambda\in (\lambda_{r+1},\lambda_{r}]$, $r\in\{1,2,\ldots,k\}$. We call this {\em nestedness property}.
The breakpoints $\lambda_r$, with $r\in\{1,2,\ldots,k\}$, are $\lambda$ values such that the set
$X^\star(\lambda)$
of optimal solutions of problem (\ref{eq:UPIT_parametric}) is {\em not} a singleton. 
More precisely, the single optimal solution of problem  (\ref{eq:UPIT_parametric}) over the open interval $(\lambda_{i+1},\lambda_{i})$ is the set ${\mathcal M}_i$, while the single optimal solution 
over the open interval $(\lambda_{i},\lambda_{i-1})$ is the set ${\mathcal M}_{i-1}$. For $\lambda=\lambda_i$ the set of optimal solutions is not a singleton and, moreover, as previously mentioned, it is a lattice 
whose maximal member is ${\mathcal M}_i$, while the minimal one is ${\mathcal M}_{i-1}$.
The computation of the nested sequence ${\mathcal M}_1\subset {\mathcal M}_2 \subset \cdots \subset {\mathcal M}_k$ is also known in the literature as {\em parametrization} (see \cite{HC00}). 
It also holds that the sequence of disjoint subsets of vertices defined as follows:
\[
{\mathcal I}_1={\mathcal M}_1,\ {\mathcal I}_2={\mathcal M}_2\setminus {\mathcal M}_1, \ldots,  \ {\mathcal I}_k={\mathcal M}_k\setminus {\mathcal M}_{k-1},
\]
is strictly decreasing with respect to the profit-to-weight ratio, i.e.:
\[
\frac{P_r}{W_r}>\frac{P_{r+1}}{W_{r+1}},\quad r\in\{1,2,\ldots,k-1\},
\]
where
\[
P_r:=\sum_{i\in\mathcal I_r}p_i,\qquad W_r:=\sum_{i\in\mathcal I_r}w_i.
\]
The following simple result establishes a relevant property of the subsets $\mathcal I_r$, $r\in\{1,2,\ldots,k\}$ (see, e.g., \cite[Theorem 9]{Vallet76}), stating that each $\mathcal I_r$ is a subset with maximum profit-to-weight ratio within the set
of vertices which have not been included in the previous subsets $\mathcal I_1,\mathcal I_2,\ldots,\mathcal I_{r-1}$.

\begin{proposition}
\label{prop:optratio}
For each $r\in\{1,2,\ldots,k\}$, the subset of vertices $\mathcal I_r$ is the maximal optimal solution of problem (\ref{eq:UPIT_parametric}) when $\lambda=\lambda_r$ and the set of vertices is restricted to $V \setminus \left[\cup_{j=1}^{r-1} \mathcal I_j\right]$.
\end{proposition}
In what follows we call {\em closure layers} the subsets ${\mathcal I}_1,{\mathcal I}_2,\ldots,{\mathcal I}_k$, and we call the sequence of such subsets an {\em optimal sequence of closure layers}.
The breakpoints of $\varphi$ are equal to the ratio of the different closure layers, i.e.,
\[
\lambda_r=\frac{P_r}{W_r}, \ r\in\{1,2,\ldots,k\}.
\]
For each vertex $v\in V$, we introduce the following definition.
\begin{definition}
\label{def:threshold}
The {\em threshold} of a vertex $v$, denoted by $\theta_v$, is the breakpoint $\lambda_r$ such that $v\in {\mathcal M}_r$ but $v\not\in {\mathcal M}_{r-1}$, or, equivalently, $v\in {\mathcal I}_r$. In other words, vertex $v$ belongs to the optimal solution of  problem (\ref{eq:UPIT_parametric}) for $\lambda\in (\lambda_{r+1},\lambda_{r})$ but not to the one for $\lambda\in (\lambda_{r},\lambda_{r-1})$. 
\end{definition}
We have that
the $r$-th closure layer of the optimal sequence is made up of the vertices which share the same threshold, i.e.,
\[
{\mathcal I}_r=\{v\in V\ :\ \theta_v=\lambda_r\}.
\]
\cref{fig:closure_layers_example} illustrates the problem and its optimal solutions on a small running instance. We make the following simple observation about minimal time and space requirements for algorithms returning closures and breakpoints.
\begin{observation}
\label{obs:opttimespace}
Any comparison-based algorithm computing closures and the related breakpoints requires time $\Omega(n \log n)$. Indeed, even for empty forests, deciding whether the $n$ vertex ratios are pairwise distinct is the element-distinctness problem, which requires $\Omega(n\log n)$ comparisons \cite{BenOr83}; the bound therefore does not rely on requiring the breakpoints in sorted order. If an ordered set of breakpoints is required, the task also entails ordering these ratios. Moreover, the minimal space required by an algorithm for MCP, including the space needed for the input of the problem, is at least $\Omega(n)$ since this is the memory required for the profit and weights associated to vertices.
\end{observation}

\begin{figure}[htbp]
\centering
\begin{tikzpicture}[x=1.85cm,y=1.85cm]
  \tikzset{idx/.style={text=black,font=\scriptsize}}
  \def\splitpw#1#2#3{
    \draw[line width=.2mm] (#1.north) -- (#1.south);
    \node[font=\small] at ([xshift=-.20cm]#1.center) {$#2$};
    \node[font=\small] at ([xshift=.20cm]#1.center) {$#3$};
  }

  \draw[dashed,thick,rounded corners=0.3cm,fill=greenback,fill opacity=0.35]
    (-0.40,-0.40) rectangle (0.40,1.40);
  \draw[dashed,thick,rounded corners=0.3cm,fill=yellowback,fill opacity=0.35]
    (0.60,-0.40) rectangle (1.40,2.35);
  \draw[dashed,thick,rounded corners=0.3cm,fill=redback,fill opacity=0.35]
    (-0.40,1.62) -- (-0.40,2.60) -- (2.40,2.60) -- (2.40,0.69) --
    (1.62,0.69) -- (1.62,2.50) -- (0.40,2.50) -- (0.40,1.62) -- cycle;

  \node[shape=circle,draw=black,fill=white,minimum size=0.80cm] (n1) at (0,2) {};
  \node[shape=circle,draw=black,fill=white,minimum size=0.80cm] (n2) at (1,2) {};
  \node[shape=circle,draw=black,fill=white,minimum size=0.80cm] (n3) at (0,1) {};
  \node[shape=circle,draw=black,fill=white,minimum size=0.80cm] (n4) at (1,1) {};
  \node[shape=circle,draw=black,fill=white,minimum size=0.80cm] (n5) at (2,1) {};
  \node[shape=circle,draw=black,fill=white,minimum size=0.80cm] (n6) at (0,0) {};
  \node[shape=circle,draw=black,fill=white,minimum size=0.80cm] (n7) at (1,0) {};

  \foreach \v/\p/\w in {n1/2/1,n2/-2/1,n3/-2/1,n4/-1/1,n5/2/1,n6/4/3,n7/5/1}{\splitpw{\v}{\p}{\w}}

  \node[idx] at ($(n1)+(0.30,-0.13)$) {$1$};
  \node[idx] at ($(n2)+(0.30,-0.13)$) {$2$};
  \node[idx] at ($(n3)+(0.30,-0.13)$) {$3$};
  \node[idx] at ($(n4)+(0.30,-0.13)$) {$4$};
  \node[idx] at ($(n5)+(0.30,-0.13)$) {$5$};
  \node[idx] at ($(n6)+(0.30,-0.13)$) {$6$};
  \node[idx] at ($(n7)+(0.30,-0.13)$) {$7$};

  \draw[->,line width=.5mm,shorten <=1pt,shorten >=1pt] (n4) -- (n1);
  \draw[->,line width=.5mm,shorten <=1pt,shorten >=1pt] (n4) -- (n2);
  \draw[->,line width=.5mm,shorten <=1pt,shorten >=1pt] (n6) -- (n3);
  \draw[->,line width=.5mm,shorten <=1pt,shorten >=1pt] (n6) -- (n4);
  \draw[->,line width=.5mm,shorten <=1pt,shorten >=1pt] (n7) -- (n4);
  \draw[->,line width=.5mm,shorten <=1pt,shorten >=1pt] (n7) -- (n5);

  \node[font=\small] at (2.95,1.65) {$\mathcal I_1,\ \lambda_1=2$};
  \node[font=\small] at (1.85,0.0) {$\mathcal I_2,\ \lambda_2=\frac23$};
  \node[font=\small] at (-0.95,0.5) {$\mathcal I_3,\ \lambda_3=\frac12$};
\end{tikzpicture}

\Description{A directed graph on seven vertices, each drawn as a circle split into its profit and its weight, with arcs pointing from a vertex to a vertex that must be taken with it. Three dashed regions group the vertices into three disjoint layers, $\{1,5\}$, $\{2,4,7\}$ and $\{3,6\}$, of decreasing profit-to-weight ratio $2$, $2/3$ and $1/2$.}
\caption{The optimal solution of the parametric MCP on an example instance with $n=7$ vertices and $m=6$ precedence arcs. Each vertex shows $p_i\,|\,w_i$, with the number next to it the vertex index $i$. The dashed regions are the canonical closure-layer partition $\mathcal I_1,\mathcal I_2,\mathcal I_3$, in decreasing order of the ratio $\lambda_r=P(\mathcal I_r)/W(\mathcal I_r)$; the corresponding nested optimal solutions are $\mathcal M_1=\mathcal I_1=\{1,5\}$, $\mathcal M_2=\mathcal I_1\cup\mathcal I_2=\{1,2,4,5,7\}$, $\mathcal M_3=\mathcal I_1\cup\mathcal I_2\cup\mathcal I_3=V$.}
\label{fig:closure_layers_example}
\end{figure}

\subsection{Related work}\label{sec:related_work}
The MCP was introduced by Lerchs and Grossmann~\cite{LG65} as a model for open-pit mine design. Picard~\cite{Picard76} was the first to give the reduction of a general MCP instance to a minimum cut on an extended $s$--$t$ network, a connection already recognized by Johnson~\cite{Johnson68} in the bipartite case. The characterization of the parametric solution through profit-to-weight ratios, i.e., the content of \cref{prop:optratio}, goes back to Vallet~\cite{Vallet76}, and Hochbaum and Chen~\cite{HC00} call \textit{parameterization} the computation of the resulting nested family of closed sets.

 To compute the function $\varphi$ and its breakpoints, one can systematically search for its distinct affine components by solving problem \eqref{eq:UPIT_parametric} for different values of $\lambda$. Problem \eqref{eq:UPIT_parametric} for a fixed $\lambda$ value is a standard MCP, which can be solved in polynomial time since the constraint matrix of \eqref{eq:UPIT_parametric} is totally unimodular, and efficient specialized algorithms have been developed. The related literature is extensive, we refer the interested reader to, e.g., \cite{BAL70,BA98,CO19,HC01,HC00,LG65}. The strategy based on the multiple solution of standard MCPs is applied, for instance, in the implementation of the algorithms developed by \cite{CEGMR12} for the Open-Pit Mine Production Scheduling Problem.

Alternative approaches to solve the parametric MCP leverage algorithms designed for the general parametric max-flow/min-cut problem over arbitrary precedence graphs. Gallo, Grigoriadis and Tarjan~\cite{GGT89} compute all the breakpoints in $\mathcal{O}(nm\log(n^2/m))$ time and list the maximum-ratio closure problem among the applications of their method; the parametric variant of the algorithm of Hochbaum~\cite{HC01} and the parametric pseudoflow algorithm~\cite{HC08} both run in $\mathcal{O}(mn\log n)$ time, the latter being the method against which we benchmark in \cref{subsec:pseudoflow_comparison}. More recently, Beines et al.~\cite{Beines25} proposed an algorithm for the source--sink monotone parametric minimum cut problem which discovers the breakpoints in increasing order of the parameter, with an $\mathcal{O}(n^2m)$ worst-case bound and a better practical behavior than previous implementations.

A second line of work originates in single-machine scheduling and linear ordering. Horn~\cite{Horn72} computes a subtree of maximum profit-to-weight ratio on rooted trees, a problem solved in $\mathcal{O}(n\log n)$ time by Adolphson and Hu~\cite{AH73}. Sidney~\cite{Sidney75} introduced the decomposition of an instance obtained by iteratively extracting a closed subset of extremal ratio, which in our terminology is the optimal sequence of closure layers, up to splitting some of the layers into pieces of equal ratio, and Margot, Queyranne and Wang~\cite{MQW03} use parametric flow algorithms to compute the finest such decomposition. Lawler~\cite{Lawler78} proved that the underlying scheduling problem is NP-hard under general precedence constraints, and gave an $\mathcal{O}(n\log n)$ algorithm for series-parallel precedence digraphs, together with an $\mathcal{O}(m^2n\log n)$ procedure computing a single closed set of extremal ratio on an arbitrary digraph. A recent application of this framework comes from blockchain systems: Mollakhani, Wuille and Guo~\cite{MWG26} formulate the mempool linearization problem faced by Bitcoin miners as the computation of such a decomposition of a transaction dependency graph.

Closest to the present work are the algorithms exploiting a tree-shaped precedence graph. Galil~\cite{Galil80} uses mergeable heaps to compute, for every vertex of a rooted tree, the maximum-ratio subtree rooted at that vertex, in $\mathcal{O}(n\log n)$ time overall. Carlson and Eppstein~\cite{CarlsonEppstein06} address the parametric problem on rooted trees, solve it in $\mathcal{O}(n\log n)$ time and prove that this bound is optimal by a reduction from sorting. Eppstein~\cite{Eppstein2017} introduced the name \textit{parametric closure problem} and solved it for several classes of precedence graphs, in $\mathcal{O}(n\log n)$ time for series-parallel ones and in $\mathcal{O}(n^2\log n)$ time for \textit{polytrees}, i.e., directed trees, which represent the case considered in this paper
(see also Figure \ref{fig:forest_types}(a) for an example of polytree, called here mixed-tree).

Finally, the technique of \cref{sec:rake} borrows its vocabulary from parallel tree contraction~\cite{Miller1989ParallelTC,Abrahamson89}, where a tree is reduced by repeated \textit{rake} and \textit{compress} operations. The notion of a \textit{cluster} and of its \textit{boundary vertices}, on which our bottom-up phase is built, is the one used to maintain information on dynamic trees: it appears in Frederickson's topology trees~\cite{Frederickson97} and in the top trees of Alstrup et al.~\cite{Alstrup05}.

\subsection{Statement of contributions}
As previously mentioned, parametric flow algorithms
are the state-of-the-art for arbitrary precedence graphs, with a time complexity of 
$\mathcal{O}(mn \log n)$, where $m$ is the number of arcs and $n$ the number of vertices \citep{HC08}.
In this paper, we present new algorithms for computing the optimal sequence of closure layers when the precedence graph $G$ is a \textit{directed forest}, i.e., a disjoint union of directed trees.
On forests, where $m = \mathcal{O}(n)$, the time-complexity of  parametric flow algorithms is $\mathcal{O}(n^2 \log n)$, and the same bound is attained by the specialized algorithm of \cite{Eppstein2017} for this class of precedence graphs; the best known time-complexity drops to $\mathcal{O}(n\log n)$ only for the restricted classes of \textit{in-trees} and \textit{out-trees}, covered by the algorithms of \cite{CarlsonEppstein06} and \cite{Eppstein2017}. Actually, it has been proven in \cite{HC08} that over trees the standard MCP can be solved in $\mathcal{O}(n)$. Equivalently, the evaluation of function $\varphi$ for a {\em fixed} $\lambda$ value, can be done in $\mathcal{O}(n)$ time.  This fact, together with the $\mathcal{O}(n)$ bound on the number of breakpoints of the function $\varphi$, suggests that an $\mathcal{O}(n^2)$ algorithm for trees is possible, but, to the authors' knowledge, such an algorithm has not been presented in the literature yet.
\newline\newline\noindent
Our first contribution in Section \ref{sec:tree_algorithms}, is to provide two distinct algorithms based on vertex aggregation, which compute an optimal sequence of closure layers with time-complexity $\mathcal{O}(n^2)$ and space-complexity $\mathcal{O}(n)$ (which is optimal according to Observation \ref{obs:opttimespace}) for general forests. 
The algorithms are based on an iterative aggregation of vertices.
We also observe that a heap-based variant of these algorithms has a worse time-complexity but a better average behavior, with average computing times increasing as  $n \log n$, as shown in Section \ref{sec:computational_results}. We further show that when the forest is composed of \textit{out-trees} or \textit{in-trees}, close variants of the proposed algorithms reduce the complexity to $\mathcal{O}(n \log n)$. 
\newline\newline\noindent
Motivated by the practical behavior of the heap-based variant of the proposed algorithms, with computing times increasing as  $n \log n$, we further present, in Section \ref{sec:rake}, a new algorithm ({\tt RaC}), also based on a vertex aggregation strategy, with time- and space-complexity $\mathcal{O}(n\log n)$ also in the case of general forests. 
Note that, according to Observation \ref{obs:opttimespace}, this is an optimal time-complexity.
The development of this algorithm involved the use of generative AI, as detailed at the beginning of \cref{sec:rake}.
\newline\newline\noindent
We also performed extensive computational experiments which are discussed in Section \ref{sec:computational_results}. As previously mentioned, the heap-based variant of the algorithms presented in Section \ref{sec:tree_algorithms} works quite well in practice and, according to the computational experiments, performs better than {\tt RaC} over random instances and over some special structured trees (binary and path trees), but much worse over other special structured trees (star trees), which is fully in accordance with the different time-complexity of the algorithms.

\section{Aggregation-based algorithms for computing the optimal sequence of closure layers on directed forests}\label{sec:tree_algorithms}

\subsection{Notation and definitions}\label{sec:notation_trees}

\begin{definition}\label{def:succprec}
For each vertex $j\in V$, we denote with $N^+(j)$ the set of {\em successors} of $j$, while we denote with $N^-(j)$  the set of its {\em predecessors}.
Similarly, for a subset $M$ of vertices, we denote with $N^+(M)=\cup_{j\in M} N^+(j)\setminus M$ the set of successors of the subset $M$, and  with $N^-(M)=\cup_{j\in M} N^-(j)\setminus M$ the set of its predecessors. 
\end{definition}
\begin{definition}\label{def:forward_tree}
Let $G=(V,A)$ be a directed forest and let $j\in V$ be a vertex. The \textit{preceding set} of vertex $j$ is the set of vertices which can be reached from $j$ with a directed path, possibly with length 0, which is the set
\[
F_j=\{k\in V\colon \textnormal{there exists a directed path from $j$ to $k$}\}.
\]

Let $a=(i,j)\in A$ be an arc. The \textit{minimal preceding set} of arc $a=(i,j)$ is the set of vertices which are in the preceding set of $i$ but not in the preceding set of $j$, namely the set
\[
F_a:=F_i\setminus F_j.
\]
\end{definition}
\begin{definition}\label{def:backward_tree}
Let $G=(V,A)$ be a directed forest and let $j\in V$ be a vertex. The \textit{succeeding set} of vertex $j$ is the set of vertices which can reach $j$ with a directed path, possibly with length 0, which is the set
\[
B_j=\{k\in V\colon \textnormal{there exists a directed path from $k$ to $j$}\}.
\]

Let $a=(i,j)\in A$ be an arc. The \textit{minimal succeeding set} of arc $a=(i,j)$ is the set of vertices which are in the succeeding set of $j$ but not in the succeeding set of $i$, namely the set
\[
B_a:=B_j\setminus B_i.
\]
\end{definition}
\begin{definition}\label{def:final_initial_items}
We indicate with $V_{\mathrm{fin}}$ the set of the \textit{final vertices} of $G$, and with $V_{\mathrm{init}}$ the set of its \textit{initial vertices}, which are defined respectively by
\[
V_{\mathrm{fin}}:=\{j\in V \colon N^+(j)=\emptyset\},\qquad
V_{\mathrm{init}}:=\{j\in V \colon N^-(j)=\emptyset\}.
\]
\end{definition}
\begin{definition}\label{def:M_plus_minus}
Let $G=(V,A)$ be a forest and let $\mathcal M'\subset\mathcal M\subset V$ be a chain of subsets of vertices. The set of arcs connecting vertices in $\mathcal M$ is the set
\[
A(\mathcal M):=\{(i,j)\in A\colon i,j\in\mathcal M\}.
\]
The set of vertices in $\mathcal M$ which are connected with $\mathcal M'$ by an arc exiting $\mathcal M'$ is the set
\[
\mathcal M^+(\mathcal M'):=\{j\in\mathcal M\setminus\mathcal M'\colon \exists i\in\mathcal M',\, (i,j)\in A\},
\]
and the set of vertices in $\mathcal M$ which are connected with $\mathcal M'$ by an arc entering $\mathcal M'$ is the set
\[
\mathcal M^-(\mathcal M'):=\{i\in\mathcal M\setminus\mathcal M'\colon \exists j\in\mathcal M',\, (i,j)\in A\}.
\]
\end{definition}
An illustration of preceding/succeeding sets associated to vertices and arcs, and of final/initial vertices, can be found in \cref{fig:F_i_B_j_final_initial_vertices}(a), while an illustration of the sets $\mathcal M^+(\mathcal M')$ and $\mathcal M^-(\mathcal M')$ is in \cref{fig:F_i_B_j_final_initial_vertices}(b).

\subsection{Main algorithm to compute the optimal sequence of closure layers on a forest}\label{sec:main_alg}
First, we need to present a simple ratio property for disjoint sets, used later to compare unions of candidate closure layers. The proof is reported for completeness.
\begin{proposition}
\label{obs:disjoint}
Let $M_1,M_2\subset V$ be nonempty. If $M_1$ and $M_2$ are disjoint, i.e., $M_1\cap M_2=\emptyset$, then it cannot hold that
\begin{equation}
\label{eq:nothold}
\frac{P(M_1\cup M_2)}{W(M_1\cup M_2)}>\frac{P(M_1)}{W(M_1)},\frac{P(M_2)}{W(M_2)}
\quad\text{or}\quad
\frac{P(M_1\cup M_2)}{W(M_1\cup M_2)}<\frac{P(M_1)}{W(M_1)},\frac{P(M_2)}{W(M_2)}.
\end{equation}
Moreover,
\begin{equation}
\label{eq:equal}
\frac{P(M_1\cup M_2)}{W(M_1\cup M_2)}\geq (\le) \frac{P(M_1)}{W(M_1)},\frac{P(M_2)}{W(M_2)} \ \Rightarrow\ \frac{P(M_1\cup M_2)}{W(M_1\cup M_2)}=\frac{P(M_1)}{W(M_1)}=\frac{P(M_2)}{W(M_2)}.
\end{equation}
\end{proposition}
\begin{proof}
Let us assume by contradiction that the first inequality in (\ref{eq:nothold}) holds (the proof for the second inequality is analogous). Then, since $M_1\cap M_2=\emptyset$:
\[
\frac{P(M_1\cup M_2)}{W(M_1\cup M_2)}=\frac{P(M_1)+P(M_2)}{W(M_1)+W(M_2)}>\frac{P(M_1)}{W(M_1)},\frac{P(M_2)}{W(M_2)}.
\]
But
\[
\frac{P(M_1)+P(M_2)}{W(M_1)+W(M_2)}>\frac{P(M_1)}{W(M_1)}, \frac{P(M_2)}{W(M_2)}
\quad\Rightarrow\quad
\frac{P(M_2)}{W(M_2)}>\frac{P(M_1)}{W(M_1)}\ \text{and}\ \frac{P(M_2)}{W(M_2)}<\frac{P(M_1)}{W(M_1)},
\]
which is not possible. In a completely similar way we can also prove (\ref{eq:equal}).
\end{proof}

\begin{figure}[htbp]
\centering
\begin{tikzpicture}[x=0.82cm,y=1cm,scale=0.5,transform shape]

\begin{scope}[shift={(0,0)}]
  \node[shape=circle,draw=black,minimum size=1.0cm,pattern=north east lines,pattern color=black] (n1)  at (-6,  3) {};
  \node[shape=circle,draw=black,minimum size=1.0cm,pattern=north east lines,pattern color=black] (n2)  at (-2,  3) {};
  \node[shape=circle,draw=orange,line width=.5mm,minimum size=1.0cm,pattern=north east lines,pattern color=orange] (n3)  at ( 2,  3) {};
  \node[shape=circle,draw=orange,line width=.5mm,minimum size=1.0cm,pattern=north east lines,pattern color=orange] (n4)  at ( 6,  3) {};

  \node[shape=circle,draw=black,minimum size=1.0cm,pattern=north east lines,pattern color=black] (n5)  at (-6,  0) {};
  \node[shape=circle,draw=darkgreen,line width=.5mm,minimum size=1.0cm]  (n6)  at (-2,  0) {\Large $j$};
  \node[shape=circle,draw=orange,line width=.5mm,minimum size=1.0cm]     (n7)  at ( 2,  0) {};
  \node[shape=circle,draw=orange,line width=.5mm,minimum size=1.0cm,pattern=north east lines,pattern color=orange] (n8)  at ( 6,  0) {};

  \node[shape=circle,draw=darkgreen,line width=.5mm,minimum size=1.0cm]  (n9)  at (-6, -3) {};
  \node[shape=circle,draw=darkgreen,line width=.5mm,minimum size=1.0cm,pattern=dots,pattern color=darkgreen] (n10) at (-2, -3) {};
  \node[shape=circle,draw=orange,line width=.5mm,minimum size=1.0cm]     (n11) at ( 2, -3) {\Large $i$};
  \node[shape=circle,draw=black,minimum size=1.0cm,pattern=dots,pattern color=black] (n12) at ( 6, -3) {};

  \node[shape=circle,draw=darkgreen,line width=.5mm,minimum size=1.0cm,pattern=dots,pattern color=darkgreen] (n13) at (-6, -6) {};
  \node[shape=circle,draw=darkgreen,line width=.5mm,minimum size=1.0cm,pattern=dots,pattern color=darkgreen] (n14) at (-2, -6) {};
  \node[shape=circle,draw=black,minimum size=1.0cm,pattern=dots,pattern color=black] (n15) at ( 2, -6) {};
  \node[shape=circle,draw=black,minimum size=1.0cm,pattern=dots,pattern color=black] (n16) at ( 6, -6) {};

  \draw[->,line width=.5mm,shorten <=1pt,shorten >=1pt] (n9)  to (n5);
  \draw[->,line width=.5mm,shorten <=1pt,shorten >=1pt] (n13) to (n9);
  \draw[->,line width=.5mm,shorten <=1pt,shorten >=1pt] (n14) to (n9);
  \draw[->,line width=.5mm,shorten <=1pt,shorten >=1pt] (n9)  to (n6);

  \draw[->,line width=.5mm,shorten <=1pt,shorten >=1pt] (n10) to (n6);
  \draw[->,line width=.5mm,shorten <=1pt,shorten >=1pt] (n6)  to (n2);
  \draw[->,line width=.5mm,shorten <=1pt,shorten >=1pt] (n6)  to (n1);

  \draw[->,line width=.5mm,shorten <=1pt,shorten >=1pt] (n11) to (n7);
  \draw[->,line width=.5mm,shorten <=1pt,shorten >=1pt] (n7)  to (n3);
  \draw[->,line width=.5mm,shorten <=1pt,shorten >=1pt] (n7)  to (n4);
  \draw[->,line width=.5mm,shorten <=1pt,shorten >=1pt] (n11) to (n8);
  \draw[->,line width=.5mm,shorten <=1pt,shorten >=1pt] (n16) to (n11);
  \draw[->,line width=.5mm,shorten <=1pt,shorten >=1pt] (n15) to (n11);
  \draw[->,line width=.5mm,shorten <=1pt,shorten >=1pt] (n12) to (n8);

  \draw[->,line width=1.2mm,shorten <=1pt,shorten >=1pt] (n11) to node[midway,above,sloped] {\huge $a$} (n6);

  \draw[->,line width=.5mm,shorten <=1pt,shorten >=1pt,darkgreen] (n13) to (n9);
  \draw[->,line width=.5mm,shorten <=1pt,shorten >=1pt,darkgreen] (n14) to (n9);
  \draw[->,line width=.5mm,shorten <=1pt,shorten >=1pt,darkgreen] (n9)  to (n6);
  \draw[->,line width=.5mm,shorten <=1pt,shorten >=1pt,darkgreen] (n10) to (n6);

  \draw[->,line width=.5mm,shorten <=1pt,shorten >=1pt,orange] (n11) to (n7);
  \draw[->,line width=.5mm,shorten <=1pt,shorten >=1pt,orange] (n7)  to (n3);
  \draw[->,line width=.5mm,shorten <=1pt,shorten >=1pt,orange] (n7)  to (n4);
  \draw[->,line width=.5mm,shorten <=1pt,shorten >=1pt,orange] (n11) to (n8);

  \draw[dashed,thick,rounded corners=10pt]
    (-6.95, 3.95) -- ( 6.95, 3.95) -- ( 6.95,-0.95) -- ( 2.95,-3.85) --
    ( 1.05,-3.85) -- (-6.95, 1.85) -- cycle;
  \node[font=\huge] at (6.55,4.35) {$F_i$};

  \draw[dashed,thick,rounded corners=10pt]
    (-6.95,-1.95) -- (-2.95, 0.95) -- (-0.95, 0.95) --
    ( 6.95,-4.85) -- ( 6.95,-6.95) -- (-6.95,-6.95) -- cycle;
  \node[font=\huge] at (-6.30,-7.45) {$B_j$};

  \node[orange,font=\huge] at (3.9, 0.3) {$F_a$};
  \node[darkgreen,font=\huge] at (-3.8, -3.0) {$B_a$};

  \node[font=\huge] at (0,-8.3) {(a)};
\end{scope}

\begin{scope}[shift={(15.5,0)}]
  \node[shape=circle,draw=teal,line width=.5mm,minimum size=1.0cm,pattern=north west lines,pattern color=teal] (n1)  at (-6,  3) {};
  \node[shape=circle,draw=teal,line width=.5mm,minimum size=1.0cm,pattern=north west lines,pattern color=teal] (n2)  at (-2,  3) {};

  \node[shape=circle,draw=black,minimum size=1.0cm]                   (n3)  at ( 2,  3) {};
  \node[shape=circle,draw=black,minimum size=1.0cm]                   (n4)  at ( 6,  3) {};

  \node[shape=circle,draw=purple,line width=.5mm,minimum size=1.0cm]  (n5)  at (-6,  0) {};
  \node[shape=circle,draw=purple,line width=.5mm,minimum size=1.0cm]  (n6)  at (-2,  0) {};
  \node[shape=circle,draw=purple,line width=.5mm,minimum size=1.0cm]  (n7)  at ( 2,  0) {};
  \node[shape=circle,draw=black,minimum size=1.0cm]                   (n8)  at ( 6,  0) {};

  \node[shape=circle,draw=black,minimum size=1.0cm]                   (n9)  at (-6, -3) {};
  \node[shape=circle,draw=purple,line width=.5mm,minimum size=1.0cm]  (n10) at (-2, -3) {};
  \node[shape=circle,draw=purple,line width=.5mm,minimum size=1.0cm]  (n11) at ( 2, -3) {};
  \node[shape=circle,draw=black,minimum size=1.0cm]                   (n12) at ( 6, -3) {};

  \node[shape=circle,draw=black,minimum size=1.0cm]                   (n13) at (-6, -6) {};
  \node[shape=circle,draw=black,minimum size=1.0cm]                   (n14) at (-2, -6) {};
  \node[shape=circle,draw=brown,line width=.5mm,minimum size=1.0cm,pattern=crosshatch,pattern color=brown] (n15) at ( 2, -6) {};
  \node[shape=circle,draw=brown,line width=.5mm,minimum size=1.0cm,pattern=crosshatch,pattern color=brown] (n16) at ( 6, -6) {};

  \draw[->,line width=.5mm,shorten <=1pt,shorten >=1pt] (n9)  to (n5);
  \draw[->,line width=.5mm,shorten <=1pt,shorten >=1pt] (n13) to (n9);
  \draw[->,line width=.5mm,shorten <=1pt,shorten >=1pt] (n14) to (n9);
  \draw[->,line width=.5mm,shorten <=1pt,shorten >=1pt] (n9)  to (n6);

  \draw[->,line width=.5mm,shorten <=1pt,shorten >=1pt] (n10) to (n6);
  \draw[->,line width=.5mm,shorten <=1pt,shorten >=1pt] (n6)  to (n2);
  \draw[->,line width=.5mm,shorten <=1pt,shorten >=1pt] (n6)  to (n1);

  \draw[->,line width=.5mm,shorten <=1pt,shorten >=1pt] (n11) to (n7);
  \draw[->,line width=.5mm,shorten <=1pt,shorten >=1pt] (n7)  to (n3);
  \draw[->,line width=.5mm,shorten <=1pt,shorten >=1pt] (n7)  to (n4);
  \draw[->,line width=.5mm,shorten <=1pt,shorten >=1pt] (n11) to (n8);
  \draw[->,line width=.5mm,shorten <=1pt,shorten >=1pt] (n16) to (n11);
  \draw[->,line width=.5mm,shorten <=1pt,shorten >=1pt] (n15) to (n11);
  \draw[->,line width=.5mm,shorten <=1pt,shorten >=1pt] (n12) to (n8);

  \draw[->,line width=.5mm,shorten <=1pt,shorten >=1pt] (n11) to (n6);

  \draw[dashed,thick,rounded corners=10pt]
    (-7.5, 3.95) -- (-0.95, 3.95) -- (6.95, -1.95) -- (6.95, -6.95) --
    (1.05, -6.95) -- (-7.5, -0.95) -- cycle;
  \node[font=\huge] at (6.60,-7.45) {$\mathcal M$};

  \draw[dashed,thick,rounded corners=10pt,purple]
    (-6.95, 0.95) -- (2.95, 0.95) -- (2.95, -3.95) -- (-2.95, -3.95) --
    (-2.95, -0.95) -- (-6.95, -0.95) -- cycle;

  \draw[->,line width=.5mm,shorten <=1pt,shorten >=1pt,purple] (n10) to (n6);
  \draw[->,line width=.5mm,shorten <=1pt,shorten >=1pt,purple] (n11) to (n7);
  \draw[->,line width=.5mm,shorten <=1pt,shorten >=1pt,purple] (n11) to (n6);
  \node[purple,font=\huge] at (0.2,-2.25) {$\mathcal M'$};

  \draw[->,line width=.5mm,shorten <=1pt,shorten >=1pt,teal] (n6) to (n1);
  \draw[->,line width=.5mm,shorten <=1pt,shorten >=1pt,teal] (n6) to (n2);

  \draw[->,line width=.5mm,shorten <=1pt,shorten >=1pt,brown] (n15) to (n11);
  \draw[->,line width=.5mm,shorten <=1pt,shorten >=1pt,brown] (n16) to (n11);

  \node[font=\huge] at (0,-8.3) {(b)};
\end{scope}

\end{tikzpicture}

\Description{Two drawings of directed trees on sixteen vertices. In the left one a bold arc $a=(i,j)$ is singled out, two dashed contours enclose the preceding set of $i$ and the succeeding set of $j$, colours mark the vertices of the minimal preceding set $F_a$ and of the minimal succeeding set $B_a$, and two hatchings mark the vertices with no successor and those with no predecessor. In the right one a dashed contour encloses a set $\mathcal M$ and a second one a subset $\mathcal M'$, with colours marking the vertices of $\mathcal M$ reached by an arc leaving $\mathcal M'$ and those from which an arc enters $\mathcal M'$.}
\caption{(a) Representation of a directed tree. The arc $a=(i,j)$ is highlighted in bold. Two dashed lines enclose the sets $F_i$ and $B_j$. The vertices in $F_a$, and their connecting arcs, are highlighted in orange, while the same is done in green for $B_a$. Final vertices (i.e., vertices $v$ with $N^+(v)=\emptyset$) are filled with a north-east line pattern, while initial vertices (i.e., vertices $v$ with $N^-(v)=\emptyset$) are filled with a dot pattern. (b) Directed graph with dashed set $\mathcal M$ and highlighted subset $\mathcal M'$ in purple. Vertices in $\mathcal M^+(\mathcal M')$ are highlighted in teal with a north-west line pattern, and the arcs from $\mathcal M'$ to $\mathcal M^+(\mathcal M')$ are highlighted in teal. Vertices in $\mathcal M^-(\mathcal M')$ are highlighted in brown with a crosshatch pattern, and the arcs from $\mathcal M^-(\mathcal M')$ to $\mathcal M'$ are highlighted in brown.}
\label{fig:F_i_B_j_final_initial_vertices}
\end{figure}
We also make a couple of remarks.
\begin{remark}\label{rem:faoptseq}
Let $\mathcal T=\big(\mathcal I_1,\mathcal I_2,\ldots,\mathcal I_{k(\mathcal T)}\big)$ be the optimal sequence of closure layers. Then, for every vertex $j\in V$, all vertices in $F_j$ are contained in closure layers of the sequence $\mathcal T$ that precede the closure layer containing $j$, or are contained in the same closure layer. Furthermore, given an arc $a=(i,j)$, vertex $i$ can be in the same closure layer as vertex $j$ only if all vertices in $F_a$ are contained in a closure layer of the sequence $\mathcal T$ that precedes the closure layer containing $j$, or are contained in the same closure layer.
\end{remark}

\begin{remark}\label{rem:wing_without_root_indep}
Notice that $N^+(F_j)=\{k\in V \setminus F_j\colon \exists i\in F_j,\, (i,k)\in A\}=\emptyset$ for every $j\in V$ by \cref{def:forward_tree}. Furthermore $N^+(F_{(i,j)}\setminus\{i\})=\emptyset$ for every $(i,j)\in A$ since $F_{(i,j)}\setminus\{i\}$ is the disjoint union of the sets $F_k$ with $k\in N^+(i)\setminus \{j\}$. 
\end{remark}

To show that our main algorithm outputs the optimal sequence of closure layers we need to prove three propositions.
\begin{proposition}\label{prop:disjoint_union_wings}
Suppose $G=(V,A)$ is a forest and let $\mathcal M\subset V$ be a subset of vertices. If $N^+(\mathcal M)=\emptyset$, then
\[
\mathcal M=\left(\bigsqcup_{a\in \tilde A} F_a\right)\sqcup \tilde V_{\mathrm{fin}}
\]
for some subset $\tilde A\subset A$ and $\tilde V_{\mathrm{fin}}\subset V_{\mathrm{fin}}$, i.e., $\mathcal M$ is the disjoint union of the minimal preceding sets associated to some arcs and some subset of final vertices.
\end{proposition}
\begin{proof}
We can show this constructively.  We start with $\tilde A=\emptyset$ and $\tilde V_{\mathrm{fin}}=\emptyset$.
Then, we add one final vertex in $V_{\mathrm{fin}}\cap \mathcal M$ to $\tilde V_{\mathrm{fin}}$ for every
connected component of the graph $(\mathcal M,A(\mathcal M))$.
Such final vertices exist since $N^+(\mathcal M)=\emptyset$ and $G$ is a forest.
Thus we obtain a first set $\mathcal M'\subset \mathcal M$ defined by
\[
\mathcal M'=\tilde V_{\mathrm{fin}}
\]
and we have $N^+(\mathcal M')=\emptyset$.
Now if $\mathcal M^-(\mathcal M')\neq \emptyset$ (i.e., $\mathcal M\neq\mathcal M'$), we consider any item $j\in\mathcal M'$ which is connected to some item in $\mathcal M^-(\mathcal M')$. Then we take any maximal backwards path $\pi_j\subset\mathcal A(\mathcal M)\setminus\mathcal A(\mathcal M')$ from $j$ to some item $v(j)$ with $\mathcal M^-(\{v(j)\})=\emptyset$, and we add the arcs in $\pi_j$ to $\tilde{\mathcal A}$, updating $\mathcal M'$ accordingly.
Since the new elements added to $\tilde{\mathcal A}$ are chosen in a backwards path in $\mathcal A(\mathcal M)\setminus\mathcal A(\mathcal M')$, we have $F_a\cap F_{a'}=\emptyset$, for every $a,a'\in\tilde{\mathcal A}$ with $a\ne a'$. Thus we obtain
\[
\mathcal M'=\left(\bigsqcup_{a\in \tilde A} F_a\right)\sqcup \tilde V_{\mathrm{fin}}.
\]
and we have enlarged  the set $\mathcal M'$ to also include $F_{v(j)}$. Hence we have again $\mathcal N^+(\mathcal M')=\emptyset$.

Then, we repeat the last step until $\mathcal M'=\mathcal M$.
\end{proof}

An illustration of the outcome of the constructive procedure described in the proof of \cref{prop:disjoint_union_wings} can be found in \cref{fig:disjoint_union}.

\begin{figure}[htbp]
\centering
\begin{tikzpicture}[x=0.82cm,y=1cm,scale=0.5,transform shape]

  \node[shape=circle,draw=teal,line width=.5mm,minimum size=1.0cm]                      (n1)  at (-6,  3) {\Large 1};
  \node[shape=circle,draw=orange,line width=.5mm,minimum size=1.0cm]                     (n2)  at (-2,  3) {\Large 2};
  \node[shape=circle,draw=purple,line width=.5mm,minimum size=1.0cm]                     (n3)  at ( 2,  3) {\Large 3};
  \node[shape=circle,draw=blue,line width=.5mm,minimum size=1.0cm]                       (n4)  at ( 6,  3) {\Large 4};

  \node[shape=circle,draw=green!50!black,line width=.5mm,minimum size=1.0cm]             (n5)  at (-6,  0) {\Large 5};
  \node[shape=circle,draw=orange,line width=.5mm,minimum size=1.0cm]                     (n6)  at (-2,  0) {\Large 6};
  \node[shape=circle,draw=purple,line width=.5mm,minimum size=1.0cm]                     (n7)  at ( 2,  0) {\Large 7};
  \node[shape=circle,draw=magenta,line width=.5mm,minimum size=1.0cm]                    (n8)  at ( 6,  0) {\Large 8};

  \node[shape=circle,draw=green!50!black,line width=.5mm,minimum size=1.0cm]             (n9)  at (-6, -3) {\Large 9};
  \node[shape=circle,draw=blue!60!black,line width=.5mm,minimum size=1.0cm]               (n10) at (-2, -3) {\Large 10};
  \node[shape=circle,draw=black,minimum size=1.0cm]                                       (n11) at ( 2, -3) {\Large 11};
  \node[shape=circle,draw=brown,line width=.5mm,minimum size=1.0cm]                      (n12) at ( 6, -3) {\Large 12};

  \node[shape=circle,draw=black,minimum size=1.0cm]                                       (n13) at (-6, -6) {\Large 13};
  \node[shape=circle,draw=black,minimum size=1.0cm]                                       (n14) at (-2, -6) {\Large 14};
  \node[shape=circle,draw=black,minimum size=1.0cm]                                       (n15) at ( 2, -6) {\Large 15};
  \node[shape=circle,draw=black,minimum size=1.0cm]                                       (n16) at ( 6, -6) {\Large 16};

  \draw[->,line width=.5mm,shorten <=1pt,shorten >=1pt] (n13) to (n9);
  \draw[->,line width=.5mm,shorten <=1pt,shorten >=1pt] (n14) to (n9);

  \draw[->,line width=.5mm,shorten <=1pt,shorten >=1pt] (n11) to (n7);
  \draw[->,line width=.5mm,shorten <=1pt,shorten >=1pt] (n11) to (n8);
  \draw[->,line width=.5mm,shorten <=1pt,shorten >=1pt] (n16) to (n11);
  \draw[->,line width=.5mm,shorten <=1pt,shorten >=1pt] (n15) to (n11);
  \draw[->,line width=.5mm,shorten <=1pt,shorten >=1pt] (n11) to (n6);

  \draw[->,line width=.6mm,shorten <=1pt,shorten >=1pt,dashed,orange]            (n6)  to (n1);
  \draw[->,line width=.6mm,shorten <=1pt,shorten >=1pt,dashed,purple]            (n7)  to (n4);
  \draw[->,line width=.6mm,shorten <=1pt,shorten >=1pt,dashed,brown]             (n12) to (n8);
  \draw[->,line width=.6mm,shorten <=1pt,shorten >=1pt,dashed,green!50!black]    (n9)  to (n6);
  \draw[->,line width=.6mm,shorten <=1pt,shorten >=1pt,dashed,blue!60!black]      (n10) to (n6);

  \draw[->,line width=.5mm,shorten <=1pt,shorten >=1pt,orange]         (n6) to (n2);
  \draw[->,line width=.5mm,shorten <=1pt,shorten >=1pt,purple]         (n7) to (n3);
  \draw[->,line width=.5mm,shorten <=1pt,shorten >=1pt,green!50!black] (n9) to (n5);

  \draw[dashed,thick,rounded corners=12pt]
    (-7.4, 4.1) -- ( 7.4, 4.1) -- ( 7.4,-4.10) --
    ( 5.05,-4.10) -- ( 2.95,-1.70) --
    ( 1.00,-1.70) --
    (-0.95,-4.10) -- (-7.4,-4.10) -- cycle;
  \node[font=\huge] at (6.35,4.40) {$\mathcal M$};

  \node[teal,font=\huge]             at (-6.20, 1.95) {$\{1\}$};
  \node[blue,font=\huge]             at ( 6.35, 1.95) {$\{4\}$};
  \node[magenta,font=\huge]          at ( 6.35, 0.95) {$\{8\}$};

  \node[orange,font=\huge]           at (-1.10, 1.45) {$F_{(6,1)}$};
  \node[purple,font=\huge]           at ( 2.80, 1.65) {$F_{(7,4)}$};
  \node[brown,font=\huge]            at ( 5.10,-2.05) {$F_{(12,8)}$};
  \node[green!50!black,font=\huge]   at (-8.7,-1.5) {$F_{(9,6)}$};
  \node[blue!60!black,font=\huge,anchor=east]     at (-2.75,-2.3) {$F_{(10,6)}$};

\end{tikzpicture}

\Description{A directed tree on sixteen vertices with a dashed contour enclosing a set $\mathcal M$. Colours split $\mathcal M$ into three isolated vertices with no successor and five further groups, each of which is the minimal preceding set of one arc, drawn dashed, that joins the group to the rest of $\mathcal M$.}
\caption{Constructive decomposition of a set $\mathcal M$ with $N^+(\mathcal M)=\emptyset$ according to the proof of Proposition~\ref{prop:disjoint_union_wings}.
We start from one final vertex per connected component in $\mathcal M$, i.e., $\mathcal M'=\tilde V_{\mathrm{fin}}=\{1,4,8\}$, and then iteratively add sets $F_a$ associated with arcs $a$ from $\mathcal M^-(\mathcal M')$ to $\mathcal M'$ (shown as dashed colored arcs), until $\mathcal M'=\mathcal M$. The resulting disjoint union is
$\mathcal M=\{1,4,8\}\sqcup F_{(6,1)}\sqcup F_{(7,4)}\sqcup F_{(12,8)}\sqcup F_{(9,6)}\sqcup F_{(10,6)}$.}
\label{fig:disjoint_union}
\end{figure}

\begin{proposition}\label{prop:best_wing_tail_contracts}
Suppose $G=(V,A)$ is a forest and let $a^*\in\arg\max_{a\in A}\frac{P(F_a)}{W(F_a)}$ with $a^*=(u,v)$. If $\frac{P(F_{a^*})}{W(F_{a^*})}\ge\frac{p_f}{w_f}$ for every $f\in V_{\mathrm{fin}}$, then the two vertices $u$ and $v$ belong to the same closure layer in the optimal sequence.
\end{proposition}
\begin{proof}
Let $\mathcal T=\big(\mathcal I_1,\mathcal I_2,\ldots,\mathcal I_{k(\mathcal T)}\big)$ be the optimal sequence of closure layers and let $t(v)\in\{1,2,\ldots,k(\mathcal T)\}$ be the index of the closure layer containing $v$, i.e., $v\in\mathcal I_{t(v)}$. Consider the set of vertices
\[
\mathcal M:=\bigcup_{r=1}^{t(v)}\mathcal I_r.
\]
Notice that,  in view of Remark \ref{rem:faoptseq}, $N^+(\mathcal M)=\emptyset$. 
 By \cref{prop:disjoint_union_wings} we have 
\begin{equation}\label{eq:M_decomposition_proof}
\mathcal M= \left(\bigsqcup_{a\in\tilde A}F_a\right)\sqcup \tilde V_{\mathrm{fin}}
\end{equation}
for some $\tilde A\subset A$ and $\tilde V_{\mathrm{fin}}\subset V_{\mathrm{fin}}$. 

Suppose by contradiction that $u\notin \mathcal I_{t(v)}$ and thus, in view of Remark \ref{rem:faoptseq}, $u\notin\mathcal M$ since $(u,v)\in A$. Hence, by \cref{rem:wing_without_root_indep} we have $N^+(F_{a^*}\cap\mathcal M)=\emptyset$. Thus, by \cref{prop:disjoint_union_wings} we have 
\begin{equation}\label{eq:Fastar_int_M_decomp_proof}
F_{a^*}\cap\mathcal M=\left(\bigsqcup_{a\in A^*}F_a\right)\sqcup V^*_{\mathrm{fin}}
\end{equation}
for some $A^*\subset A$ and $V^*_{\mathrm{fin}}\subset V_{\mathrm{fin}}$.

We have
\[
\frac{P(F_{a^*}\setminus\mathcal M)}{W(F_{a^*}\setminus\mathcal M)}\ge\frac{P(F_{a^*})}{W(F_{a^*})}\ge\frac{P(\mathcal M)}{W(\mathcal M)}\ge\frac{P(\mathcal I_{t(v)})}{W(\mathcal I_{t(v)})}
\]
where in the last inequality we used the fact that $\mathcal T=\big(\mathcal I_1,\mathcal I_2,\ldots,\mathcal I_{k(\mathcal T)}\big)$ is a decreasing sequence of closure layers.
The second inequality follows from: (i) the fact that $\frac{P(F_{a^*})}{W(F_{a^*})}\ge\frac{P(F_{a})}{W(F_a)}$ for every $a\in A$ and $\frac{P(F_{a^*})}{W(F_{a^*})}\ge\frac{p_f}{w_f}$ for every $f\in V_{\mathrm{fin}}$; (ii) \cref{eq:M_decomposition_proof}; (iii) \cref{obs:disjoint}.
The first inequality follows again from the same properties of $F_{a^*}$, from \cref{eq:Fastar_int_M_decomp_proof}, and from \cref{obs:disjoint}.
This chain of inequalities leads to a contradiction. 
Indeed, $\frac{P(F_{a^*}\setminus \mathcal M)}{W(F_{a^*}\setminus \mathcal M)} \geq \frac{P(\mathcal I_{t(v)})}{W(\mathcal I_{t(v)})}$ and since $F_{a^*}\setminus \mathcal M$ and $\mathcal I_{t(v)}$ are disjoint, the ratio of their union is not lower than the ratio of $\mathcal I_{t(v)}$, by \cref{obs:disjoint}.
Observing that $N^+(F_{a^*}\cup\mathcal M)=\emptyset$, we also have that $\mathcal I_{t(v)}\cup (F_{a^*}\setminus \mathcal M)$ is a feasible closure layer after $\mathcal I_1,\mathcal I_2,\ldots,\mathcal I_{t(v)-1}$.
This contradicts the fact that $\mathcal I_{t(v)}$ is the maximal optimal solution of problem (\ref{eq:UPIT_parametric}) when $\lambda=\lambda_{t(v)}$ and the set of vertices is restricted to $V \setminus \left[\cup_{j=1}^{t(v)-1} \mathcal I_j\right]$ (see Proposition \ref{prop:optratio}). 
\end{proof}
\cref{prop:best_wing_tail_contracts} allows us, by collapsing the arc $a^*$, to reduce the problem of finding the optimal sequence of closure layers to a smaller graph.

\begin{proposition}\label{prop:best_final_nodes_are_closure_layer}
Suppose $G=(V,A)$ is a forest and let $g\in\mathcal B_{\mathrm{fin}}:=\arg \max_{f\in V_{\mathrm{fin}}}\frac{p_f}{w_f}$. 
If $\frac{p_{g}}{w_{g}}>\frac{P(F_a)}{W(F_a)}$ for every $a\in A$, then $\mathcal B_{\mathrm{fin}}$ is the first closure layer in the optimal sequence of closure layers.
\end{proposition}
\begin{proof}
In view of Remark \ref{rem:faoptseq}, any candidate subset $\mathcal M\subset V$ to be the first closure layer in the optimal sequence must satisfy $N^+(\mathcal M)=\emptyset$. By \cref{prop:disjoint_union_wings} we have
\[\mathcal M=\left(\bigsqcup_{a\in\tilde A}F_a\right)\sqcup \tilde V_{\mathrm{fin}}\]
for some subsets $\tilde A\subset A$, $\tilde V_{\mathrm{fin}}\subset V_{\mathrm{fin}}$. Since $g\in\mathcal B_{\mathrm{fin}}=\arg\max_{f\in V_{\mathrm{fin}}}\frac{p_f}{w_f}$ and  $\frac{p_g}{w_g}>\frac{P(F_a)}{W(F_a)}$ for every $a\in A$, by \cref{obs:disjoint} the ratio associated to the set $\mathcal B_{\mathrm{fin}}$ is larger than or equal to the ratio associated to $\mathcal M$, with the equality possible only if $\tilde A=\emptyset$.
This implies that $\mathcal B_{\mathrm{fin}}$ is the maximal subset of $V$ among those with largest ratio, i.e., $\mathcal B_{\mathrm{fin}}=\mathcal I_1$.
\end{proof}

Now we are ready to describe our first algorithm, based on aggregation of vertices.
In \cref{alg:algorithm}, called \textit{Peel-and-Contract} ({\tt PaC}) in what follows, the initialization (lines \ref{ln:start_init}-\ref{ln:end_init}) creates a copy $(\bar V,\bar A)$ of the graph and associates to each copied vertex $j$ a set of vertices $R_j\subset V$ in the original graph; then, in the main cycle (lines \ref{ln:start_main_cycle}-\ref{ln:end_main_cycle}) at each iteration we either find a new closure layer in the optimal sequence (line \ref{ln:end_if_best_final}) as prescribed by \cref{prop:best_final_nodes_are_closure_layer}, and we remove it from the graph $(\bar V,\bar A)$, or we ``contract'' an arc joining two adjacent vertices according to \cref{prop:best_wing_tail_contracts} (line \ref{ln:end_if_best_wing}). Throughout the description of {\tt PaC} and of its variants, the sets $N^+(\cdot)$, $N^-(\cdot)$, $F_j$, $F_a$ and the totals $P(\cdot)$, $W(\cdot)$ are understood with respect to the current working graph $(\bar V,\bar A)$ and to the current aggregated profits $\bar p_j$ and weights $\bar w_j$; we write $\bar N^+(\cdot)$ and $\bar N^-(\cdot)$ for the successors and predecessors in $\bar A$ when the distinction matters.

In particular, in procedure \textsc{FindBestWing} (\cref{alg:find_best_wing}) we determine the arc $a^*=(u,v)\in\bar A$ whose associated set $F_{a^*}$ is the one with the largest ratio.
Then, in procedure \textsc{FindBestFinalVertices} (\cref{alg:find_best_final_vertices}) we determine the set $\mathcal B_{\mathrm{fin}}$ of the final vertices of $\bar V$ with the largest ratio.
Now, if the largest determined ratio is achieved only by the final vertices in $\mathcal B_{\mathrm{fin}}$, then the subset of vertices $\bigcup_{f\in\mathcal B_{\mathrm{fin}}}R_f\subset V$ is the next closure layer in the optimal sequence by \cref{prop:best_final_nodes_are_closure_layer}. Hence, procedure \textsc{RemoveFinalVertices} (\cref{alg:remove_final_vertices}) adds to the optimal sequence the new closure layer, and removes the vertices in $\mathcal B_{\mathrm{fin}}$ together with the corresponding arcs from the graph $(\bar V,\bar A)$.
Conversely, if the largest determined ratio is achieved by (or also by) a set $F_{a^*}\subset\bar V$ associated to some arc $a^*\in\bar A$, procedure \textsc{ContractArc} (\cref{alg:contract_arc}) removes arc $a^*=(u,v)$ from $\bar A$ according to \cref{prop:best_wing_tail_contracts}, joining the two vertices $u$ and $v$ and the corresponding sets $R_u$, $R_v$ of associated vertices in $V$. An illustration of a full run of {\tt PaC} on a small instance is reported in \cref{app:ec_pac_execution}.

\begin{algorithm}[htbp]
\caption{Algorithm {\tt PaC} to compute the optimal sequence of closure layers on a forest}\label{alg:algorithm}
\small
\begin{algorithmic}[1]
\State \textbf{input} $V$, $A$, $V_{\mathrm{fin}}$, $\boldsymbol p$, $\boldsymbol w$
\State\Comment{Initialization:}
\State $\bar V=V$, $\bar A= A$, $\bar V_{\mathrm{fin}}=V_{\mathrm{fin}}$\label{ln:start_init}
\For{$j\in\bar V$}
\State Set $\bar p_j=p_j$, $\bar w_j=w_j$ profit and weight for vertex $j\in\bar V$; $R_j=\{j\}$
\EndFor
\State $k=0$\label{ln:end_init}
\While{$\bar V\neq \emptyset$}\label{ln:start_main_cycle}
\State\Comment{Determine $a^*=(u,v)$ from \cref{prop:best_wing_tail_contracts} and its ratio $\frac{P(F_{a^*})}{W(F_{a^*})}$:}
\State $(u,v)\gets\Call{FindBestWing}{\bar V, \bar A}$\label{ln:start_find_best_wing}\label{ln:end_find_best_wing}
\State\Comment{Determine the set $\mathcal B_{\mathrm{fin}}$ from \cref{prop:best_final_nodes_are_closure_layer} and the ratio $\frac{\bar p_g}{\bar w_g}$:}
\State $(g,\mathcal B_{\mathrm{fin}})\gets\Call{FindBestFinalVertices}{\bar V_{\mathrm{fin}}}$\label{ln:start_find_best_final_vertices}\label{ln:end_find_best_final_vertices}
\State\Comment{If largest ratio is achieved by some final vertices (vertices in $\mathcal B_{\mathrm{fin}}$), we get a new closure layer:}
\If{$\frac{\bar p_g}{\bar w_g}>\frac{P(F_{(u,v)})}{W(F_{(u,v)})}$}\label{ln:start_if_best_final}
\State $\Call{RemoveFinalVertices}{\mathcal B_{\mathrm{fin}}}$\label{ln:end_if_best_final}
\Else\label{ln:start_if_best_wing}
\State\Comment{If largest ratio is achieved by (or also by) $F_{a^*}$ we ``contract'' the arc $a^*$:}
\State $\Call{ContractArc}{u,v}$\label{ln:end_if_best_wing}
\EndIf
\EndWhile\label{ln:end_main_cycle}
\State \Return{$\mathcal T=\big(\mathcal I_1,\mathcal I_2,\ldots,\mathcal I_k\big)$}
\end{algorithmic}
\end{algorithm}

\begin{algorithm}[htbp]
\caption{Sub-procedure \textsc{FindBestWing}: determine arc $a^*=(u,v)$ with largest ratio $\frac{P(F_{a^*})}{W(F_{a^*})}$}\label{alg:find_best_wing}
\small
\begin{algorithmic}[1]
\Procedure{FindBestWing}{$\bar V, \bar A$}
\State Compute $P(F_j)$, $W(F_j)$ for every $j\in\bar V$
\label{ln:compute_profit_weight_forward_trees}
\State $u=0$, $v=0$, $P(F_{(0,0)})=-\infty$, $W(F_{(0,0)})=1$
\For{$a=(i,j)\in\bar A$}
\State $P(F_a)=P(F_i)-P(F_j)$, $W(F_a)=W(F_i)-W(F_j)$\label{ln:compute_profit_weight_forward_trees1}
\If{$\frac{P(F_a)}{W(F_a)}>\frac{P(F_{(u,v)})}{W(F_{(u,v)})}$} $u=i$, $v=j$ \EndIf
\EndFor
\State \Return $(u,v)$
\EndProcedure
\end{algorithmic}
\end{algorithm}

\begin{theorem}\label{prop:main_algorithm_complexity}
{\tt PaC} computes the optimal sequence of closure layers on a forest in $\mathcal{O}(n^2)$ time, while space-complexity is $\mathcal{O}(n)$.
\end{theorem}
\begin{proof}
At each iteration, either Proposition \ref{prop:best_final_nodes_are_closure_layer} identifies the next closure layer and removes it, or Proposition \ref{prop:best_wing_tail_contracts} identifies an arc whose endpoints must belong to the same closure layer and contracts it. These operations preserve the optimal sequence of the original instance, after replacing each contracted vertex by the associated set $R_j$. Since each iteration decreases the cardinality of $\bar V$ by at least one, there are at most $n$ iterations. In a forest, completing the four procedures in \cref{alg:find_best_wing,alg:find_best_final_vertices,alg:remove_final_vertices,alg:contract_arc} can be done in $\mathcal{O}(n)$ time. 
This is rather obvious for the last three procedures, each involving one or two {\tt for} cycles with at most $n$ iterations, and with a constant number of operations per iteration. Concerning the procedure in \cref{alg:find_best_wing}, it also takes $\mathcal{O}(n)$ by a recursive computation of the values $P(F_j)$ and $W(F_j)$ starting from the final vertices of the graph, up to the initial ones.\newline
Hence, the (at most) $n$ iterations, with a computational cost per iteration equal to $\mathcal{O}(n)$, require a number of operations which is $\mathcal{O}(n^2)$.
\newline\newline\noindent
Concerning space, we notice that at each iteration the algorithm keeps track of the different sets of aggregated vertices and of the related profits and weights, which requires $\mathcal{O}(n)$ space,
\end{proof}

\begin{algorithm}[htbp]
\caption{Sub-procedure \textsc{FindBestFinalVertices}: determine set $\mathcal B_{\mathrm{fin}}$ of final vertices with largest ratio $\frac{\bar p_g}{\bar w_g}$}\label{alg:find_best_final_vertices}
\small
\begin{algorithmic}[1]
\Procedure{FindBestFinalVertices}{$\bar V_{\mathrm{fin}}$}
\State $g=0$, $\bar p_0=-\infty$, $\bar w_0=1$, $\mathcal B_{\mathrm{fin}}=\emptyset$
\For{$f\in\bar V_{\mathrm{fin}}$}
\If{$\frac{\bar p_f}{\bar w_f}>\frac{\bar p_g}{\bar w_g}$} $g=f$, $\mathcal B_{\mathrm{fin}}=\{f\}$
\ElsIf{$\frac{\bar p_f}{\bar w_f}=\frac{\bar p_g}{\bar w_g}$} $\mathcal B_{\mathrm{fin}}=\mathcal B_{\mathrm{fin}}\cup\{f\}$
\EndIf
\EndFor
\State \Return $(g,\mathcal B_{\mathrm{fin}})$
\EndProcedure
\end{algorithmic}
\end{algorithm}

\begin{remark}\label{rem:numerical_precision}
We remark that the proposed approach does not suffer from numerical precision problems. Indeed, comparisons between ratios of integers can be obviously done by comparing integer numbers, while all other operations only involve integer values. This is a notable difference with respect, e.g., to the approach based on pseudo-flow computations. Indeed, such approach might be unable to recognize two distinct closure layers when the difference between their ratios is below a given numerical precision, as discussed in \cref{sec:computational_results}.
\end{remark}

\begin{algorithm}[htbp]
\caption{Sub-procedure \textsc{RemoveFinalVertices}: create next closure layer from $\mathcal B_{\mathrm{fin}}$ and update the graph}\label{alg:remove_final_vertices}
\small
\begin{algorithmic}[1]
\Procedure{RemoveFinalVertices}{$\mathcal B_{\mathrm{fin}}$}
\State $k=k+1$, $\mathcal I_k=\bigcup_{f\in\mathcal B_{\mathrm{fin}}}R_f$, $U=\emptyset$
\For{$i\in\bar N^-(\mathcal B_{\mathrm{fin}})$} 
\State $U=U\cup\{i\}$\EndFor 
\For{$(i,f)\in\bar A$ with $f\in\mathcal B_{\mathrm{fin}}$}
\State $\bar A=\bar A\setminus\{(i,f)\}$
\EndFor
\For{$i\in U$} \If{$\bar N^+(i)=\emptyset$} $\bar V_{\mathrm{fin}}=\bar V_{\mathrm{fin}}\cup\{i\}$ \EndIf \EndFor
\State $\bar V_{\mathrm{fin}}=\bar V_{\mathrm{fin}}\setminus\mathcal B_{\mathrm{fin}}$, $\bar V=\bar V\setminus\mathcal B_{\mathrm{fin}}$
\EndProcedure
\end{algorithmic}
\end{algorithm}

\subsection{Dual Variant}\label{sec:dual_alg}
{\tt PaC} finds the closure layers in the optimal sequence in decreasing order with respect to their ratio. It is possible to devise an analogous variant of {\tt PaC}, which finds the closure layers in the optimal sequence in increasing order through an iterative aggregation of vertices. In this variant, instead of computing the ratio of final vertices and preceding sets, we need to compute the ratio of \textit{initial vertices} and \textit{succeeding sets}, namely the sets $V_{\mathrm{init}}$, $B_j$ and $B_a$ already introduced in \cref{sec:notation_trees}. 

We have the following three propositions, whose proofs are  analogous to those of \cref{prop:disjoint_union_wings,prop:best_wing_tail_contracts,prop:best_final_nodes_are_closure_layer}, and we omit them in this paper.

\begin{algorithm}[htbp]
\caption{Sub-procedure \textsc{ContractArc}: contract arc $a^*=(u,v)$ and merge the associated vertex sets}\label{alg:contract_arc}
\small
\begin{algorithmic}[1]
\Procedure{ContractArc}{$u,v$}
\For{$i\in\bar N^-(u)$} $\bar A=\bar A\setminus\{(i,u)\}$, $\bar A=\bar A\cup\{(i,v)\}$ \EndFor
\For{$j\in\bar N^+(u)\setminus\{v\}$} $\bar A=\bar A\setminus\{(u,j)\}$, $\bar A=\bar A\cup\{(v,j)\}$ 
\If{$v\in\bar V_{\mathrm{fin}}$}
\State $\bar V_{\mathrm{fin}}=\bar V_{\mathrm{fin}}\setminus\{v\}$
\EndIf
\EndFor
\State $R_v=R_u\cup R_v$, $\bar p_v=\bar p_u+\bar p_v$, $\bar w_v=\bar w_u+\bar w_v$
\State $\bar A=\bar A\setminus\{(u,v)\}$, $\bar V=\bar V\setminus\{u\}$
\EndProcedure
\end{algorithmic}
\end{algorithm}

\begin{proposition}\label{prop:disjoint_union_fins}
Suppose $G=(V,A)$ is a forest and let $\mathcal M\subset V$ be a subset of vertices. If $N^-(\mathcal M)=\emptyset$, then
\[
\mathcal M=\left(\bigsqcup_{a\in \tilde A} B_a\right)\sqcup \tilde V_{\mathrm{init}}
\]
for some subset $\tilde A\subset A$ and $\tilde V_{\mathrm{init}}\subset V_{\mathrm{init}}$, i.e., $\mathcal M$ is the disjoint union of the minimal succeeding sets associated to some arcs and some subset of initial vertices.
\end{proposition}
\begin{proposition}\label{prop:worst_fin_head_contracts}
Suppose $G=(V,A)$ is a forest and let $a^*\in\arg\min_{a\in A}\frac{P(B_a)}{W(B_a)}$ with $a^*=(u,v)$. If $\frac{P(B_{a^*})}{W(B_{a^*})}\le\frac{p_i}{w_i}$ for every $i\in V_{\mathrm{init}}$, then the two vertices $u$ and $v$ belong to the same closure layer in the optimal sequence.
\end{proposition}
\begin{proposition}\label{prop:worst_initial_nodes_are_closure_layers}
Suppose $G=(V,A)$ is a forest and let $g\in\mathcal B_{\mathrm{init}}:=\arg \min_{i\in V_{\mathrm{init}}}\frac{p_i}{w_i}$. 
If $\frac{p_{g}}{w_{g}}<\frac{P(B_a)}{W(B_a)}$ for every $a\in A$, then $\mathcal B_{\mathrm{init}}$ is the last closure layer in the optimal sequence of closure layers.
\end{proposition}

Thanks to these results, {\tt PaC} can be converted to find the closure layers in the optimal sequence in increasing order with respect to their ratio. We refer to this algorithm as the \textit{dual variant} of {\tt PaC}, and denote it with {\tt DPaC}. An illustration of a full run of this dual variant is reported in \cref{app:ec_dpac_execution}.

We also have the same complexity of $\mathcal{O}(n)$ per iteration and $\mathcal{O}(n^2)$ total complexity, for which we give only the statement, since the proof is analogous to the one of \cref{prop:main_algorithm_complexity}.
\begin{theorem}\label{prop:dual_variant_Ccomplexity}
{\tt DPaC} computes the optimal sequence of closure layers on a forest with time-complexity $\mathcal{O}(n^2)$ and space-complexity $\mathcal{O}(n)$.
\end{theorem}

\begin{remark}\label{rem:primal_dual}
One can also combine the two approaches in a single iteration to halve the number of total iterations, but at the cost of increasing (doubling) the number of operations per iteration.
\end{remark}

\subsection{Improved variants with all out-trees or all in-trees}\label{sec:heaped_alg} 
If we know that the forest graph $G$ is composed of subtrees whose arcs all point in the same direction, we can employ a heap data structure in {\tt PaC} and obtain a better total complexity of $\mathcal{O}(n\log n)$. To describe this variant we first give the definition of out-tree and in-tree.

\begin{definition}\label{def:in-tree}
Let $G=(V,A)$ be a connected tree. We say that $G$ is an \textit{in-tree} if for every vertex $j\in V$ we have $|N^+(j)|\le 1$. As a consequence, there is only one vertex $j_0$ with $N^+(j_0)=\emptyset$, which we call the \textit{root} of the in-tree.
\end{definition}

\begin{definition}\label{def:out-tree}
Let $G=(V,A)$ be a connected tree. We say that $G$ is an \textit{out-tree} if for every vertex $j\in V$ we have $|N^-(j)|\le 1$. As a consequence, there is only one vertex $j_0$ with $N^-(j_0)=\emptyset$, which we call the \textit{root} of the out-tree.
\end{definition}

\cref{fig:forest_types} illustrates these two restricted topologies against a general (mixed) tree, where neither restriction holds.

\begin{figure}[htbp]
\centering
\begin{tikzpicture}[x=1.05cm,y=1.05cm,every node/.style={circle,draw=black,fill=white,inner sep=0pt,minimum size=9pt}]

  \begin{scope}[shift={(0,0)}]
    \node (n1) at (0,2) {};
    \node (n2) at (1,2) {};
    \node (n3) at (0,1) {};
    \node (n4) at (1,1) {};
    \node (n5) at (2,1) {};
    \node (n6) at (0,0) {};
    \node (n7) at (1,0) {};
    \draw[->,line width=.5mm,shorten <=2pt,shorten >=2pt] (n4) -- (n1);
    \draw[->,line width=.5mm,shorten <=2pt,shorten >=2pt] (n4) -- (n2);
    \draw[->,line width=.5mm,shorten <=2pt,shorten >=2pt] (n6) -- (n3);
    \draw[->,line width=.5mm,shorten <=2pt,shorten >=2pt] (n6) -- (n4);
    \draw[->,line width=.5mm,shorten <=2pt,shorten >=2pt] (n7) -- (n4);
    \draw[->,line width=.5mm,shorten <=2pt,shorten >=2pt] (n7) -- (n5);
    \node[draw=none,fill=none,font=\normalsize] at (0.5,-0.6) {(a) mixed tree};
  \end{scope}

  \begin{scope}[shift={(4.2,0)}]
    \node (r)  at (1,2)   {};
    \node (m1) at (0.4,1) {};
    \node (m2) at (2,1)   {};
    \node (l1) at (0,0)   {};
    \node (l2) at (0.9,0) {};
    \node (l3) at (2,0)   {};
    \draw[->,line width=.5mm,shorten <=2pt,shorten >=2pt] (l1) -- (m1);
    \draw[->,line width=.5mm,shorten <=2pt,shorten >=2pt] (l2) -- (m1);
    \draw[->,line width=.5mm,shorten <=2pt,shorten >=2pt] (l3) -- (m2);
    \draw[->,line width=.5mm,shorten <=2pt,shorten >=2pt] (m1) -- (r);
    \draw[->,line width=.5mm,shorten <=2pt,shorten >=2pt] (m2) -- (r);
    \node[draw=none,fill=none,font=\small] at ($(r)+(0,0.34)$) {root};
    \node[draw=none,fill=none,font=\normalsize] at (1,-0.6) {(b) in-tree};
  \end{scope}

  \begin{scope}[shift={(8.4,0)}]
    \node (r)  at (1,0)   {};
    \node (m1) at (0.4,1) {};
    \node (m2) at (2,1)   {};
    \node (l1) at (0,2)   {};
    \node (l2) at (0.9,2) {};
    \node (l3) at (2,2)   {};
    \draw[->,line width=.5mm,shorten <=2pt,shorten >=2pt] (r) -- (m1);
    \draw[->,line width=.5mm,shorten <=2pt,shorten >=2pt] (r) -- (m2);
    \draw[->,line width=.5mm,shorten <=2pt,shorten >=2pt] (m1) -- (l1);
    \draw[->,line width=.5mm,shorten <=2pt,shorten >=2pt] (m1) -- (l2);
    \draw[->,line width=.5mm,shorten <=2pt,shorten >=2pt] (m2) -- (l3);
    \node[draw=none,fill=none,font=\small,anchor=west] at ($(r)+(0.32,0)$) {root};
    \node[draw=none,fill=none,font=\normalsize] at (1,-0.6) {(c) out-tree};
  \end{scope}

\end{tikzpicture}

\Description{Three small directed trees side by side. In the first, arcs point in both directions, so that a vertex may have several predecessors and several successors at once. In the second, every vertex has at most one outgoing arc and all arcs converge towards a single root drawn at the top. In the third, every vertex has at most one incoming arc and all arcs diverge from a single root drawn at the bottom.}
\caption{Three directed trees (arrows point from a vertex to its prerequisite). (a) A general (mixed) tree: a vertex may have several predecessors and several successors at once, so arcs may point in any direction; it is neither an in-tree nor an out-tree. (b) An in-tree (\cref{def:in-tree}): every vertex has at most one successor. (c) An out-tree (\cref{def:out-tree}): every vertex has at most one predecessor. A forest is a disjoint union of such trees.}
\label{fig:forest_types}
\end{figure}

\cref{alg:in-tree}, denoted in what follows with {\tt HIPaC}, is devised for in-trees. In the initialization (lines \ref{ln:heaped_start_init}-\ref{ln:heaped_end_init}) we create a copy of the set of vertices and we insert the elements in a heap with respect to their ratio.
Then in the main cycle (lines \ref{ln:heaped_start_main_cycle}-\ref{ln:heaped_end_main_cycle}), at each iteration we extract the vertex $g$ at the top of the heap.
If vertex $g$ is a final vertex, we add it to the current (or to a new) closure layer (lines \ref{ln:heaped_start_if_best_final}-\ref{ln:heaped_end_if_best_final}), we delete it from the graph and we update the heap accordingly.
If vertex $g$ is not a final vertex, we contract the only arc exiting $g$ and we update the heap accordingly.

\begin{theorem}\label{prop:heaped_algorithm_complexity}
If $G$ is a forest of in-trees, {\tt HIPaC} computes the optimal sequence of closure layers in $\mathcal{O}(n\log n)$ time.
\end{theorem}
\begin{proof}
We assume that the heap is a Fibonacci one (see, e.g., \cite[Section 1.1]{Fibonacci_Heap}). Building the heap requires $\mathcal{O}(n)$ operations. Next, at each iteration of the main cycle, we first remove the element with largest value from the heap, which requires $\mathcal{O}(\log n)$. Next, when we are contracting an arc, we need to update the value of a single element in the heap. This, again, requires $\mathcal{O}(\log n)$ operations. Since the number of iterations is $\mathcal{O}(n)$, we can conclude  that the overall complexity of \cref{alg:in-tree} is $\mathcal{O}(n\log n)$.
\end{proof}

\begin{algorithm}[htbp]
\caption{Improved variant of {\tt PaC}, denoted with {\tt HIPaC}, to compute the optimal sequence of closure layers on a forest of in-trees}
\label{alg:in-tree}
\small
\begin{algorithmic}[1]
\State \textbf{input} $V$, $A$, $V_{\mathrm{fin}}$, $\boldsymbol p$, $\boldsymbol w$
\State\Comment{Initialization:}
\State $\bar V=V$, $\bar A=A$, $\bar V_{\mathrm{fin}}=V_{\mathrm{fin}}$\label{ln:heaped_start_init}
\For{$j\in\bar V$}
\State Set $\bar p_j=p_j$, $\bar w_j=w_j$ profit and weight for vertex $j\in\bar V$; $R_j=\{j\}$
\EndFor
\State $k=0$; Build max-heap with values $\frac{\bar p_j}{\bar w_j}$ for every $j\in\bar V$; $r=+\infty$ \label{ln:heaped_end_init}
\While{$\bar V\neq \emptyset$}\label{ln:heaped_start_main_cycle}
\State Let $g\in\bar V$ be the vertex at the top of the heap (i.e., one of the vertices with largest ratio)
\State\Comment{If top of the heap $g$ is a final vertex, we add it to the next closure layer}
\If{$g\in\bar V_{\mathrm{fin}}$}\label{ln:heaped_start_if_best_final}
\If{$\frac{\bar p_g}{\bar w_g}<r$} $k=k+1$, $\mathcal I_k=R_g$, $r=\frac{\bar p_g}{\bar w_g}$
\Else\ $\mathcal I_k=\mathcal I_k\cup R_g$
\EndIf
\For{$i\in\bar N^-(g)$} $\bar V_{\mathrm{fin}}=\bar V_{\mathrm{fin}}\cup\{i\}$, $\bar A=\bar A\setminus\{(i,g)\}$ \EndFor
\State $\bar V_{\mathrm{fin}}=\bar V_{\mathrm{fin}}\setminus\{g\}$, $\bar V=\bar V\setminus\{g\}$; remove $g$ from heap and update heap\label{ln:heaped_end_if_best_final}
\State\hspace{-1.5em}\Comment{If top of the heap $g$ is not a final vertex, we ``contract'' the only arc $(g,v)$ exiting $g$}
\Else\label{ln:heaped_start_if_best_wing}
\State Let $(g,v)$ be the only arc exiting $g$
\For{$i\in\bar N^-(g)$} $\bar A=\bar A\setminus\{(i,g)\}$, $\bar A=\bar A\cup\{(i,v)\}$ \EndFor
\State $R_v=R_g\cup R_v$, $\bar p_v=\bar p_g+\bar p_v$, $\bar w_v=\bar w_g+\bar w_v$\Comment{The union of sets is performed in $\mathcal{O}(1)$ operations by concatenating lists.}
\State $\bar A=\bar A\setminus\{(g,v)\}$, $\bar V=\bar V\setminus\{g\}$; update $\frac{\bar p_v}{\bar w_v}$ in heap
\EndIf\label{ln:heaped_end_if_best_wing}
\EndWhile\label{ln:heaped_end_main_cycle}
\State \Return{$\mathcal T=\big(\mathcal I_1,\mathcal I_2,\ldots,\mathcal I_k\big)$}
\end{algorithmic}
\end{algorithm}

\begin{remark}\label{rem:why_heap}
Notice that \cref{alg:in-tree} is doing essentially the same operations of {\tt PaC}, except the recomputation, at each iteration in the main cycle, of the values $P(F_j)$, $W(F_j)$ for every $j\in\bar V$ (line \ref{ln:compute_profit_weight_forward_trees} in \cref{alg:find_best_wing}).
Indeed, when the graph is only composed of in-trees, for every arc $a=(i,j)\in A$ the set of vertices $F_a$ is just the singleton $\{i\}$, as illustrated in \cref{fig:wing_collapse}(a).
Furthermore, if at the end of an iteration of the main cycle we remove some final vertices, we do not have to update any of the values $P(F_a),W(F_a)$ with $a\in A$, since final vertices do not belong to sets $F_a$ with $a\in A$, when there are only in-trees.
If instead, at the end of an iteration of the main cycle we contract an arc $a=(i,j)\in A$, we only need to update the value of $F_{a'}$, where $a'$ is the only arc exiting $j$, since in the graph there are only in-trees.
Thus, exploiting the in-tree structure, we can employ a heap to extract the vertex with the largest ratio at each iteration, without unnecessary recomputation.
\end{remark}
\begin{remark}\label{rem:dual_for_heaped}
{\tt DPaC} can also be improved using a heap-based data structure, in the case of a forest of out-trees. Indeed, on an out-tree it is the minimal succeeding set of an arc that collapses to a singleton, $B_a=\{j\}$ for $a=(i,j)$, as illustrated in \cref{fig:wing_collapse}(b), so that a single min-heap over the current vertices again replaces the two-way comparison. In this case we obtain an algorithm which computes the optimal sequence of closure layers in reverse order, i.e., starting from the closure layer with lowest ratio. The resulting algorithm is denoted with {\tt HOPaC}. The total complexity is $\mathcal{O}(n \log n )$ also in this case. We state this result without proof, since it is analogous to the one of \cref{prop:heaped_algorithm_complexity}.
\end{remark}
\begin{theorem}\label{prop:dual_heaped_algorithm_complexity}
If $G$ is a forest of out-trees, {\tt HOPaC} computes the optimal sequence of closure layers in $\mathcal{O}(n\log n)$ time.
\end{theorem}
\begin{figure}[htbp]
\centering
\begin{tikzpicture}[x=1.02cm,y=1.02cm]
  \tikzset{
    wc/.style      ={circle,draw=black,fill=white,minimum size=.66cm,line width=.30mm,font=\small,inner sep=0pt},
    wcpath/.style  ={circle,draw=teal,fill=teal!12,minimum size=.66cm,line width=.34mm,font=\small,inner sep=0pt},
    wcfwing/.style ={circle,draw=orange,fill=orangeback!45,minimum size=.66cm,line width=.40mm,font=\small,inner sep=0pt},
    wcbwing/.style ={circle,draw=darkgreen,fill=greenback!55,minimum size=.66cm,line width=.40mm,font=\small,inner sep=0pt},
    wcarc/.style   ={->,line width=.30mm,shorten <=2pt,shorten >=2pt},
    wcfsel/.style  ={->,orange,line width=.48mm,shorten <=2pt,shorten >=2pt},
    wcbsel/.style  ={->,darkgreen,line width=.48mm,shorten <=2pt,shorten >=2pt},
    wcbox/.style   ={draw=black,densely dashed,rounded corners=.30cm,line width=.30mm,inner sep=2.6mm}
  }

  \begin{scope}[shift={(0,0)}]
    \node[wcfwing] (ai)  at (0,0) {$i$};
    \node[wcpath]  (aj)  at (0,1) {$j$};
    \node[wcpath]  (ak)  at (0,2) {};
    \node[wcpath]  (art) at (0,3) {$j_0$};
    \node[wc] (au1) at (1.25,0) {};
    \node[wc] (au2) at (1.25,1) {};
    \node[wc] (au3) at (1.25,2) {};

    \draw[wcfsel] (ai)--(aj);
    \draw[wcarc]  (aj)--(ak);
    \draw[wcarc]  (ak)--(art);
    \draw[wcarc]  (au1)--(aj);
    \draw[wcarc]  (au2)--(ak);
    \draw[wcarc]  (au3)--(art);
    \node[text=orange,font=\small] at (-.30,.5) {$a$};

    \node[wcbox,fit=(ai)(aj)(ak)(art)] (aFi) {};
    \node[font=\small,anchor=south east] at ([xshift=1mm]aFi.north west) {$F_i$};

    \draw[decorate,decoration={brace,amplitude=5pt,mirror},teal,line width=.32mm]
          ([xshift=-9mm]art.north) -- ([xshift=-9mm]aj.south);
    \node[text=teal,font=\small,anchor=east] at (-1.30,2) {$F_j$};

    \node[text=orange,font=\small,anchor=east] at (-1.30,0) {$F_a=\{i\}$};
    \draw[orange,line width=.28mm,shorten >=2pt] (-1.22,0) -- (ai);

    \node[font=\small] at (-.45,-1.05) {$F_a=F_i\setminus F_j=\{i\}$};
    \node[font=\normalsize] at (-.45,-1.70) {(a) in-tree};
  \end{scope}

  \begin{scope}[shift={(5.7,0)}]
    \node[wcpath]  (brt) at (0,0) {$j_0$};
    \node[wcpath]  (bk)  at (0,1) {};
    \node[wcpath]  (bi)  at (0,2) {$i$};
    \node[wcbwing] (bj)  at (0,3) {$j$};
    \node[wc] (bu1) at (1.25,3) {};
    \node[wc] (bu2) at (1.25,2) {};
    \node[wc] (bu3) at (1.25,1) {};

    \draw[wcarc]  (brt)--(bk);
    \draw[wcarc]  (bk)--(bi);
    \draw[wcbsel] (bi)--(bj);
    \draw[wcarc]  (brt)--(bu3);
    \draw[wcarc]  (bk)--(bu2);
    \draw[wcarc]  (bi)--(bu1);
    \node[text=darkgreen,font=\small] at (-.30,2.5) {$a$};

    \node[wcbox,fit=(brt)(bk)(bi)(bj)] (bBj) {};
    \node[font=\small,anchor=south east] at ([xshift=1mm]bBj.north west) {$B_j$};

    \draw[decorate,decoration={brace,amplitude=5pt,mirror},teal,line width=.32mm]
          ([xshift=-9mm]bi.north) -- ([xshift=-9mm]brt.south);
    \node[text=teal,font=\small,anchor=east] at (-1.30,1) {$B_i$};

    \node[text=darkgreen,font=\small,anchor=east] at (-1.30,3) {$B_a=\{j\}$};
    \draw[darkgreen,line width=.28mm,shorten >=2pt] (-1.22,3) -- (bj);

    \node[font=\small] at (-.45,-1.05) {$B_a=B_j\setminus B_i=\{j\}$};
    \node[font=\normalsize] at (-.45,-1.70) {(b) out-tree};
  \end{scope}
\end{tikzpicture}

\Description{Two directed trees. On the left, an in-tree in which the preceding set of a vertex is the unique directed path from that vertex to the root, with the side branches only entering the path, so that removing the preceding set of the head of an arc leaves its tail alone. On the right, the mirror situation on an out-tree, where the same difference leaves the head of the arc alone.}
\caption{Why the set associated with an arc collapses to a singleton on the two restricted topologies. \textbf{(a)} In an in-tree every vertex has at most one successor, so the preceding set $F_i$ is the unique directed path from $i$ to the root $j_0$ (dashed contour), and the side branches only ever \emph{enter} that path. For $a=(i,j)$ the path $F_j$ (teal) is exactly $F_i$ deprived of its first vertex, hence $F_a=\{i\}$ (orange). \textbf{(b)} In an out-tree every vertex has at most one predecessor, so the succeeding set $B_j$ is the unique directed path from the root $j_0$ up to $j$, and the side branches only ever \emph{leave} it; hence $B_a=\{j\}$ (green). In both cases the collapse holds for every arc, and $a\mapsto$ tail (respectively, head) is a bijection between $\bar A$ and the non-root vertices: the $|\bar A|$ set ratios coincide with the vertex ratios $\bar p_v/\bar w_v$, so a single heap over $\bar V$ replaces the two-way comparison of {\tt PaC} -- a max-heap in the primal algorithm on in-trees, a min-heap in its dual on out-trees. On a general tree neither difference collapses; compare \cref{fig:F_i_B_j_final_initial_vertices}(a).}
\label{fig:wing_collapse}
\end{figure}

\subsection{Heap-based variants}
\cref{rem:why_heap} explains why a heap data structure cannot improve and, actually, worsens the time-complexity of {\tt PaC}, when we have general trees. Indeed, if at the end of an iteration of the main cycle we are removing some final vertex $f\in\bar V_{\mathrm{fin}}$, then it is necessary to update the values of $P(F_a)$, $W(F_a)$ for every $a=(i,j)\in\bar A$ with $i\in B_f$ and $j\notin B_f$, which could be a number of updates of  the order of $\mathcal{O}(n)$ in the worst case. 
Similarly, if at the end of  an iteration of the main cycle we are contracting an arc $a^*\in\bar A$, then it is necessary to update the values of $P(F_a)$, $W(F_a)$ for every $a=(i,j)$ with $i\in B_{a^*}$ and $j\notin B_{a^*}$, which could be a number of updates of the order of $\mathcal{O}(n)$ in the worst case. Thus, maintaining a heap data structure with the relevant ratios, in the case of general trees, would cost $\mathcal{O}(n\log n)$ operations per iteration of the main cycle, in the worst case, which would worsen the total complexity to $\mathcal{O}(n^2 \log n)$ to complete all iterations, instead of $\mathcal{O}(n^2)$ if we just recompute all relevant ratios at each iteration.
Nevertheless, using a heap in the main algorithm still provides a notable speedup in practice, as observed in \cref{sec:computational_results}.
We will denote with {\tt HPaC} the heap-based variant of {\tt PaC}, and with {\tt DHPaC} the heap-based variant of {\tt DPaC}.
One implementation detail is worth stating, since it is what makes the space bound $\mathcal{O}(n)$ hold in practice for the heap-based variants as well: the candidate heaps of {\tt HPaC}/{\tt DHPaC} use lazy deletion, and are rebuilt from the live candidates whenever
their size exceeds a constant multiple of the number of live candidates. This keeps the
memory of the implementation $\mathcal{O}(n)$ on every input – including some instances, such as the high-degree stars, 
where an unbounded lazy heap could use $\mathcal{O}(n^2)$ memory.

\section{Rake-and-Compress (\texttt{RaC}) algorithm for general trees}
\label{sec:rake}
This section is dedicated to the description of the {\tt RaC} algorithm.
Following the ACM Policy on Authorship, we describe here the role played by generative AI in obtaining this algorithm. Motivated by the practical behavior of the heap-based algorithm of \cref{sec:tree_algorithms}, whose computing times increase empirically as $n\log n$ despite its $\mathcal{O}(n^2\log n)$ worst-case complexity, the authors asked \texttt{Claude} (Anthropic) and \texttt{ChatGPT} (OpenAI) whether a suitable data structure could reduce that worst-case complexity to $\mathcal{O}(n\log n)$ on general directed forests. The answer was negative, but both tools suggested a different approach, also based on the idea of contracting arcs, together with a sketch of an $\mathcal{O}(n\log n)$ complexity argument for it. Starting from that suggestion, the authors reconstructed the pseudocode of the approach and completed the complexity proof, which had only been sketched, obtaining the algorithm described below and the results of \cref{sec:complbottomup,sec:compltopdown}. No other conjecture, proof or scientific claim in this paper originates from a generative AI tool, and the authors take full responsibility for all of them.
We restrict the attention to trees since the more general case of a forest can be handled by solving a separate problem over each subtree of the forest. We also assume $|V|\geq 2$: an isolated vertex $v$ has threshold $\theta_v=p_v/w_v$ (see Definition \ref{def:threshold}) and this is directly added to the list of thresholds which is sorted at the end of the top-down phase in \cref{sec:topdown}.
We introduce the following notation.
Let $G=(V,A)$ be a tree. For some $Y\subseteq V$ we denote with $T(Y)=(Y,A(Y))$, where $A(Y)=\{(i,j)\in A\ :\ i,j\in Y\}$, the subgraph induced by $Y$, which is a subforest of the tree $G$. 
We denote with ${\mathcal L}_G\subset V$ the set of leaves of the tree $G$, i.e., the set of vertices belonging to a single arc in $A$ (vertices with degree equal to one). 
We denote with ${\mathcal D}^2_G \subset V$ the set of vertices in $V$ with degree equal to 2 (no matter if in- or out-degree).
Now, we need to introduce the definition of {\em cluster}.

\begin{definition}
A cluster ${\mathcal C}=(C,\partial {\mathcal C})$ is made up of a subset $C\subseteq V$ of vertices, whose induced subgraph $T(C)$ is a subtree of the tree $G$, together with its boundary $\partial {\mathcal C}\subseteq C$, which is made up of a single vertex or by two vertices, and satisfies
\begin{equation}\label{eq:condition_boundary_cluster}
\{i\in C\ :\ \exists\ j\in V\setminus C\ \text{such that}\ (i,j)\in A\ \text{or}\ (j,i)\in A\}\subseteq \partial {\mathcal C}.
\end{equation}

If $|\partial {\mathcal C}|=1$, the cluster is called a {\em 1-cluster}, while if $|\partial {\mathcal C}|=2$, the cluster is called a {\em 2-cluster}.
We also introduce a single cluster with empty boundary and all vertices in its interior, namely, cluster $(V,\emptyset)$.
\end{definition}
To each cluster ${\mathcal C}=(C,\partial {\mathcal C})$ and to each possible subset of its boundary,  defined by a state $\sigma\in \{0,1\}^{\partial {\mathcal C}}$, we associate two functions, $f_{\mathcal C}^\sigma$ and $q_{\mathcal C}^\sigma$.
\begin{definition}
Function
$f_{\mathcal C}^\sigma$ is the equivalent of function $\varphi$ but restricted to the vertices belonging to the interior $\operatorname{int}(C)=C\setminus \partial {\mathcal C}$ of the cluster and subject to the boundary conditions established by state $\sigma$. 
More precisely:
\[
\begin{array}{lll}
f_{\mathcal C}^\sigma(\lambda)=\max\limits_{ {\boldsymbol{x}} \in \{0,1\}^{C}   } & \sum\limits_{i\in \operatorname{int}(C)} (p_i-\lambda\,w_i)\,x_i &  \\ [10pt]
 & x_i-x_j \le 0, & (i,j)\in A,\ i,j\in C, \\ [8pt] 
 & x_i=\sigma_i, & i \in \partial {\mathcal C}. \\[8pt]
 \end{array}
\]
Note that if the feasible region of the problem above is empty (e.g., if $\partial {\mathcal C}=\{i,j\}$, $(i,j)\in A$, and $\sigma_i=1$, $\sigma_j=0$), then we set $f_{\mathcal C}^\sigma(\lambda)=-\infty$ for all $\lambda\in (-\infty,+\infty)$.
Functions $f_{\mathcal C}^\sigma$ are convex piece-wise linear functions characterized by a number of breakpoints not larger than $|\operatorname{int}(C)|$.
\end{definition}
\begin{definition}
Function
$q_{\mathcal C}^\sigma$ is the equivalent of function $\varphi$ restricted to the vertices belonging to the exterior $\operatorname{ext}(C)=V\setminus C$ of the cluster, subject to the boundary conditions established by state $\sigma$, with the addition of the contribution of the boundary itself. 
More precisely.
\[
\begin{array}{lll}
q_{\mathcal C}^\sigma(\lambda)=\max\limits_{ {\boldsymbol{x}} \in \{0,1\}^{\operatorname{ext}(C)\cup \partial {\mathcal C}}   } & \sum\limits_{i\in \operatorname{ext}(C)\cup \partial {\mathcal C}} (p_i-\lambda\,w_i)\,x_i &  \\ [10pt]
 & x_i-x_j \le 0, & (i,j)\in A,\ i,j\in \operatorname{ext}(C)\cup \partial {\mathcal C}, \\ [8pt] 
 & x_i=\sigma_i, & i \in \partial {\mathcal C}. \\[8pt]
 \end{array}
\]
Again, if the feasible region of the problem above is empty, then we set $q_{\mathcal C}^\sigma(\lambda)=-\infty$ for all $\lambda\in (-\infty,+\infty)$.
\end{definition}
We make the following key observation.
\begin{observation}
\label{obs:sep}
Once the boundary state $\sigma$ is fixed, the two problems defining  $f_{\mathcal C}^\sigma$ and  $q_{\mathcal C}^\sigma$ are fully separated, so that
for each cluster ${\mathcal C}$, exploiting \eqref{eq:condition_boundary_cluster}, it holds that
\[
\varphi(\lambda)=\max_{\sigma \in \{0,1\}^{\partial {\mathcal C}}} f_{\mathcal C}^\sigma(\lambda)+q_{\mathcal C}^\sigma(\lambda).
\]
\end{observation}
Functions $f_{\mathcal C}^\sigma$ are computed during the bottom-up phase of the proposed algorithm. As we will see, functions $q_{\mathcal C}^\sigma$ will turn out to be useful to explain how to compute
the closure layers during the top-down phase, but they will not be explicitly computed.
\subsection{Bottom-up phase}
The bottom-up phase is inspired by the tree contraction strategy of Miller and Reif \cite{Miller1989ParallelTC}. 
This phase is based on the execution of two operations called {\em compress} and {\em rake}.
\begin{definition}
Let ${\mathcal C}=(C,\partial {\mathcal C})$ be a 2-cluster with $\partial {\mathcal C}=\{i,j\}$.
The {\em compress} operation ${\tt Compress1}({\mathcal C},i)$
receives in input ${\mathcal C}$, together with its functions $f_{\mathcal C}^{\sigma_i,\sigma_j}$, $\sigma_i,\sigma_j \in \{0,1\}$, and vertex $i\in \partial {\mathcal C}$.
It returns in output
the cluster ${\mathcal C}'=(C,\partial {\mathcal C} \setminus \{i\})$ with its related functions $f_{\mathcal C'}^{\sigma_j}$, $\sigma_j\in \{0,1\}$, defined as follows:
\[
f_{\mathcal C'}^{\sigma_j}(\lambda)=\max_{\sigma_i\in \{0,1\}} f_{\mathcal C}^{\sigma_i,\sigma_j}(\lambda)+\sigma_i(p_i-\lambda w_i).
\]
In a special case which occurs at the end of the bottom-up phase, the operation ${\tt Compress1}$ is also applied to a 1-cluster  ${\mathcal C}$ with boundary $\{i\}$. In that case the single function $f_{\mathcal C'}^{\emptyset}$ is defined as follows:
\[
f_{\mathcal C'}^{\emptyset}(\lambda)=\max_{\sigma_i\in \{0,1\}} f_{\mathcal C}^{\sigma_i}(\lambda)+\sigma_i(p_i-\lambda w_i).
\]
Let ${\mathcal C}_1=(C_1,\partial {\mathcal C}_1), \ {\mathcal C}_2=(C_2,\partial {\mathcal C}_2)$ be two 2-clusters with $\partial {\mathcal C}_1=\{i,j\}$ and $\partial {\mathcal C}_2=\{i,h\}$, i.e., they share vertex $i$ in their boundaries.
The {\em compress} operation ${\tt Compress2}({\mathcal C}_1,{\mathcal C}_2,i)$
receives in input the two clusters ${\mathcal C}_1$ and ${\mathcal C}_2$, together with the vertex $i$ such that $C_1\cap C_2=\{i\}$ (vertex $i$ is the common vertex belonging to the boundaries of the two clusters), and the functions $f_{{\mathcal C}_1}^{\sigma_i,\sigma_j}(\lambda), \sigma_i,\sigma_j\in \{0,1\}$ and $f_{{\mathcal C}_2}^{\sigma_i,\sigma_h}(\lambda), \sigma_i,\sigma_h\in \{0,1\}$. It returns in output
the cluster ${\mathcal C}'=(C_1\cup C_2, \{j,h\})$ with its related functions $f_{\mathcal C'}^{\sigma_j, \sigma_h}$, $\sigma_j,\sigma_h\in \{0,1\}$, defined as follows:
\begin{equation}
\label{eq:rescompr2}
f_{{\mathcal C}'}^{\sigma_j,\sigma_h}(\lambda)=\max_{\sigma_i\in \{0,1\}} f_{{\mathcal C}_1}^{\sigma_i,\sigma_j}(\lambda)+f_{{\mathcal C}_2}^{\sigma_i,\sigma_h}(\lambda)+\sigma_i(p_i-\lambda w_i).
\end{equation}
\end{definition}
\begin{lemma}
\label{lem:breakcom}
The number of breakpoints of functions  $f_{\mathcal C'}^{\sigma_j,\sigma_h}(\lambda)$, where cluster ${\mathcal C}'=(C_1\cup C_2, \{j,h\})$ has been obtained through the application of the {\tt Compress2} operation on the two clusters
${\mathcal C}_1=(C_1,\partial {\mathcal C}_1), \ {\mathcal C}_2=(C_2,\partial {\mathcal C}_2)$ is not larger than $|\operatorname{int}(C_1)|+|\operatorname{int}(C_2)|+1$ and the time to compute them is $\mathcal{O}(|\operatorname{int}(C_1)|+|\operatorname{int}(C_2)|+1)$. Instead, if the cluster ${\mathcal C}'=(C,\partial {\mathcal C} \setminus \{i\})$ has been obtained through the application of the {\tt Compress1} operation on the cluster  ${\mathcal C}=(C,\partial {\mathcal C})$, the number of breakpoints of the functions $f_{\mathcal C'}^{\sigma_j}(\lambda)$ is not larger than $|\operatorname{int}(C)|+1$ and the time to compute them is $\mathcal{O}(|\operatorname{int}(C)|+1)$.
\end{lemma}
\begin{proof}
We only prove the result for the {\tt Compress2} operation. 
We have that
$f_{\mathcal C'}^{\sigma_j,\sigma_h}$ is the parametric optimum over
the lattice of closures of $\mathcal C'$  with fixed boundary, hence by
nestedness it has at most $|\operatorname{int}(C')|=|\operatorname{int}(C_1)|+|\operatorname{int}(C_2)|+1$
breakpoints. To compute it, merge the breakpoint lists of $f_{{\mathcal C}_1}^{0,\sigma_j}+f_{{\mathcal C}_2}^{0,\sigma_h}$ and
$f_{{\mathcal C}_1}^{1,\sigma_j}+f_{{\mathcal C}_2}^{1,\sigma_h}+(p_i-\lambda w_i)$ (each of size $\mathcal{O}(|\operatorname{int}(C_1)|+|\operatorname{int}(C_2)|)$) and take the
pointwise maximum in one linear scan; by nestedness the argmax over
$\sigma_i$ switches at most once, so the scan finds the single
crossing $\bar\lambda$. The overall cost of this operation is $\mathcal{O}(|\operatorname{int}(C_1)|+|\operatorname{int}(C_2)|+1)$.
\end{proof}
\begin{definition}\label{def:rake}
Let ${\mathcal C}_0=(C_0,\partial {\mathcal C}_0)$ be a 2-cluster with $\partial {\mathcal C}_0=\{i,j\}$ and ${\mathcal C}_s=(C_s,\partial {\mathcal C}_s)$, $s\in\{1,2,\ldots,r\}$, be 1-clusters with $\partial {\mathcal C}_s=\{i\}$ and such that $C_s\cap C_t=\{i\}$ for $s\neq t$, $s,t\in\{0,1,\ldots,r\}$ (i.e., all these clusters have pairwise disjoint interiors).
The {\em rake} operation ${\tt Rake}({\mathcal C}_0,\{{\mathcal C}_1,{\mathcal C}_2,\ldots,{\mathcal C}_r\})$
receives in input all those clusters together with the functions $f_{{\mathcal C}_0}^{\sigma_i,\sigma_j}$, $\sigma_i,\sigma_j \in \{0,1\}$ associated to cluster ${\mathcal C}_0$, and functions
$f_{{\mathcal C}_s}^{\sigma_i}$, $\sigma_i\in \{0,1\}$, associated to clusters ${\mathcal C}_s$, $s\in\{1,2,\ldots,r\}$.
It returns in output the cluster  ${\mathcal C}'=(\cup_{s=0}^r C_s, \{i,j\})$ with its related functions $f_{\mathcal C'}^{\sigma_i,\sigma_j}$, $\sigma_i,\sigma_j\in \{0,1\}$.
More precisely:
\[
f_{\mathcal C'}^{\sigma_i,\sigma_j}(\lambda)=f_{{\mathcal C}_0}^{\sigma_i,\sigma_j}(\lambda)+\sum_{s=1}^r f_{{\mathcal C}_s}^{\sigma_i}(\lambda).
\]
The rake operation is also extended to a special case when also ${\mathcal C}_0$ is such that $\partial {\mathcal C}_0=\{i\}$, i.e., it is a 1-cluster with the same boundary of all other clusters ${\mathcal C}_1,{\mathcal C}_2,\ldots,{\mathcal C}_r$.
In this case we set
\[
f_{\mathcal C'}^{\sigma_i}(\lambda)=\sum_{s=0}^r f_{{\mathcal C}_s}^{\sigma_i}(\lambda).
\]
\end{definition}
The following lemma gives the number of breakpoints for the functions associated to clusters obtained through rake operations (later on we will also discuss the number of operations to compute them).
\begin{lemma}
\label{lem:breakrake}
The number of breakpoints for the functions $f_{\mathcal C'}^{\sigma}$ of a cluster ${\mathcal C'}$ obtained through the application of a {\tt rake} operation on clusters ${\mathcal C}_s=(C_s,\partial {\mathcal C}_s)$, $s\in\{0,1,\ldots,r\}$ is not larger than $\sum_{s=0}^r |\operatorname{int}(C_s)|$.
\end{lemma}
\begin{proof}
It is enough to observe that functions $f_{\mathcal C'}^{\sigma}$ are obtained through sums of the functions associated to the clusters ${\mathcal C}_s$, each of which has at most $|\operatorname{int}(C_s)|$ breakpoints.
\end{proof}
The bottom-up phase is detailed in Algorithm \ref{alg:bottomup}. In this algorithm for some vertex $i\in V$ the list $\ell_1(i)$ contains the indices of all 1-clusters where $i$ is in the boundary, while $\ell_2(i)$ contains the indices of all 2-clusters where $i$ is in the boundary. The subset $V'\subseteq V$, initialized with $V$ itself,  is the subset of vertices that at some iteration are in the boundary of at least one of the clusters or, equivalently, that have not been moved in the interior of one cluster yet.
The set $A'$ at some iteration is a set of undirected arcs that join vertices after all the compress operations. It is initialized with the set $A$ of arcs of the tree $G$, ignoring direction of these arcs, but it is not necessarily a subset of $A$ since some compress operations remove two adjacent arcs and substitute them with a new arc joining the two extremes of the adjacent arcs different from the common extreme. With a slight abuse of notation we keep writing $(i,j)$ for the undirected arc of $A'$ joining $i$ and $j$, so that $(i,j)$ and $(j,i)$ denote the same element of $A'$.

\begin{remark}
It is important to point out what we keep in memory for each cluster ${\mathcal C}$. We keep in memory each function $f_{\mathcal C}^\sigma$ together with its breakpoints $\lambda^{\mathcal C}_1,\lambda^{\mathcal C}_2,\ldots,\lambda^{\mathcal C}_t$ (not to be confused with the breakpoints $\lambda_1,\lambda_2,\ldots,\lambda_k$ of $\varphi$) and, for each interval between two consecutive breakpoints
$(\lambda^{\mathcal C}_{i+1},\lambda^{\mathcal C}_{i})$, $i\in\{0,1,\ldots,t\}$, the profit $P\left(S^\star_{{\mathcal C},\sigma}(\lambda^{\mathcal C}_{i+1},\lambda^{\mathcal C}_{i})\right)$ and the weight $W\left(S^\star_{{\mathcal C},\sigma}(\lambda^{\mathcal C}_{i+1},\lambda^{\mathcal C}_{i})\right)$, where $S_{{\mathcal C},\sigma}^\star(\lambda^{\mathcal C}_{i+1},\lambda^{\mathcal C}_{i})$ is the optimal subset of vertices over interval $(\lambda^{\mathcal C}_{i+1},\lambda^{\mathcal C}_{i})$ for cluster ${\mathcal C}$ with boundary conditions $\sigma$ (with $\lambda^{\mathcal C}_0=+\infty$ and $\lambda^{\mathcal C}_{t+1}=-\infty$). 
We notice that we might also keep in memory the optimal subsets $S_{{\mathcal C},\sigma}^\star(\lambda^{\mathcal C}_{i+1},\lambda^{\mathcal C}_{i})$ themselves.
This would considerably simplify the identification of the optimal sequence of closure layers, since they correspond to the differences 
$S_{{\mathcal C}^\star,\emptyset}^\star(\lambda_{i+1},\lambda_{i})\setminus S_{{\mathcal C}^\star,\emptyset}^\star(\lambda_{i},\lambda_{i-1})$, $i\in\{1,2,\ldots,k\}$, where $k$ is the number of breakpoints of function $\varphi$. Indeed, it holds that
${\mathcal M}_i=S_{{\mathcal C}^\star,\emptyset}^\star(\lambda_{i+1},\lambda_{i})$, $i\in\{1,2,\ldots,k\}$. However, this simplification comes at the cost of a memory occupation which rises to $\mathcal{O}(n^2)$. As we will see, we are able to reduce such cost through the top-down phase of the algorithm, which allows computing the closure layers with a computational cost of $\mathcal{O}(n \log n)$.
\end{remark}
Time- and space-complexity $\mathcal{O}(n \log n)$ of the bottom-up phase are proved in Section \ref{sec:complbottomup}.
\begin{algorithm}[htbp]
\caption{Bottom-up phase}\label{alg:bottomup}
\scriptsize
\begin{algorithmic}[1]
\State \textbf{input} $G=(V,A)$,  $\boldsymbol p \in \mathbb{Z}^V$, $\boldsymbol w\in \mathbb{Z}_+^V$
\State\Comment{Initialization:}
\State $\kappa=0$, ${\mathcal A}=\emptyset$
\State Let $T'=(V',A')$ be an undirected tree with $V'=V$ and $A'=A$ \Comment{$T'$ is a tree initially obtained from $G$ by removing arc orientation, but updated (contracted) during the execution of the algorithm}
\For{$i\in V$}
\State $\ell_1(i)=\ell_2(i)=\{ \}$
\EndFor
\For{$(i,j)\in A$}
\State Let ${\mathcal C}_{\kappa+1}=(\{i,j\},\{i,j\})$
\State Set $f_{{\mathcal C}_{\kappa+1}}^{1,0}(\lambda)=-\infty$, $f_{{\mathcal C}_{\kappa+1}}^{0,0}(\lambda)=f_{{\mathcal C}_{\kappa+1}}^{0,1}(\lambda)=f_{{\mathcal C}_{\kappa+1}}^{1,1}(\lambda)=0$ 
\State $\kappa=\kappa+1$
\State Set $\ell_2 (i)=\ell_2 (i)\cup \{\kappa\}$,  $\ell_2 (j)=\ell_2 (j)\cup \{\kappa\}$.
\EndFor
\While{$|V'|>1$}  \Comment{It must hold that $|\ell_2(i)|=1$ for each $i\in {\mathcal L}_{T'}$ and $|\ell_1(i)|=0$ for each $i\in \mathcal L_{T'}\cup\mathcal D^2_{T'}$}\label{ln:startwhile_bottomup}
\If{$|V'|>2$}
\State Compute a maximal independent set $I$ in the subgraph of $T'$ induced by $\mathcal L_{T'}\cup\mathcal D^2_{T'}$ and such that $  {\mathcal L}_{T'}\subseteq I \subseteq {\mathcal L}_{T'}\cup {\mathcal D}^2_{T'}$
\State Set $K_1={\mathcal L}_{T'}$ and $K_2={\mathcal D}^2_{T'}$
\Else
\State Set $I=\{i\}$ where $i\in V'$ is chosen in an arbitrary way, $K_1=V'$, and $K_2=\emptyset$
\EndIf
\For{$i\in K_1 \cap I$}
\State Let $r\in \ell_2(i)$ and $j\in \partial {\mathcal C}_{r}\setminus \{i\}$
\State Let ${\mathcal C}_{\kappa+1}={\tt Compress1}({\mathcal C}_r,i)$
\State Set $\kappa=\kappa+1$, $V'=V'\setminus \{i\}$, $A'=A'\setminus \{(i,j)\}$
\State $\ell_1(j)=\ell_1(j)\cup \{\kappa\}$, $\ell_2 (j)=\ell_2 (j) \setminus \{r\}$
\State Set ${\mathcal A}={\mathcal A} \cup\{ ({\mathcal C}_\kappa,{\mathcal C}_r)\}$ with flag ${\tt comp1}$ for the new arc 
\EndFor
\For{$i\in K_2\cap I$} \Comment{It must hold that $|\ell_2(i)|=2$}
\State Let $r,s\in \ell_2(i)$ with $\partial {\mathcal C}_{r}=\{i,j\}$ and  $\partial {\mathcal C}_{s}=\{i,h\}$;
\State Let ${\mathcal C}_{\kappa+1}={\tt Compress2}({\mathcal C}_r,{\mathcal C}_s,i)$
\State Set $\kappa=\kappa+1$, $V'=V'\setminus \{i\}$, $A'=A'\cup \{(j,h)\}\setminus \{(i,j), (i,h)\}$
\State $\ell_2(j)=\ell_2(j)\cup \{\kappa\}\setminus \{r\}$, $\ell_2 (h)=\ell_2 (h)\cup \{\kappa\} \setminus \{s\}$
\State Set ${\mathcal A}={\mathcal A} \cup\{ ({\mathcal C}_\kappa,{\mathcal C}_r),  ({\mathcal C}_\kappa,{\mathcal C}_s)\}$ with flag ${\tt comp2}$ for the new arcs
\EndFor
\For{$i\in {\mathcal L}_{T'}\cup {\mathcal D}^2_{T'}\ :\ \ell_1(i)\neq \emptyset$} \Comment{$T'$ is the updated tree after the compress operations}
\State Let $r\in \ell_2(i)$ and $j\in \partial {\mathcal C}_{r}\setminus \{i\}$; \Comment{If $|\ell_2(i)|>1$, the choice of $r$ is arbitrary}
\State Set ${\mathcal C}_{\kappa+1}={\tt Rake}({\mathcal C}_r, \{{\mathcal C}_h, h\in \ell_1(i)\})$ 
\State Set $\kappa=\kappa+1$
\State ${\mathcal A}={\mathcal A} \cup \{({\mathcal C}_\kappa,{\mathcal C}_r), ({\mathcal C}_\kappa,{\mathcal C}_h)$, $h\in \ell_1(i)$\}, with flag ${\tt rake}$ for all the new arcs 
\State $\ell_2 (i)=\ell_2 (i)\cup \{\kappa\}\setminus \{r\}$,  $\ell_1 (i)=\emptyset$, $\ell_2 (j)=\ell_2 (j)\cup \{\kappa\}\setminus \{r\}$ \Comment{Note that the rake operation empties all sets $\ell_1(i)$ for $i\in \mathcal L_{T'}\cup\mathcal D^2_{T'}$}
\EndFor
\If{$|V'|=1$}
\State Let $i\in V'$
\If{$\ell_2(i)=\emptyset\ \text{{\tt and}}\ |\ell_1(i)|\geq 2$}
\State Set ${\mathcal C}_{\kappa+1}={\tt Rake}(\{{\mathcal C}_h,\ h\in \ell_1(i)\})$ 
\State Set $\kappa=\kappa+1$
\State ${\mathcal A}={\mathcal A} \cup \{({\mathcal C}_\kappa,{\mathcal C}_h)$, $h\in \ell_1(i)$\}, with flag ${\tt rake}$ for all the new arcs 
\State $\ell_1 (i)=\{\kappa\}$ 
\EndIf
\EndIf
\EndWhile \Comment{When we exit the {\tt while} cycle, ${\mathcal C}_\kappa=(V,\{i\})$ for some $i\in V$}\label{ln:endwhile_bottomup}
\State Let $i\in \partial {\mathcal C}_\kappa$
\State Let ${\mathcal C}^\star={\tt Compress1}({\mathcal C}_\kappa,i)=(V,\emptyset)$\label{ln:finalcompress_bottomup}
\State Let $\varphi(\lambda)=f_{{\mathcal C}^\star}^{\emptyset}(\lambda)$
\State Set ${\mathcal V}=\{{\mathcal C}_1,{\mathcal C}_2,\ldots,{\mathcal C}_\kappa\}$
\State \Return{$\mathcal G=({\mathcal V},{\mathcal A}), \varphi, f_{{\mathcal C}}^\sigma,\ {\mathcal C}\in {\mathcal V},\ \sigma \in \{0,1\}^{\partial {\mathcal C}}$}
\end{algorithmic}
\end{algorithm}
\subsection{Top-down phase}\label{sec:topdown}
The top-down phase, detailed in Algorithm \ref{alg:topdown}, allows computing the threshold $\theta_v$ for each vertex $v\in V$ (see Definition \ref{def:threshold}).
First, we introduce, for each vertex $v$, the two functions
$\hat{g}_x^v(\lambda)$, $x\in \{0,1\}$:
\[
\begin{array}{lll}
\hat{g}_x^v(\lambda)=\max_{ {\boldsymbol{x}} \in \{0,1\}^V   } & \sum_{i\in V} (p_i-\lambda\,w_i)\,x_i &  \\ [8pt]
 & x_i-x_j \le 0, & (i,j)\in A, \\ [8pt] 
& x_v=x, &
 \end{array}
\]
i.e., $\hat{g}_x^v(\lambda)$ is the optimal value of problem (\ref{eq:UPIT_parametric}) when we further impose that $x_v=x$. 
Then, $\varphi(\lambda)=\max\{\hat{g}_0^v(\lambda),\hat{g}_1^v(\lambda)\}$ and the threshold $\theta_v$ is the value at which
$\hat{g}_1^v(\lambda)-\hat{g}_0^v(\lambda)$ changes sign.\newline\newline\noindent
Next, for some $\theta_v\in \mathbb{R}$, let
\[
\sigma_v(\lambda)=\left\{
\begin{array}{ll}
1 & \text{if}\ \lambda\leq \theta_v \\ [8pt]
0 & \text{otherwise},
\end{array}
\right.
\]
be the indicator function of the interval $(-\infty, \theta_v]$.
\newline\newline\noindent
For the root node of the tree of clusters $\mathcal G=({\mathcal V},{\mathcal A})$ returned by the bottom-up phase, we compute functions $g_0^{i_\kappa}$ and $g_1^{i_\kappa}$ as follows:
\begin{equation}
\label{eq:grootnode}
\begin{array}{l}
g_0^{i_\kappa}(\lambda)=f_{{\mathcal C}_\kappa}^0(\lambda) \\ [8pt]
g_1^{i_\kappa}(\lambda)=f_{{\mathcal C}_\kappa}^1(\lambda)+(p_{i_\kappa}-\lambda w_{i_\kappa}).
\end{array}
\end{equation}
Then, we compute the single value $\theta_{i_\kappa}$ at which $g_0^{i_\kappa}(\theta_{i_\kappa})-g_1^{i_\kappa}(\theta_{i_\kappa})$ changes sign, i.e., we compute the threshold of vertex $i_\kappa$.
\newline\newline\noindent
Now, let us assume that by moving from the top of the tree of clusters $\mathcal G$, we reach the single cluster ${\mathcal C}'$ (see Observation \ref{obs:enterone}) where a vertex $b$ has been moved into the interior of that cluster after the execution of a {\tt Compress2} operation on two clusters ${\mathcal C}_1$ and ${\mathcal C}_2$ with boundaries
 $\partial {\mathcal C}_1=\{a,b\}$ and $\partial {\mathcal C}_2=\{b,c\}$. Due to the fact that we move from the top of the tree of clusters, the two thresholds $\theta_a$ and $\theta_c$ have been already computed when we reach such cluster. Then, for $x\in \{0,1\}$, let us define 
 \begin{equation}
 \label{eq:gcomp2}
 g_x^b(\lambda)=f_{{\mathcal C}_1}^{\sigma_a(\lambda),x}(\lambda)+f_{{\mathcal C}_2}^{x,\sigma_c(\lambda)}(\lambda)+x(p_b-\lambda w_b).
\end{equation} 
We prove the following lemma.
\begin{lemma}
\label{lem:changesign}
The threshold $\theta_b$ is the $\lambda$ value at which $g_{1}^b(\lambda)-g_{0}^b(\lambda)$ changes sign. 
\end{lemma}
\begin{proof}
We have that
\[
\hat{g}_{\sigma_b(\lambda)}^b(\lambda)= g_{\sigma_b(\lambda)}^b(\lambda)+q_{{\mathcal C}'}^{\sigma_a(\lambda),\sigma_c(\lambda)}(\lambda)=\varphi(\lambda),
\]
while for $x=1-\sigma_b(\lambda)$ and $\lambda\neq \theta_b$ it holds that
\[
g_{x}^b(\lambda)+q_{{\mathcal C}'}^{\sigma_a(\lambda),\sigma_c(\lambda)}(\lambda)\leq \hat{g}_{x}^b(\lambda)<\varphi(\lambda).
\]
The first inequality follows from the fact that
\[
\begin{array}{lll}
\hat{g}_{x}^b(\lambda)&=&\max_{\sigma_a,\sigma_c\in \{0,1\}}\left[ f_{{\mathcal C}_1}^{\sigma_a,x}(\lambda)+f_{{\mathcal C}_2}^{x,\sigma_c}(\lambda)+x(p_b-\lambda w_b)+q_{{\mathcal C}'}^{\sigma_a,\sigma_c}(\lambda)\right] \geq \\ [8pt]
& \geq & g_{x}^b(\lambda)+q_{{\mathcal C}'}^{\sigma_a(\lambda),\sigma_c(\lambda)}(\lambda).
\end{array}
\]
Therefore, for $\lambda<\theta_b$ we have
\[
\hat{g}_{1}^b(\lambda)=g_{1}^b(\lambda)+q_{{\mathcal C}'}^{\sigma_a(\lambda),\sigma_c(\lambda)}(\lambda)=\varphi(\lambda)>\hat{g}_{0}^b(\lambda)\geq g_{0}^b(\lambda)+q_{{\mathcal C}'}^{\sigma_a(\lambda),\sigma_c(\lambda)}(\lambda),
\]
so that $g_{1}^b(\lambda)-g_{0}^b(\lambda)>0$, while for  $\lambda>\theta_b$
\[
\hat{g}_{0}^b(\lambda)=g_{0}^b(\lambda)+q_{{\mathcal C}'}^{\sigma_a(\lambda),\sigma_c(\lambda)}(\lambda)=\varphi(\lambda)>\hat{g}_{1}^b(\lambda)\geq g_{1}^b(\lambda)+q_{{\mathcal C}'}^{\sigma_a(\lambda),\sigma_c(\lambda)}(\lambda),
\]
so that $g_{1}^b(\lambda)-g_{0}^b(\lambda)<0$. Therefore, the threshold $\theta_b$ is the $\lambda$ value at which $g_{1}^b(\lambda)-g_{0}^b(\lambda)$ changes sign. 
\end{proof}
Assuming, w.l.o.g., that $\theta_a\leq \theta_c$, functions $g_x^b$, $x\in \{0,1\}$, are convex piece-wise linear functions over each one of the three intervals $(-\infty,\theta_a]$, $(\theta_a,\theta_c]$, and $(\theta_c,+\infty)$. Over each interval, the number of breakpoints of functions $g_x^b$, $x\in \{0,1\}$, is $\mathcal{O}(|\operatorname{int}(C_1)|+|\operatorname{int}(C_2)|+1)$. The value at which $g_{1}^b(\lambda)-g_{0}^b(\lambda)$ changes sign is searched for over the three intervals. Note that one should take into account that the sign might change at one extreme of the intervals, i.e., at $\theta_a$ or $\theta_c$. In that case one should study what happens both in the interval $(-\infty,\theta_a]$ and in the interval $(\theta_a,\theta_c]$ (respectively, both in the interval  $(\theta_a,\theta_c]$ and in the interval $(\theta_c,+\infty)$). 
 \newline\newline\noindent
Next, let us assume that we reach the single cluster (see, again, Observation \ref{obs:enterone}) where vertex $b$ has been moved into the interior of such cluster through the execution of a {\tt Compress1} operation on a single cluster ${\mathcal C}_1$ with boundary
 $\partial {\mathcal C}_1=\{a,b\}$. As before, the threshold $\theta_a$ has been already computed. For $x\in \{0,1\}$, we define the two functions
 \begin{equation}
 \label{eq:gcomp1}
 g_x^b(\lambda)=f_{{\mathcal C}_1}^{\sigma_a(\lambda),x}(\lambda)+x(p_b-\lambda w_b).
 \end{equation}
As before, we can prove that the threshold $\theta_b$ is the single $\lambda$ value where $g_0^b(\theta_b)-g_1^b(\theta_b)$ changes sign. In this case, 
functions $g_x^b$, $x\in \{0,1\}$, are convex piece-wise linear functions over each one of the two intervals $(-\infty,\theta_a]$ and $(\theta_a,+\infty)$. Over each interval, the number of breakpoints of functions $g_x^b$, $x\in \{0,1\}$, is $\mathcal{O}(|\operatorname{int}(C_1)|+1)$.
\newline\newline\noindent
We remark that during the top-down phase, arcs with flag {\tt rake}, which do not move vertices in the interiors of clusters, are traversed without doing anything.
\newline\newline\noindent
Time- and space-complexity $\mathcal{O}(n \log n)$ of the top-down phase are proved in Section \ref{sec:compltopdown}.

\Cref{app:ec_rac_execution} illustrates both phases of \texttt{RaC} on the running instance of \cref{fig:closure_layers_example}, treating its arcs as undirected, as \texttt{RaC} does throughout this section.

\begin{algorithm}[htbp]
\caption{Top-down phase}\label{alg:topdown}
\scriptsize
\begin{algorithmic}[1]
\State \textbf{input} the tree of clusters $\mathcal G=({\mathcal V},{\mathcal A})$, its root ${\mathcal C}_\kappa$, functions $f_{{\mathcal C}_i}^\sigma$, $i\in\{1,2,\ldots,\kappa\}$, $\boldsymbol p \in \mathbb{Z}^V$, $\boldsymbol w\in \mathbb{Z}_+^V$
\State\Comment{Initialization:}
\State ${\mathcal Q}=\{{\mathcal C}_\kappa\}$ 
\State Compute functions $g_0^{i_\kappa}$ and $g_1^{i_\kappa}$ through (\ref{eq:grootnode})
\State Compute the threshold $\theta_{i_\kappa}$ at which $g_0^{i_\kappa}(\lambda)-g_1^{i_\kappa}(\lambda)$ changes sign
\While{${\mathcal Q} \neq \emptyset$} 
\State Set ${\mathcal Q}'={\mathcal Q}$
\For{${\mathcal C}_r\in {\mathcal Q}'$}
\If{no outgoing arcs in ${\mathcal A}$  from ${\mathcal C}_r$} \Comment{${\mathcal C}_r$ is a leaf of $\mathcal G$}
\State Set ${\mathcal Q}={\mathcal Q}\setminus \{{\mathcal C}_r\}$ 
\ElsIf{single arc in ${\mathcal A}$ from ${\mathcal C}_r$ with flag {\tt comp1}} \Comment{ ${\mathcal C}_r$ 1-cluster with boundary $\{i_r\}$}
\State Let $({\mathcal C}_r, {\mathcal C}_l)\in {\mathcal A}$, with $l<r$, be the outgoing arc
\State Let $i_l\in \partial {\mathcal C}_l\setminus \partial {\mathcal C}_r$
\State Compute functions $g_0^{i_l}$ and $g_1^{i_l}$ through (\ref{eq:gcomp1}) with ${\mathcal C}_1={\mathcal C}_l$
\State Compute the threshold $\theta_{i_l}$ at which $g_0^{i_l}(\lambda)-g_1^{i_l}(\lambda)$ changes sign
\State Set ${\mathcal Q}={\mathcal Q}\cup \{{\mathcal C}_l\} \setminus \{{\mathcal C}_r\}$ 
\ElsIf{two arcs in ${\mathcal A}$ from ${\mathcal C}_r$ with flag {\tt comp2}} \Comment{ ${\mathcal C}_r$ 2-cluster with boundary $\{i_{l_1},i_{l_2}\}$}
\State Let $({\mathcal C}_r, {\mathcal C}_{l_1}), ({\mathcal C}_r, {\mathcal C}_{l_2})\in {\mathcal A}$, with $l_1,l_2<r$, be the outgoing arcs
\State  Let $i_r\in \partial {\mathcal C}_{l_1}\cap \partial {\mathcal C}_{l_2}$
\State Compute functions $g_0^{i_r}$ and $g_1^{i_r}$ through (\ref{eq:gcomp2}) with ${\mathcal C}_1={\mathcal C}_{l_1}$, ${\mathcal C}_2={\mathcal C}_{l_2}$
\State Compute the threshold $\theta_{i_r}$ at which $g_0^{i_r}(\lambda)-g_1^{i_r}(\lambda)$ changes sign
\State Set ${\mathcal Q}={\mathcal Q}\cup \{{\mathcal C}_{l_1}, {\mathcal C}_{l_2}\} \setminus \{{\mathcal C}_r\}$ 
\Else \Comment{arcs in ${\mathcal A}$ from ${\mathcal C}_r$ with flag {\tt rake}}
\State Let $({\mathcal C}_r, {\mathcal C}_{l_p})\in {\mathcal A}$, with $l_p<r$, $p\in\{1,2,\ldots,P\}$, $P\geq 2$ (flag {\tt rake}), be the outgoing arcs
\State Set ${\mathcal Q}={\mathcal Q}\cup \{{\mathcal C}_{l_p}\}_{p=1}^P \setminus \{{\mathcal C}_r\}$ 
\EndIf
\EndFor
\EndWhile
\State Let $\bar{\lambda}_j$, $j\in\{1,2,\ldots,t\}$, $t\leq |V|$, be the $t$ distinct threshold values $\theta_v$, $v\in V$, sorted in decreasing order.
\State Set ${\mathcal I}_j=\{v\ :\ \theta_v=\bar{\lambda}_j\}$, $j\in\{1,2,\ldots,t\}$
\State \Return{$\{{\mathcal I}_1,{\mathcal I}_2,\ldots, {\mathcal I}_t\},\ \{\bar{\lambda}_1,\bar{\lambda}_2,\ldots,\bar{\lambda}_t\}$}
\end{algorithmic}
\end{algorithm}
\subsection{Complexity of the bottom-up phase}
\label{sec:complbottomup}
To prove the complexity of the bottom-up phase we need some results.
We first make a couple of observations. Note that we call {\em active} clusters those belonging to the sets $\ell_1,\ell_2$, i.e., the set of active clusters is
\[
\mathcal C_{\mathrm{act}}=\left\{\mathcal C_r\ :\ r\in \bigcup_{i\in V'} \left[\ell_1(i)\cup \ell_2(i)\right]\right\}.
\]
\begin{observation}
\label{obs:disjointa}
Active clusters have pair-wise disjoint interiors.
\end{observation}
\begin{observation}
\label{obs:disjointb}
Each vertex belongs to the interior of at most two clusters involved in rake operations executed at a given iteration.
Indeed, a 2-cluster ${\mathcal C}$ with boundary $\{i,j\}$ might be involved in a first rake operation raking the clusters in $\ell_1(i)$. The resulting cluster, denoted with ${\mathcal C}'$ and still with boundary $\{i,j\}$, might be involved in a second rake operation
raking the clusters in $\ell_1(j)$. The interiors of  ${\mathcal C}$ and  ${\mathcal C}'$ may partially overlap, so that a vertex may belong to the interior of at most two clusters involved in rake operations executed at a given iteration.
\end{observation}
Next, we prove that at the beginning of each iteration of the {\tt while} cycle we have that $\ell_1(i)=\emptyset$ for all $i\in \mathcal L_{T'}\cup\mathcal D^2_{T'}$, while the sets $\ell_2(i)$ are related to the set of arcs $A'$.
\begin{lemma}
\label{lem:el1el2}
At the beginning of each iteration of the {\tt while} cycle it holds that:
\begin{itemize}
\item for each $i\in \mathcal L_{T'}\cup\mathcal D^2_{T'}$, $\ell_1(i)=\emptyset$;
\item for each $i\in V'$ it holds that $r\in \ell_2(i)$, with $\partial {\mathcal C}_r=\{i,j\}$ if and only if $(i,j)\in A'$;
\end{itemize}
As a consequence of the second point, it also holds that $|\ell_2(i)|=1$ for $i\in \mathcal L_{T'}$ and $|\ell_2(i)|=2$ for $i\in \mathcal D^2_{T'}$.
\end{lemma}
\begin{proof}
The statement is true after the initialization phase. Next, at the beginning of each iteration $\ell_1(i)=\emptyset$ for $i\in \mathcal L_{T'}\cup\mathcal D^2_{T'}$ since all 1-clusters with $i$ in the boundary are removed through rake operations at the end of the previous iteration.
Concerning the second point, the sentence is certainly true after initialization. A {\tt Compress1} operation on a cluster with boundary $\{i,j\}$, removes such 2-cluster and replaces it with a 1-cluster, by also removing the arc $(i,j)$ from $A'$. Since {\tt Compress2} operations generate new 2-clusters whose boundary is made up of vertices $j,h$ joined by an arc $(j,h)$ added to $A'$, while the vertices of the boundary of 2-clusters generated by rake operations are also connected by arcs in $A'$, this is also true across all the iterations. 
\end{proof}
The following result limits from above the number of active clusters at each iteration.
\begin{lemma}
\label{lem:active}
At each iteration of the {\tt while} cycle, the set of active clusters
\[
\mathcal C_{\mathrm{act}}=\left\{\mathcal C_r\ :\ r\in \bigcup_{i\in V'} \left[\ell_1(i)\cup \ell_2(i)\right]\right\}
\]
contains at most $2n$ clusters. 
\end{lemma}
\begin{proof}
This is true after initialization, when the cardinality of $\cup_{i\in V'} \left[\ell_1(i)\cup \ell_2(i)\right]$ is equal to $|A'|=|A|$. As proved in Lemma \ref{lem:el1el2}, at the beginning of each iteration of the {\tt while} cycle the sets $\ell_1(i)$ are empty for $i\in \mathcal L_{T'}\cup\mathcal D^2_{T'}$. Then, all clusters with $\ell_1(i)\neq \emptyset$ must be such that the degree of $i$ in $T'$ is larger than 2 and must have disjoint interiors in view of Observation \ref{obs:disjointa}. Moreover, again in view of Lemma \ref{lem:el1el2}, there is one 2-cluster associated to each arc in $A'$. Therefore:
\[
|\cup_{i\in V'} \left[\ell_1(i)\cup \ell_2(i)\right]|\leq n+|A'|\leq 2n.
\]
During the execution of an iteration 
the cardinality of $\cup_{i\in V'} \left[\ell_1(i)\cup \ell_2(i)\right]$ does not increase, so that it never exceeds $2n$.
\end{proof}
Next lemma shows that at each iteration the number of vertices in $V'$ is reduced by a constant factor and, consequently, that the number of iterations is $\mathcal{O}(\log n)$.
\begin{lemma}
\label{lem:numbiter}
At each iteration all vertices $i\in I$ are removed from $V'$ and the cardinality of $V'$ at the end of the iteration is reduced to at most $(5/6)$ of its cardinality at the beginning of the iteration. Consequently, the number of iterations is $\mathcal{O}(\log n)$.
\end{lemma}
\begin{proof}
The key observation is that in any tree the number of leaves is at least as large as the number of vertices with degree not lower than three. Consequently,
$|\mathcal L_{T'}|+|\mathcal D^2_{T'}|\ge|V'|/2$. Moreover, the subgraph induced by
$\mathcal L_{T'}\cup\mathcal D^2_{T'}$ has degree not larger than two and by a greedy algorithm one can compute an independent set 
with cardinality $\ge\frac13\big(|\mathcal L_{T'}|+|\mathcal
D^2_{T'}|\big)\ge|V'|/6$. Therefore, at each iteration the cardinality of $V'$ is reduced by a factor $(5/6)$. Since we stop when $|V'|=1$, the maximum number of iterations is $t$ such that
\[
\left(5/6\right)^t n\leq 1,
\]
i.e., $t=\log_{\frac{6}{5}}  n=\mathcal{O}(\log n)$.
\end{proof}
Note that each independent set can be computed by a greedy approach in time $\mathcal{O}(|V'|)$, so that the overall cost for the computation of the independent sets is $\mathcal{O}\left(\sum_{i=1}^t \left(5/6\right)^i n\right)=\mathcal{O}(n)$.
\newline\newline\noindent
Next, we make some observations.
\begin{observation}
\label{obs:enterone}
Each vertex $v$ is moved in the interior of a cluster by a single {\tt Compress1} or {\tt Compress2} operation.
\end{observation}
The following observation is a consequence of Lemmas \ref{lem:breakcom} and \ref{lem:breakrake}. 
\begin{observation}
\label{obs:interiors} 
Any compress or rake operation involving clusters ${\mathcal C}_1,{\mathcal C}_2,\ldots,{\mathcal C}_s$, $s\geq 1$, returns a new cluster ${\mathcal C}'$ and for all possible states $\sigma$ (at most four, since the boundary of each cluster has cardinality at most equal to two), function $f_{{\mathcal C}'}^\sigma$ has at most $1+\sum_{j=1}^s |\operatorname{int}(C_j)|$ breakpoints.
\end{observation}
As a consequence of Observation \ref{obs:interiors}, we have the following corollary.
\begin{corollary}
\label{corol:breakpoints}
The total number of breakpoints associated to active clusters at each iteration is $\mathcal{O}(n)$.
\end{corollary}
\begin{proof}
The total number of breakpoints associated to active clusters is
\[
\sum_{\mathcal C_r \in \mathcal C_{\mathrm{act}}}\sum_{\sigma} \left[|\operatorname{int}(C_r)|+1\right].
\]
Since the interiors of the active clusters involved in compress operations are disjoint, since the possible states $\sigma$ are at most four, and since the number of active clusters is not larger than $2n$, then we have that the total number of breakpoints associated to all active clusters involved in compress operations (and the number of operations required to compute them) is $\mathcal{O}(n)$. In a similar way we see that the total number of breakpoints associated to all active clusters involved in rake operations is $\mathcal{O}(n)$.
\end{proof}
Recalling Lemma \ref{lem:breakcom}, the cost of computing the breakpoints during a compress operation is $\mathcal{O}(|\operatorname{int}(C_1)|+|\operatorname{int}(C_2)|+1)$ for a {\tt Compress2} operation and  $\mathcal{O}(|\operatorname{int}(C)|+1)$ for a {\tt Compress1} operation. The proof
of Corollary \ref{corol:breakpoints} shows that the overall cost of computing breakpoints generated by compress operations at each iteration is $\mathcal{O}(n)$, so that the cost of computing them across all iterations is $\mathcal{O}(n \log n)$. 
\newline\newline\noindent
Now we would like to establish the cost of rake operations. Let us denote with $r^j_1,r^j_2,\ldots,r^j_{H_j}$ the rake operations performed at iteration $j$ of the {\tt while} cycle. Moreover, for a rake operation 
$r_h^j$ and for each fixed boundary state, let ${\mathcal C}_1^{r_h},{\mathcal C}_2^{r_h},\ldots,{\mathcal C}_{\eta_{r_h}}^{r_h}$ be the clusters 
whose breakpoint lists are nonempty. Empty breakpoint lists require no merging work; in particular, the functions associated with an initial cluster are constant, with value either 0 or $-\infty$.
If all breakpoint lists are empty, the merging cost is zero. Otherwise, let 
$L_s^{r_h}$ be the number of breakpoints in the list associated with cluster ${\mathcal C}_s^{r_h}$ (it holds that $1\leq L_s^{r_h}\leq |\operatorname{int}(C_s^{r_h})|$), and let $S^j_{r_h}=\sum_{s=1}^{\eta_{r_h}} L_s^{r_h}$.
Recalling Observation \ref{obs:disjointb},
at each iteration $j$ the interiors of clusters involved in all rake operations are not necessarily disjoint but each vertex appears in at most two interiors. 
Then, $\sum_{h=1}^{H_j} S_{r_h}^j\leq 2 n$.
The Huffman's greedy algorithm (see \cite{huffman1952method,moffat2019huffman}) merges the list of breakpoints of the clusters involved in the rake operation
$r_h^j$ by merging at each iteration the two current shortest lists and thus building a so-called {\em merging tree}. This requires a number of operations equal to
\begin{equation}
\label{eq:merglistrake}
\sum_{s=1}^{\eta_{r_h}} L_s^{r_h}\, d_s^{r_h} \;\le\; \sum_{s=1}^{\eta_{r_h}} L_s^{r_h} \log_2\frac{S_{r_h}^j}{L_s^{r_h}} \;+\; S_{r_h}^j,
\end{equation}
where:
$d_s^{r_h}$ is the depth of the list associated to cluster ${\mathcal C}_s^{r_h}$ in the merging tree.
Also notice that the greedy algorithm needs to build a priority queue whose cost is $\mathcal{O}(\eta_{r_h}\log \eta_{r_h})$ for each rake operation. We notice that summing the values $\eta_{r_h}$ across all rake operations, we have that such sum
is $\mathcal{O}(n)$ since each 1-cluster is involved in at most one rake operation and the overall number of 1-clusters is $\mathcal{O}(n)$ (1-clusters are generated only through {\tt Compress1} operations and their number at each iteration is not larger than $|V'|$). Thus, summing $\mathcal{O}(\eta_{r_h}\log \eta_{r_h})$ across all rake operations gives a cost $\mathcal{O}(n \log n)$ for the priority queues.
\newline\newline\noindent
Now, we need an upper bound for (\ref{eq:merglistrake}), which is obtained as follows. 
Let $\hat d_s=\left\lceil\log_2(S^j_{r_h}/L^{r_h}_s)\right\rceil$. 
Then, $\sum_s 2^{-\hat d_s}\le\sum_s L^{r_h}_s/S^j_{r_h}=1$. In view of Kraft inequality (see, e.g., \cite{CoverThomas2006}) there exists a merging tree with depth $\hat d_s$.
Since the greedy algorithm provides the merging tree with minimum cost, we have that:
\[
\sum_s L^{r_h}_s d_s^{r_h}\;\le\;\sum_s L^{r_h}_s \hat d_s
\;\le\;\sum_s L^{r_h}_s \log_2\frac{S^j_{r_h}}{L^{r_h}_s}+S^j_{r_h}.
\]
Now, let $\nu^{r_h}_s=|\operatorname{int}(C_s^{r_h})|\geq L^{r_h}_s$ and $\Sigma_{r_h}^j=\sum_s\nu_s^{r_h}\geq S^j_{r_h}$.
Then:
\begin{equation}
\label{eq:upbound}
\sum_s L^{r_h}_s \log_2\frac{S^j_{r_h}}{L^{r_h}_s}+S^j_{r_h}
\le 
\sum_s\nu_s^{r_h} \log_2(\Sigma_{r_h}^j/\nu^{r_h}_s) + \Sigma_{r_h}^j.
\end{equation}
Now, we are ready to prove the following result about the cost of rake operations.
\begin{lemma}\label{lem:costrake}
The overall cost of the rake operations executed during the bottom-up phase is $\mathcal{O}(n \log n)$.
\end{lemma}
\begin{proof}
First, we notice that, similarly to $\sum_{h=1}^{H_j} S_{r_h}^j\leq 2 n$, it also holds that $\sum_{h=1}^{H_j} \Sigma_{r_h}^j\leq 2 n$. If we further sum across all iterations the resulting term is $\mathcal{O}(n \log n)$, which is the overall contribution coming from the terms $\Sigma^j_{r_h}$ in the upper bound (\ref{eq:upbound}). Thus, we can restrict the attention to the sum
$\sum_{s=1}^{\eta_{r_h}} \nu_s^{r_h} \log_2(\Sigma_{r_h}^j/\nu_s^{r_h})$. We notice that such sum can be recovered if we assign  
a cost equal to $\log_2(\Sigma_{r_h}^j/\nu_s^{r_h})$ to each vertex $v\in \operatorname{int}(C_s^{r_h})$. Indeed, if we sum across all clusters involved in the rake operation and across all vertices belonging to
the interiors of these clusters, we find exactly the sum $\sum_{s=1}^{\eta_{r_h}} \nu_s^{r_h} \log_2(\Sigma_{r_h}^j/\nu_s^{r_h})$.
\newline\newline\noindent
Now, for a given vertex $v$, let $r_{h}^{j_1},r_{h}^{j_2},\ldots, r_{h}^{j_{z_v}}$ be the rake operations in which $v$ is involved at the iterations $j_1\leq j_2\leq \cdots\leq j_{z_v}$. Recall that two consecutive indices $j_m$ and $j_{m+1}$ can be equal, since, as previously observed, the same vertex may be involved in two rake operations of a given iteration, but not more than two. 
The key observation is the following. Let $L^{m}$ be the value $\nu$ associated to the single cluster containing $v$ at iteration $j_m$, and let $S^{m}$ be the value $\Sigma$ produced by the rake operation involving the cluster containing $v$. It holds that  
$L^{m+1}\ge S^m$ for $m\in\{1,2,\ldots,z_v-1\}$. Indeed, after the rake operation at iteration $j_m$, the resulting cluster has $\nu$-value at least $S^m$, and the inequality follows from the fact that the cardinality of the interiors of the clusters generated by these rake operations never decreases.
As a consequence:
\[
\sum_{m=1}^{z_v}  \log_2\frac{S^m}{L^m} =\sum_{m=1}^{z_v}  \left(\log_2 S^m -\log_2 L^m\right)\leq  \log_2 S^{z_v} - \log_2 L^{1} \le \log_2 n, 
\]
which is the overall cost associated to a single vertex $v$ across all iterations.
Now, to bound from above the cost of all rake operations, we just need to sum the upper bound $\log_2 n$ across all vertices $v$, which leads to the upper bound $\mathcal{O}(n \log n)$.
\end{proof}
Now, having observed that the overall cost of the compress and rake across all iterations of the {\tt while} cycle is $\mathcal{O}(n \log n)$, and 
since $\mathcal{O}(n \log n)$ is also the sum of the cardinality of the interiors of all clusters in ${\mathcal V}$, generated during the bottom-up phase, we have the following complexity result.
\begin{proposition}
\label{prop:complbottom}
The time and space complexity of the bottom-up phase is $\mathcal{O}(n \log n)$.
\end{proposition}
\subsection{Complexity of the top-down phase}
\label{sec:compltopdown}
We have the following result for the complexity of the top-down phase.
\begin{proposition}
The top-down phase requires $\mathcal{O}(n \log n)$ operations.
\end{proposition}
\begin{proof}
We observe that each vertex is compressed in the interior of a cluster exactly once (see Observation \ref{obs:enterone}) and that each cluster (except the root node) has a single father in the tree of clusters $\mathcal G$.
For each vertex $b$ the evaluation of the change of sign for function $g_{1}^b(\lambda)-g_{0}^b(\lambda)$ requires an effort of  $\mathcal{O}(|\operatorname{int}(C_1)|+|\operatorname{int}(C_2)|+1)$ if $b$ has been moved in the interior of a cluster through a {\tt Compress2} operation. Indeed, detecting the change of sign requires a linear scan of the list obtained by merging the breakpoints of the two clusters. Similarly, the cost is  $\mathcal{O}(|\operatorname{int}(C_1)|+1)$ if $b$ has been moved in the interior of a cluster through a {\tt Compress1} operation. 
By summing this effort across  all clusters in ${\mathcal V}$, we have that this is $\mathcal{O}( n \log n)$. Indeed, we proved that the sum of the cardinalities of the interiors of all active clusters at a given iteration is $\mathcal{O}(n)$, so that by summing through all
the $\mathcal{O}(\log n)$ iterations,
we have the  $\mathcal{O}( n \log n)$ result, which is also the cost of the top-down phase.
\end{proof}
Note that the distinct threshold values $\bar{\lambda}_j$, $j\in\{1,2,\ldots,t\}$, obtained at the end of the top-down phase are also the breakpoints of function $\varphi$, so that $t=k$. Moreover, the thresholds are returned in decreasing order, which requires a further $\mathcal{O}(n \log n)$ cost, which does not change the overall complexity of the algorithm.

\section{A comparison between {\tt PaC}/{\tt HPaC} and {\tt RaC}}
From the algorithmic point of view {\tt PaC}/{\tt HPaC} and {\tt RaC} share the common idea of progressively contracting the tree until it is reduced to a single vertex.
In {\tt PaC}/{\tt HPaC} the contraction is a greedy one, where an arc is contracted into a single vertex when it is the current one with largest ratio (or, in the dual version, with smallest ratio).
Instead, in {\tt RaC} contractions are not made to decide which vertices should go into the same closure layer, but to iteratively compute functions $f_{{\mathcal C}}^\sigma$ associated to the clusters, until we are able to
identify function $\varphi$. While contracting vertices, closure layers are not formed. The composition of closure layers is made in the following top-down phase.
\newline\newline\noindent
In terms of time-complexity, {\tt RaC}, whose time-complexity is $\mathcal{O}(n \log n)$, is better than {\tt PaC}, whose time-complexity is $\mathcal{O}(n^2)$, and than  {\tt HPaC}, whose time-complexity even rises to $\mathcal{O}(n^2 \log n)$. 
On the other hand, {\tt RaC} has also space-complexity $\mathcal{O}(n \log n)$, while {\tt PaC}/{\tt HPaC} have space-complexity $\mathcal{O}(n)$.
\newline\newline\noindent
The comparison between these approaches will be completed in the following experimental section (in particular, in Section \ref{subsec:rac_comparison}). There we will also evaluate the practical performance of the approaches and observe that in practical cases {\tt HPaC} usually performs a number of operations  $\approx n \log n$, so that it is often highly competitive with respect to {\tt RaC}, although the latter performs better on special structured instances such as star graphs.

\section{Computational Results}\label{sec:computational_results}
All algorithms are implemented in C++ and run single-threaded. Every ratio comparison uses exact integer arithmetic, so the reported closure layers and thresholds are exact. We report median CPU times and median per-instance paired ratios. The complete experimental protocol, machine specification, validation checks, raw data, and additional analyses are available in the accompanying \texttt{parametric-closure-forests} repository (release \texttt{v0.4.0}; see Code and Data Availability). The C++ implementation of all the algorithms, the instance generators and the benchmark and analysis pipeline underlying this section were developed with the assistance of Claude Code (Anthropic), used interactively as a coding assistant; all the code was independently reviewed by the authors and validated against an exhaustive correctness oracle, cross-algorithm differential testing and sanitizer builds, as documented in the repository.

\subsection{Benchmark instances}\label{subsec:instances}

This subsection describes the generation of the benchmark instances used in the experiments that follow. Every instance combines two independently generated components, an affine-coefficient family and a topology, and is produced by a seeded, deterministic script, so that the whole benchmark can be regenerated exactly.

We begin with the affine coefficients. For each vertex $i$, weight $w_i$ is always drawn strictly positive; profit $p_i$ may be signed. Six families are used, chosen only for the shape of the resulting ratio $p_i/w_i$, never for any external classification:
\begin{itemize}
  \item \texttt{independent-positive}: $w_i,p_i$ drawn independently and uniformly from $\{1,2,\ldots,1000\}$ -- dispersed, positive ratios.
  \item \texttt{independent-signed}: $w_i$ uniform on $\{1,2,\ldots,1000\}$, $p_i$ uniform on $\{-1000,-999,\ldots,1000\}$, independently -- some vertices locally unfavorable.
  \item \texttt{correlated}: $w_i$ uniform on $\{1,2,\ldots,1000\}$, $p_i=w_i+\delta_i$ with $\delta_i$ uniform on $\{-50,-49,\ldots,50\}$ -- ratios concentrated near 1.
  \item \texttt{anti-correlated}: $w_i$ uniform on $\{1,2,\ldots,1000\}$, $p_i=(1001-w_i)+\delta_i$ with $\delta_i$ uniform on $\{-50,-49,\ldots,50\}$ -- more heterogeneous ratios.
  \item \texttt{near-ties}: one target ratio $\mathrm{num}/\mathrm{den}$ per instance, with $\mathrm{num},\mathrm{den}$ uniform on $\{1,2,\ldots,20\}$; for each vertex $i$, $w_i=\mathrm{den}\cdot t_i$ and $p_i=\mathrm{num}\cdot t_i+\jmath_i$, with $t_i$ uniform on $\{1,2,\ldots,50\}$ and jitter $\jmath_i$ uniform on $\{-2,-1,1,2\}$ -- ratios near-equal but distinct, stressing exact rational comparisons between close thresholds.
  \item \texttt{exact-ties}: as \texttt{near-ties}, but with no jitter and with $G=\max(1,\lfloor n/20\rfloor)$ target ratios $\mathrm{num}_g/\mathrm{den}_g$ instead of one, each drawn uniformly on $\{1,2,\ldots,20\}^2$; vertex $i$ is assigned to group $g=i \bmod G$, so that the vertices sharing a ratio are spread over the whole graph rather than adjacent -- exactly equal ratios, stressing the canonicalization rule that merges consecutive closure layers at an equal threshold.
\end{itemize}

We now turn to the topology of the precedence graph, which is generated independently of the coefficients. Six topologies are used. In the three named \texttt{mixed}, every edge is oriented in either direction with probability $1/2$, so that the resulting directed graph is in general neither a forest of in-trees nor a forest of out-trees, i.e., it is the general case addressed by {\tt PaC} and {\tt RaC}.
\begin{itemize}
  \item \texttt{mixed-forest}: a random forest with randomly oriented arcs. An undirected forest is built by connecting each vertex $v\in\{2,3,\ldots,n\}$ to a uniformly chosen predecessor $u<v$ independently with probability $\varrho$ (the density), and each resulting edge is then oriented independently as $(u,v)$ or $(v,u)$ with probability $1/2$ each; $\varrho\in\{0.3,0.6,0.9,1.0\}$: the resulting forest has $\varrho(n-1)$ arcs and $n-\varrho(n-1)$ connected components in expectation, so $\varrho=1.0$ gives a connected spanning tree and $\varrho=0.3$ a forest of many small trees.
  \item \texttt{in-forest}: a random forest of in-trees. Vertices are scanned $u\in\{1,2,\ldots,n-1\}$ and, independently with probability $\varrho$, one arc $(u,j)$ is added with $j$ uniform on $\{u{+}1,u{+}2,\ldots,n\}$, giving out-degree at most one everywhere.
  \item \texttt{out-forest}: a random forest of out-trees. Vertices are scanned $v\in\{2,3,\ldots,n\}$ and, independently with probability $\varrho$, one arc $(u,v)$ is added with $u$ uniform on $\{1,2,\ldots,v{-}1\}$, giving in-degree at most one everywhere.
  \item \texttt{path-mixed}, \texttt{binary-mixed}, \texttt{star-mixed}: three structured topologies, whose underlying undirected graph is, respectively, a path, a balanced binary tree and a star. In the implementation the vertices are labelled from $0$ to $n-1$: the parent of $v>0$ is $v-1$ on the path and $\lfloor (v-1)/2\rfloor$ on the binary tree, while the center of a star is vertex 0. As for \texttt{mixed-forest}, orient every edge independently as $(\text{parent},v)$ or $(v,\text{parent})$ with probability $1/2$.
\end{itemize}

The benchmark is obtained by combining each topology with each affine-coefficient family. Throughout this section, $n$ denotes the number of vertices of an instance, and every reported value is aggregated over all the instances that share the same topology and the same value of $n$. For each of the three random topologies \texttt{mixed-forest}, \texttt{in-forest} and \texttt{out-forest}, which depend on the density $\varrho$, and for each value of $n$, we generate one instance for every combination of the four densities, the six coefficient families and ten random seeds, that is, $4\times 6\times 10=240$ instances per topology and per value of $n$. For each of the three structured topologies \texttt{path-mixed}, \texttt{binary-mixed} and \texttt{star-mixed}, which have no density parameter, and for each value of $n$, we generate one instance for every combination of the six coefficient families and ten random seeds, that is, $6\times 10=60$ instances per topology and per value of $n$.

\subsection{Justification of the Heap-Based Variant}\label{subsec:heap_based_variant}
We first compare the direct-scan algorithms ({\tt PaC} and its dual {\tt DPaC}) with their heap-based counterparts ({\tt HPaC}, {\tt DHPaC}) on \texttt{mixed-forest} instances with $n\in\{100,200,\ldots,1\,000\}\cup\{10\,000,20\,000\}$, the sizes on which the direct-scan algorithms are affordable. \cref{fig:hfma_vs_fma_scaling} reports the median CPU times. On these instances we observed that the four algorithms are essentially equivalent up to $n\approx 300$ (at $n=100$ the direct scan is in fact marginally faster, median {\tt PaC}/{\tt HPaC} $=0.91$), and that the heap pays off from $n\approx 400$ on: the median ratio {\tt PaC}/{\tt HPaC} grows to $2.3$ at $n=1\,000$, $20.3$ at $n=10\,000$, and $42.9$ at $n=20\,000$, with the duals behaving analogously.

\begin{figure}[htbp]
\centering

\begin{minipage}[c]{0.47\textwidth}
\centering
\scriptsize
\setlength{\tabcolsep}{3.2pt}
\begin{tabular}{@{}rrrrrr@{}}
\toprule
& & \multicolumn{4}{c}{CPU time (ms)} \\
\cmidrule(l){3-6}
$n$ & \#inst & \texttt{PaC} & \texttt{HPaC} & \texttt{DPaC} & \texttt{DHPaC} \\
\midrule
100 & 240 & 0.1 & 0.1 & 0.1 & 0.1 \\
200 & 240 & 0.2 & 0.2 & 0.2 & 0.2 \\
300 & 240 & 0.3 & 0.3 & 0.3 & 0.3 \\
400 & 240 & 0.5 & 0.3 & 0.5 & 0.3 \\
500 & 240 & 0.7 & 0.4 & 0.7 & 0.4 \\
600 & 240 & 0.9 & 0.5 & 1.0 & 0.5 \\
700 & 240 & 1.2 & 0.6 & 1.3 & 0.6 \\
800 & 240 & 1.8 & 0.8 & 1.6 & 0.7 \\
900 & 240 & 2.2 & 0.9 & 2.0 & 0.8 \\
1\,000 & 240 & 2.2 & 0.9 & 2.4 & 0.9 \\
10\,000 & 240 & 315.1 & 12.7 & 417.5 & 13.2 \\
20\,000 & 240 & 1\,425.2 & 28.0 & 1\,897.3 & 29.8 \\
\bottomrule
\end{tabular}
\end{minipage}%
\hfill
\begin{minipage}[c]{0.49\textwidth}
\centering
\begin{tikzpicture}
\begin{axis}[width=\linewidth, height=4.8cm,
  xlabel={$n$}, ylabel={CPU time (ms)}, xmode=log, ymode=log,
  xmin=100, xmax=20000,
  xtick={100,1000,10000}, xticklabels={$10^2$,$10^3$,$10^4$},
  scaled x ticks=false, pcfplot]
\addplot[color=red!80, mark=square*, thick] coordinates {
  (100,0.068) (200,0.175) (300,0.311) (400,0.478) (500,0.676)
  (600,0.909) (700,1.222) (800,1.761) (900,2.169) (1000,2.217)
  (10000,315.1) (20000,1425.2)};
\addlegendentry{\texttt{PaC}}
\addplot[color=red!80, mark=diamond*, thick, dashed] coordinates {
  (100,0.055) (200,0.180) (300,0.327) (400,0.505) (500,0.720)
  (600,0.973) (700,1.280) (800,1.612) (900,1.994) (1000,2.377)
  (10000,417.5) (20000,1897.3)};
\addlegendentry{\texttt{DPaC}}
\addplot[color=darkgreen, mark=*, thick] coordinates {
  (100,0.075) (200,0.165) (300,0.254) (400,0.337) (500,0.421)
  (600,0.514) (700,0.582) (800,0.830) (900,0.946) (1000,0.896)
  (10000,12.67) (20000,27.99)};
\addlegendentry{\texttt{HPaC}}
\addplot[color=darkgreen, mark=pentagon*, thick, dashed] coordinates {
  (100,0.064) (200,0.171) (300,0.287) (400,0.342) (500,0.434)
  (600,0.515) (700,0.608) (800,0.749) (900,0.842) (1000,0.946)
  (10000,13.21) (20000,29.83)};
\addlegendentry{\texttt{DHPaC}}
\end{axis}
\end{tikzpicture}
\end{minipage}
\Description{A table and a log-log plot of the median CPU time in milliseconds against the number of vertices, from one hundred to twenty thousand, for the two direct-scan algorithms and their two heap-based counterparts. The four curves nearly coincide up to about three hundred vertices and then separate: the heap-based ones reach about 28 milliseconds at twenty thousand vertices, the direct-scan ones between 1425 and 1897 milliseconds.}
\caption{Median CPU time of {\tt PaC}, {\tt DPaC}, {\tt HPaC} and {\tt DHPaC} on \texttt{mixed-forest} instances.}
\label{fig:hfma_vs_fma_scaling}
\end{figure}

\subsection{Scaling on Large Random Forests and Structured Instance Classes}\label{subsec:hfma_scaling}

\cref{fig:hfma_large_scaling} reports the median CPU times of {\tt HPaC}, {\tt DHPaC} and {\tt RaC} on \texttt{mixed-forest} instances with $n$ from $10\,000$ to $100\,000$ in steps of $10\,000$. {\tt HPaC} solves the full parametric problem on forests with $10^5$ vertices in about $0.2$\,s; {\tt DHPaC} stays within $25\%$ of it at every size (median paired ratio between $0.96$ and $1.24$), so the choice between the primal and the dual schedule is immaterial at this scale. On these instances the median paired ratio {\tt RaC}/{\tt HPaC} grows from $4.2$ at $n=10^4$ to $13.3$ at $n=10^5$. This aggregate ordering is unchanged across the tested densities and affine-coefficient families.

\begin{figure}[htbp]
\centering

\begin{minipage}[c]{0.47\textwidth}
\centering
\scriptsize
\setlength{\tabcolsep}{3.2pt}
\begin{tabular}{@{}rrrrr@{}}
\toprule
& & \multicolumn{3}{c}{CPU time (ms)} \\
\cmidrule(l){3-5}
$n$ & \#inst & \texttt{HPaC} & \texttt{DHPaC} & \texttt{RaC} \\
\midrule
10\,000 & 240 & 12.7 & 13.2 & 50.5 \\
20\,000 & 240 & 28.0 & 29.8 & 131.5 \\
30\,000 & 240 & 44.9 & 47.2 & 262.8 \\
40\,000 & 240 & 62.9 & 66.9 & 399.7 \\
50\,000 & 240 & 84.3 & 92.3 & 543.9 \\
60\,000 & 240 & 96.1 & 117.9 & 652.5 \\
70\,000 & 240 & 117.4 & 141.7 & 837.0 \\
80\,000 & 240 & 141.1 & 161.5 & 1\,027.5 \\
90\,000 & 240 & 164.9 & 198.1 & 1\,304.7 \\
100\,000 & 240 & 203.9 & 214.6 & 1\,709.4 \\
\bottomrule
\end{tabular}
\end{minipage}%
\hfill
\begin{minipage}[c]{0.49\textwidth}
\centering
\begin{tikzpicture}
\begin{axis}[width=0.97\linewidth, height=4.8cm,
  xlabel={$n$}, ylabel={CPU time (ms)},
  xmin=10000, xmax=100000,
  xtick={10000,40000,70000,100000},
  minor xtick={20000,30000,50000,60000,80000,90000},
  xminorgrids=true, minor grid style={gray!15},
  scaled x ticks=false, /pgf/number format/.cd, fixed, 1000 sep={}, /tikz/.cd, pcfplot]
\addplot[color=blue!70, mark=triangle*, thick] coordinates {
  (10000,50.52) (20000,131.54) (30000,262.82) (40000,399.68) (50000,543.87)
  (60000,652.49) (70000,836.95) (80000,1027.47) (90000,1304.68) (100000,1709.35)};
\addlegendentry{\texttt{RaC}}
\addplot[color=darkgreen, mark=pentagon*, thick, dashed] coordinates {
  (10000,13.21) (20000,29.83) (30000,47.16) (40000,66.90) (50000,92.26)
  (60000,117.87) (70000,141.74) (80000,161.51) (90000,198.05) (100000,214.58)};
\addlegendentry{\texttt{DHPaC}}
\addplot[color=darkgreen, mark=*, thick] coordinates {
  (10000,12.67) (20000,27.99) (30000,44.93) (40000,62.88) (50000,84.29)
  (60000,96.14) (70000,117.43) (80000,141.06) (90000,164.92) (100000,203.94)};
\addlegendentry{\texttt{HPaC}}
\end{axis}
\end{tikzpicture}
\end{minipage}
\Description{A table and a plot of the median CPU time in milliseconds against the number of vertices, from ten thousand to one hundred thousand, for the two heap-based algorithms and for the rake-and-compress one. The latter grows from 51 to 1709 milliseconds, while the two heap-based curves stay almost flat by comparison, from about 13 to about 215 milliseconds.}
\caption{Median CPU time of {\tt HPaC}, {\tt DHPaC} and {\tt RaC} on large \texttt{mixed-forest} instances.}
\label{fig:hfma_large_scaling}
\end{figure}

We next focus on three structured instance classes -- paths, balanced binary trees, and stars, always with mixed arc orientations -- which isolate the effect of the topology.\label{subsec:rac_comparison}
Paths and binary trees behave like the random forests: the median ratio {\tt RaC}/{\tt HPaC} lies between $5.8$ and $9.0$ on paths and between $3.2$ and $3.9$ on binary trees, from $n=100$ to $n=100\,000$. Stars invert the picture, increasingly sharply as $n$ grows (\cref{fig:star}): {\tt RaC} is already faster than {\tt HPaC} at $n=100$ (median ratio $0.75$) and is about $250$ times faster at $n=20\,000$. The mechanism is structural: on a star, every peel or contraction touches the hub and invalidates the candidate ratios of a constant fraction of the remaining leaves, so any heap-based schedule pays $\Theta(n)$ maintenance per event, while {\tt RaC}'s rake operations absorb all leaves in $\mathcal{O}(\log n)$ rounds regardless of the hub's update history. The direct scan is much less affected, since it performs no candidate maintenance at all: {\tt PaC} stays about one order of magnitude below {\tt HPaC} over the whole range, and {\tt RaC} overtakes it between $n=1\,000$ and $n=2\,000$.

\begin{figure}[htbp]
\centering

\begin{minipage}[c]{0.47\textwidth}
\centering
\scriptsize
\setlength{\tabcolsep}{3.2pt}
\begin{tabular}{@{}rrrrr@{}}
\toprule
& & \multicolumn{3}{c}{CPU time (ms)} \\
\cmidrule(l){3-5}
$n$ & \#inst & \texttt{PaC} & \texttt{HPaC} & \texttt{RaC} \\
\midrule
100 & 60 & 0.1 & 0.6 & 0.4 \\
200 & 60 & 0.2 & 2.0 & 0.8 \\
500 & 60 & 1.1 & 12.6 & 2.1 \\
1\,000 & 60 & 4.5 & 62.0 & 5.2 \\
2\,000 & 60 & 17.6 & 231.2 & 9.7 \\
5\,000 & 60 & 123.8 & 1\,427.6 & 24.8 \\
10\,000 & 60 & 643.1 & 6\,898.8 & 63.3 \\
20\,000 & 60 & 2\,890.8 & 26\,127.1 & 122.1 \\
\bottomrule
\end{tabular}
\end{minipage}%
\hfill
\begin{minipage}[c]{0.49\textwidth}
\centering
\begin{tikzpicture}
\begin{axis}[width=\linewidth, height=4.8cm,
  xlabel={$n$}, ylabel={CPU time (ms)}, xmode=log, ymode=log,
  xmin=100, xmax=20000,
  xtick={100,1000,10000}, xticklabels={$10^2$,$10^3$,$10^4$},
  scaled x ticks=false, pcfplot]
\addplot[color=darkgreen, mark=*, thick] coordinates {
  (100,0.56) (200,1.99) (500,12.62) (1000,62.00) (2000,231.24)
  (5000,1427.55) (10000,6898.75) (20000,26127.06)};
\addlegendentry{\texttt{HPaC}}
\addplot[color=red!80, mark=square*, thick] coordinates {
  (100,0.07) (200,0.22) (500,1.11) (1000,4.48) (2000,17.56)
  (5000,123.75) (10000,643.06) (20000,2890.83)};
\addlegendentry{\texttt{PaC}}
\addplot[color=blue!70, mark=triangle*, thick] coordinates {
  (100,0.43) (200,0.81) (500,2.08) (1000,5.19) (2000,9.72)
  (5000,24.75) (10000,63.28) (20000,122.07)};
\addlegendentry{\texttt{RaC}}
\end{axis}
\end{tikzpicture}
\end{minipage}
\Description{A table and a log-log plot of the median CPU time in milliseconds against the number of vertices, from one hundred to twenty thousand, on star instances, for the direct-scan, the heap-based and the rake-and-compress algorithms. The heap-based curve is the highest and the steepest, from 0.6 to 26127 milliseconds; the direct-scan one stays about one order of magnitude below it; the rake-and-compress one grows only from 0.4 to 122 milliseconds and crosses the direct-scan one between one and two thousand vertices.}
\caption{Median CPU time of {\tt PaC}, {\tt HPaC} and {\tt RaC} on \texttt{star-mixed} instances.}
\label{fig:star}
\end{figure}

\subsection{Specialized Variants for In-Trees and Out-Trees}\label{subsec:specialized_algorithms_performance}

We evaluate the specialized $\mathcal{O}(n\log n)$ heap variants on the orientation each is built for: {\tt HIPaC} on \texttt{in-forest} and {\tt HOPaC} on \texttt{out-forest} instances ($4\,800$ instances per orientation: $20$ sizes, four densities, six families, ten seeds), always against the general {\tt HPaC} on the same instances. \cref{fig:specialized_large_scaling} reports the median times for $n\in\{10\,000,20\,000,\ldots,100\,000\}$. On these instances the specialized variants take a stable fraction of {\tt HPaC}'s time -- median per-size ratios between $0.52$ and $0.75$ for {\tt HIPaC} and between $0.57$ and $0.78$ for {\tt HOPaC} -- with no erosion of the advantage up to $n=10^5$. The gain reflects \cref{rem:why_heap}: on a single orientation the minimal preceding (succeeding) set of every arc collapses to a singleton, so one heap value per vertex suffices and no closure-sum propagation is ever needed. On the same oriented instances {\tt RaC} loses ground as $n$ grows, with median {\tt RaC}/{\tt HPaC} rising from about $2.8$ at $n\le 1\,000$ to $12.8$ at $n=10^5$ on in-forests, and similarly on out-forests.

\begin{figure}[htbp]
\centering

\begin{minipage}[c]{0.47\textwidth}
\centering
\scriptsize
\setlength{\tabcolsep}{3.2pt}
\begin{tabular}{@{}rrrrrr@{}}
\toprule
& & \multicolumn{2}{c}{\texttt{in-forest} (ms)} & \multicolumn{2}{c}{\texttt{out-forest} (ms)} \\
\cmidrule(lr){3-4}\cmidrule(l){5-6}
$n$ & \#inst & \texttt{HIPaC} & \texttt{HPaC} & \texttt{HOPaC} & \texttt{HPaC} \\
\midrule
10\,000 & 240 & 6.5 & 14.2 & 7.6 & 13.7 \\
20\,000 & 240 & 14.4 & 29.5 & 16.5 & 31.9 \\
30\,000 & 240 & 25.0 & 49.6 & 27.8 & 51.4 \\
40\,000 & 240 & 35.2 & 73.8 & 40.3 & 73.6 \\
50\,000 & 240 & 45.1 & 95.3 & 54.7 & 95.3 \\
60\,000 & 240 & 52.8 & 109.4 & 61.5 & 106.1 \\
70\,000 & 240 & 66.0 & 127.8 & 75.5 & 135.6 \\
80\,000 & 240 & 78.7 & 168.4 & 88.4 & 164.4 \\
90\,000 & 240 & 86.5 & 203.1 & 106.0 & 185.7 \\
100\,000 & 240 & 114.2 & 237.6 & 132.1 & 229.5 \\
\bottomrule
\end{tabular}
\end{minipage}%
\hfill
\begin{minipage}[c]{0.49\textwidth}
\centering
\begin{tikzpicture}
\begin{axis}[width=0.97\linewidth, height=4.8cm,
  xlabel={$n$}, ylabel={CPU time (ms)},
  xmin=10000, xmax=100000,
  xtick={10000,40000,70000,100000},
  minor xtick={20000,30000,50000,60000,80000,90000},
  xminorgrids=true, minor grid style={gray!15},
  scaled x ticks=false, /pgf/number format/.cd, fixed, 1000 sep={}, /tikz/.cd, pcfplot]
\addplot[color=orange!90!black, mark=triangle*, thick] coordinates {
  (10000,6.5) (20000,14.4) (30000,25.0) (40000,35.2) (50000,45.1)
  (60000,52.8) (70000,66.0) (80000,78.7) (90000,86.5) (100000,114.2)};
\addlegendentry{\texttt{HIPaC}}
\addplot[color=darkgreen, mark=*, thick] coordinates {
  (10000,14.2) (20000,29.5) (30000,49.6) (40000,73.8) (50000,95.3)
  (60000,109.4) (70000,127.8) (80000,168.4) (90000,203.1) (100000,237.6)};
\addlegendentry{\texttt{HPaC} (in)}
\addplot[color=violet, mark=square*, thick] coordinates {
  (10000,7.6) (20000,16.5) (30000,27.8) (40000,40.3) (50000,54.7)
  (60000,61.5) (70000,75.5) (80000,88.4) (90000,106.0) (100000,132.1)};
\addlegendentry{\texttt{HOPaC}}
\addplot[color=darkgreen, mark=*, thick, dashed] coordinates {
  (10000,13.7) (20000,31.9) (30000,51.4) (40000,73.6) (50000,95.3)
  (60000,106.1) (70000,135.6) (80000,164.4) (90000,185.7) (100000,229.5)};
\addlegendentry{\texttt{HPaC} (out)}
\end{axis}
\end{tikzpicture}
\end{minipage}
\Description{A table and a plot of the median CPU time in milliseconds against the number of vertices, from ten thousand to one hundred thousand, with four curves: the two specialized heap variants on the oriented instances they are designed for, and the general heap-based algorithm on the same instances. Each specialized curve stays at roughly half the time of the corresponding general one at every size.}
\caption{Median CPU time of the specialized variants {\tt HIPaC} and {\tt HOPaC} against the general {\tt HPaC}, on \texttt{in-forest} and \texttt{out-forest} instances.}
\label{fig:specialized_large_scaling}
\end{figure}

\subsection{Comparison with the Fully Parametric Pseudoflow Algorithm}\label{subsec:pseudoflow_comparison}

Finally, we compare {\tt HPaC} with the public implementation of the \emph{fully parametric} pseudoflow algorithm of \cite{HC08} (\texttt{FPHPF} in what follows; version 2, May 2025, see Code and Data Availability), which applies to arbitrary precedence graphs. As already mentioned, its complexity is $\mathcal{O}(mn\log n)$ on a graph with $n$ vertices and $m$ arcs; since a forest has $m\le n-1$, on the instances considered here that bound reads $\mathcal{O}(n^2\log n)$. 
Each instance is encoded through the standard maximum-closure/minimum-cut reduction, with source- and sink-adjacent arc capacities affine in the parameter; the range is $[-\max_i|p_i|-1,\,\max_i|p_i|+1]$, which contains every vertex ratio since $w_i\ge 1$. From the output of \texttt{FPHPF} we take only the partition of the vertices into closure layers, and we recompute every breakpoint exactly as $\lambda_r=P(\mathcal I_r)/W(\mathcal I_r)$, so that no floating-point value produced by \texttt{FPHPF} enters the comparison. Its internal solve timer (input parsing excluded) is compared against {\tt HPaC}'s in-process time for the full parametric sweep on the same instance; both are medians of $11$ repetitions, measured in the same session. The comparison runs on all $2\,400$ \texttt{mixed-forest} instances with $n\in\{100,200,\ldots,1\,000\}$.

\begin{figure}[htbp]
\centering

\begin{minipage}[c]{0.47\textwidth}
\centering
\scriptsize
\setlength{\tabcolsep}{3.2pt}
\begin{tabular}{@{}rrrrr@{}}
\toprule
& & \multicolumn{2}{c}{CPU time (ms)} & \\
\cmidrule(lr){3-4}
$n$ & \#inst & \texttt{HPaC} & \texttt{FPHPF} & ratio \\
\midrule
100 & 240 & 0.15 & 1.98 & 9.2 \\
200 & 240 & 0.37 & 5.85 & 14.2 \\
300 & 240 & 0.52 & 9.62 & 17.0 \\
400 & 240 & 0.65 & 12.09 & 18.0 \\
500 & 240 & 0.71 & 14.11 & 18.6 \\
600 & 240 & 0.71 & 15.96 & 20.5 \\
700 & 240 & 0.76 & 18.12 & 22.6 \\
800 & 240 & 0.90 & 23.79 & 21.7 \\
900 & 240 & 1.00 & 29.16 & 24.9 \\
1\,000 & 240 & 1.07 & 35.04 & 27.6 \\
\bottomrule
\end{tabular}
\end{minipage}%
\hfill
\begin{minipage}[c]{0.49\textwidth}
\centering
\begin{tikzpicture}
\begin{axis}[width=\linewidth, height=4.8cm,
  xlabel={$n$}, ylabel={CPU time (ms)}, ymode=log,
  xmin=100, xmax=1000,
  xtick={100,400,700,1000},
  minor xtick={200,300,500,600,800,900},
  xminorgrids=true, minor grid style={gray!15},
  scaled x ticks=false, /pgf/number format/.cd, fixed, 1000 sep={}, /tikz/.cd, pcfplot]
\addplot[color=brown, mark=square*, thick] coordinates {
  (100,1.977) (200,5.854) (300,9.625) (400,12.088) (500,14.112)
  (600,15.958) (700,18.119) (800,23.789) (900,29.157) (1000,35.038)};
\addlegendentry{\texttt{FPHPF}}
\addplot[color=darkgreen, mark=*, thick] coordinates {
  (100,0.148) (200,0.373) (300,0.525) (400,0.646) (500,0.714)
  (600,0.707) (700,0.762) (800,0.898) (900,1.003) (1000,1.069)};
\addlegendentry{\texttt{HPaC}}
\end{axis}
\end{tikzpicture}
\end{minipage}
\Description{A table and a semi-logarithmic plot of the median CPU time in milliseconds against the number of vertices, from one hundred to one thousand, for the heap-based algorithm of this paper and for the fully parametric pseudoflow implementation. The pseudoflow curve grows from about 2 to about 35 milliseconds, the other one stays between 0.15 and 1.07 milliseconds, so the gap widens from about nine to about twenty-eight times.}
\caption{Median CPU time of {\tt HPaC} and of the fully parametric pseudoflow implementation \texttt{FPHPF} on \texttt{mixed-forest} instances, both solving each instance from scratch; ``ratio'' is the median over the instances of the paired ratio \texttt{FPHPF}/{\tt HPaC}.}
\label{fig:fphpf}
\end{figure}

On these instances the two methods return the same closure-layer partition and the same breakpoints on every instance, and {\tt HPaC} is faster on every instance. The median paired ratio \texttt{FPHPF}/{\tt HPaC} grows from $9.2$ at $n=100$ to $27.6$ at $n=1\,000$ (\cref{fig:fphpf}), with an overall median of $19.5$.
The gap keeps widening with size: on the larger \texttt{mixed-forest} instances the median ratio is about $250$ at $n=10^4$, $550$ at $n=2\cdot10^4$, $1\,300$ at $n=5\cdot10^4$ and $4\,000$ at $n=10^5$, where a single \texttt{FPHPF} run takes about $15$ minutes against $0.23$\,s for {\tt HPaC}; the complete measurements are reported in the repository.

From $n=10^4$ on, \texttt{FPHPF} also stops returning the exact layer sequence on some instances: it merges consecutive closure layers on $1$ of $24$ instances at $n=10^4$, on $2$ of $24$ at $n=2\cdot10^4$, on $2$ of $6$ at $n=5\cdot10^4$, and on the $n=10^5$ instance tested. The cause is a fixed absolute tolerance in the implementation, which brackets every computed intersection of two linear pieces by $\pm10^{-7}$ in the parameter, so that two breakpoints closer than about $2\cdot10^{-7}$ cannot both be found; the thresholds involved are perfectly distinct as floating-point numbers (in the first such case they are $2299/2226$ and $2362/2287$, whose difference is $1/5\,090\,862$). On the $2\,400$ instances with $n\le 1\,000$ the smallest gap between consecutive breakpoints is $2.21\cdot10^{-7}$, just above that threshold, and indeed no merge occurs; as $n$ grows the breakpoints become denser and the merges become systematic. This is the numerical-precision phenomenon of \cref{rem:numerical_precision}; the algorithms of this paper are unaffected, since they only compare integers. 
\newline\newline\noindent
For the sake of correctness, although the above results show that {\tt HPaC} outperforms \texttt{FPHPF}, we stress once again that \texttt{FPHPF} applies to arbitrary precedence graphs, while our algorithms are specific to forests.

\section{Conclusion}\label{sec:Conclusion}
In this paper we studied the parametric MCP with affine costs when the precedence graph is a forest. We presented two aggregation-based algorithms, {\tt PaC} and {\tt DPaC}, that compute the complete sequence of closure layers in $\mathcal{O}(n^2)$ time and $\mathcal{O}(n)$ space on general forests. On the restricted classes of in-trees and out-trees, their specialized heap-based variants {\tt HIPaC} and {\tt HOPaC} run in $\mathcal{O}(n\log n)$ time. For general forests, the heap-based implementations {\tt HPaC} and {\tt DHPaC} retain an $\mathcal{O}(n^2\log n)$ worst-case bound. We also presented {\tt RaC}, which computes the same parametric solution on general forests in $\mathcal{O}(n\log n)$ time and space, which matches the $\Omega(n\log n)$ lower bound for comparison-based algorithms.

The computational study complements these bounds. On the tested random forests, paths, and balanced binary trees, {\tt HPaC} is faster than {\tt RaC}; on stars, {\tt RaC} becomes increasingly faster as the instance size grows. The specialized variants {\tt HIPaC} and {\tt HOPaC} consistently improve on the general heap-based implementation on their respective oriented classes. Finally, we compared {\tt HPaC} with the public implementation of the fully parametric pseudoflow algorithm of \cite{HC08}: {\tt HPaC} is faster on every tested instance, by a margin growing with $n$, and, relying only on exact integer comparisons, it is not affected by the numerical tolerance which makes that implementation merge close layers on the larger instances. This comparison is of course limited to forests, whereas the pseudoflow method applies to arbitrary precedence graphs.

Several directions are left open by this work. A first one concerns the class of precedence graphs. The clusters used by {\tt RaC} have a boundary of at most two vertices, which is what keeps the number of boundary states, and hence the work per cluster, constant. It is natural to ask whether the same machinery, applied to the decomposition tree of a series-parallel digraph, recovers the $\mathcal{O}(n\log n)$ bound known for that class \cite{Lawler78,Eppstein2017}, and how far it can be pushed towards richer classes: precedence graphs of bounded treewidth, where a boundary of $k$ vertices leads to $2^k$ boundary states, or graphs that differ from a forest by a bounded number of arcs.

A second direction concerns the problem itself. Both {\tt PaC} and {\tt RaC} rely on precedence constraints being hard: in the minimum-cut formulation of the closure problem, the arcs encoding them have infinite capacity, so that contracting one of them never contributes to the objective. Extending the approach to parametric minimum-cut problems on tree-shaped networks with finite arc capacities, where a violated precedence is paid for rather than forbidden, would require the aggregation steps of {\tt PaC} and the cluster functions of {\tt RaC} to keep track of the capacities of the arcs involved, and would make the present algorithms applicable to a genuinely larger class of parametric problems.

\begingroup \parindent 0pt \parskip 4ex
\def\enotesize{\normalsize} 
\endgroup

\section*{Code and Data Availability}

The public GitHub repository \url{https://github.com/fabiofurini/parametric-closure-forests} contains the implementation, instance generators, full experimental protocol, complete computational details, and the raw and processed benchmark data underlying \cref{sec:computational_results}. The archived release \texttt{v0.4.0} (commit \texttt{9423cd4}) is self-contained and includes the full instance archive (\texttt{instances.zip}) and the raw and processed benchmark data (\texttt{results.zip}), together with their SHA-256 checksums. The pseudoflow implementation compared against in \cref{subsec:pseudoflow_comparison} is the fully parametric HPF solver of the Hochbaum group, version 2 (May 2025), available at \url{https://github.com/hochbaumGroup/pseudoflow-parametric-cut-v2} (commit \texttt{fbd480f}); its source is not redistributed in our repository but downloaded from the authors' repository by a script pinned to that commit, and the only local changes are two output format strings (timing and parameter values printed with more digits), recorded as a patch.

\bibliographystyle{ACM-Reference-Format}
\bibliography{biblio}
\appendix

\section*{Appendices}

The three appendices report complete step-by-step executions of the three algorithms on the running instance of \cref{fig:closure_layers_example}, with the quantities compared at every iteration shown explicitly.

\captionsetup{font=small}

\section{Execution of \texttt{PaC}}
\label{app:ec_pac_execution}

\newcommand{\pacsplit}[3]{%
  \draw[line width=.2mm] (#1.north) -- (#1.south);
  \node[font=\tiny] at ([xshift=-.15cm]#1.center) {$#2$};
  \node[font=\tiny] at ([xshift=.15cm]#1.center) {$#3$};
}
\newcommand{\pacidx}[2]{%
  \node[font=\tiny,anchor=north west,inner sep=0pt]
    at ([xshift=.075cm,yshift=-.055cm]#1.south east) {$#2$};%
}
\newcommand{\pacidxsw}[2]{%
  \node[font=\tiny,anchor=north east,inner sep=0pt]
    at ([xshift=-.075cm,yshift=-.055cm]#1.south west) {$#2$};%
}
\newcommand{\pactop}[2]{%
  \node[pacratio,anchor=south,inner sep=0pt] at ([yshift=.13cm]#1.north) {$#2$};%
}
\newcommand{\pacal}[5][.5]{%
  \coordinate (pac@m) at ($(#2)!#1!(#3)$);
  \node[pacratio] at ($(pac@m)!#4!90:(#3)$) {$#5$};%
}
\tikzset{
  pacnormal/.style={circle,draw=black,fill=white,minimum size=.74cm,line width=.31mm},
  pacwing/.style={circle,draw=orange,fill=orangeback!18,minimum size=.74cm,line width=.36mm},
  pacfinal/.style={circle,draw=teal,fill=teal!10,minimum size=.74cm,line width=.36mm},
  pacarc/.style={->,line width=.30mm,shorten <=2pt,shorten >=2pt},
  pacselarc/.style={->,orange,dashed,line width=.42mm,shorten <=2pt,shorten >=2pt},
  pacinside/.style={->,orange,line width=.38mm,shorten <=2pt,shorten >=2pt},
  pacratio/.style={text=purple,font=\scriptsize,inner sep=.4pt}
}

\setlength{\abovecaptionskip}{3pt}
\setlength{\belowcaptionskip}{0pt}
\captionsetup{font=small}

At every iteration, \texttt{PaC} computes the largest ratio among the minimal preceding sets associated with the current arcs,
\[
 \mu_A=\max_{a\in\bar A}\frac{P(F_a)}{W(F_a)},
\]
and the largest ratio among the current final vertices,
\[
 \mu_{\mathrm{fin}}=\max_{f\in\bar V_{\mathrm{fin}}}\frac{\bar p_f}{\bar w_f}.
\]
Let $a^*=(u,v)$ attain $\mu_A$ and let $\mathcal B_{\mathrm{fin}}$ contain all final vertices attaining $\mu_{\mathrm{fin}}$. If $\mu_{\mathrm{fin}}>\mu_A$, the sets represented by the vertices in $\mathcal B_{\mathrm{fin}}$ form the next closure layer and are removed. If $\mu_A\geq\mu_{\mathrm{fin}}$, arc $a^*$ is contracted and its endpoints are represented by a single aggregated vertex. Thus, equality is resolved by contraction. In the diagrams, purple labels are the ratios compared at the current iteration; orange identifies the selected wing $F_{a^*}$ and its boundary arc $a^*$; teal identifies the selected final vertices.

\noindent\textbf{Step 1 -- first comparison and contraction of $(7,4)$.}
Initially the final vertices are $\{1,2,3,5\}$, and the largest final-vertex ratio is $\mu_{\mathrm{fin}}=2$, attained by $\mathcal B_{\mathrm{fin}}=\{1,5\}$. The largest wing ratio is attained by $a^*=(7,4)$: its minimal preceding set is $F_{(7,4)}=\{7,5\}$ and has ratio $7/2$. Since $7/2>2$, the algorithm contracts $(7,4)$. Vertex 4 becomes the representative of the merged set $R_4=\{4,7\}$, with $\bar p_4=4$ and $\bar w_4=2$.

\begin{center}
\begin{tikzpicture}[x=1.62cm,y=1.62cm]
  \node[pacnormal] (n1) at (0,2) {}; \pacsplit{n1}{2}{1}
  \node[pacnormal] (n2) at (1,2) {}; \pacsplit{n2}{-2}{1}
  \node[pacnormal] (n3) at (0,1) {}; \pacsplit{n3}{-2}{1}
  \node[pacnormal] (n4) at (1,1) {}; \pacsplit{n4}{-1}{1}
  \node[pacwing]   (n5) at (2,1) {}; \pacsplit{n5}{2}{1}
  \node[pacnormal] (n6) at (0,0) {}; \pacsplit{n6}{4}{3}
  \node[pacwing]   (n7) at (1,0) {}; \pacsplit{n7}{5}{1}

  \draw[pacarc]    (n4)--(n1);
  \draw[pacarc]    (n4)--(n2);
  \draw[pacarc]    (n6)--(n3);
  \draw[pacarc]    (n6)--(n4);
  \draw[pacselarc] (n7)--(n4);
  \draw[pacinside] (n7)--(n5);

  \pacal[.40]{n4}{n1}{.30cm}{-\frac32}
  \pacal[.42]{n4}{n2}{-.30cm}{\frac12}
  \pacal{n6}{n3}{.30cm}{\frac12}
  \pacal{n6}{n4}{-.32cm}{\frac12}
  \pacal[.45]{n7}{n4}{-.32cm}{\frac72}
  \pacal{n7}{n5}{-.32cm}{1}

  \pactop{n1}{2}\pactop{n2}{-2}\pactop{n3}{-2}\pactop{n5}{2}

  \pacidxsw{n1}{1}\pacidx{n2}{2}\pacidx{n3}{3}\pacidx{n4}{4}
  \pacidx{n5}{5}\pacidx{n6}{6}\pacidx{n7}{7}
\end{tikzpicture}
\captionof*{figure}{\textbf{(a)} Initial PaC iteration. The selected wing is $F_{(7,4)}=\{7,5\}$, with ratio $7/2$; the best final vertices have ratio $2$.}
\end{center}

\noindent\textbf{Step 2 -- a tie and the second contraction.}
After contracting $(7,4)$, the current representative 4 stores $R_4=\{4,7\}$. The best wing is now associated with $a^*=(4,2)$: $F_{(4,2)}=\{1,4,5\}$ has ratio $2$. The best final vertices are still $\mathcal B_{\mathrm{fin}}=\{1,5\}$, also with ratio $2$. Because the algorithm contracts whenever $\mu_A\geq\mu_{\mathrm{fin}}$, the tie is resolved by contracting $(4,2)$. Vertex 2 becomes the representative of $R_2=\{2,4,7\}$, with $\bar p_2=2$ and $\bar w_2=3$.

\begin{center}
\begin{tikzpicture}[x=1.62cm,y=1.62cm]
  \node[pacwing]   (n1) at (0,2) {}; \pacsplit{n1}{2}{1}
  \node[pacnormal] (n2) at (1,2) {}; \pacsplit{n2}{-2}{1}
  \node[pacnormal] (n3) at (0,1) {}; \pacsplit{n3}{-2}{1}
  \node[pacwing]   (n4) at (1,1) {}; \pacsplit{n4}{4}{2}
  \node[pacwing]   (n5) at (2,2) {}; \pacsplit{n5}{2}{1}
  \node[pacnormal] (n6) at (0,0) {}; \pacsplit{n6}{4}{3}

  \draw[pacinside] (n4)--(n1);
  \draw[pacselarc] (n4)--(n2);
  \draw[pacarc]    (n6)--(n3);
  \draw[pacarc]    (n6)--(n4);
  \draw[pacinside] (n4)--(n5);

  \pacal[.42]{n4}{n1}{.32cm}{1}
  \pacal[.42]{n4}{n2}{-.30cm}{2}
  \pacal{n6}{n3}{.30cm}{\frac54}
  \pacal{n6}{n4}{-.32cm}{\frac12}
  \pacal{n4}{n5}{-.32cm}{1}

  \pactop{n1}{2}\pactop{n2}{-2}\pactop{n3}{-2}\pactop{n5}{2}

  \pacidxsw{n1}{1}\pacidx{n2}{2}\pacidx{n3}{3}\pacidx{n4}{4}\pacidx{n5}{5}\pacidx{n6}{6}
\end{tikzpicture}
\captionof*{figure}{\textbf{(b)} After the first contraction. The best wing and the best final vertices both have ratio $2$, so PaC contracts $(4,2)$.}
\end{center}

\noindent\textbf{Step 3 -- the first peel and the first closure layer.}
After contracting $(4,2)$, vertex 2 represents $R_2=\{2,4,7\}$. The current final vertices are $\{1,3,5\}$; the best ones are again $\mathcal B_{\mathrm{fin}}=\{1,5\}$ with ratio $2$. The best wing is $F_{(6,3)}$, whose ratio is $5/4$. Hence $\mu_{\mathrm{fin}}>\mu_A$, and the final vertices are peeled. Since $R_1=\{1\}$ and $R_5=\{5\}$, the first closure layer is
\[
 \mathcal I_1=R_1\cup R_5=\{1,5\},\qquad \lambda_1=2.
\]
Removing vertices 1 and 5 leaves the three-vertex tree shown in the next step.

\begin{center}
\begin{tikzpicture}[x=1.62cm,y=1.62cm]
  \node[pacfinal]  (n1) at (0,2) {}; \pacsplit{n1}{2}{1}
  \node[pacnormal] (n3) at (0,1) {}; \pacsplit{n3}{-2}{1}
  \node[pacnormal] (n2) at (1,1) {}; \pacsplit{n2}{2}{3}
  \node[pacfinal]  (n5) at (2,2) {}; \pacsplit{n5}{2}{1}
  \node[pacnormal] (n6) at (0,0) {}; \pacsplit{n6}{4}{3}

  \draw[pacarc] (n2)--(n1);
  \draw[pacarc] (n2)--(n5);
  \draw[pacarc] (n6)--(n3);
  \draw[pacarc] (n6)--(n2);

  \pacal[.42]{n2}{n1}{.32cm}{1}
  \pacal{n2}{n5}{-.32cm}{1}
  \pacal{n6}{n3}{.30cm}{\frac54}
  \pacal{n6}{n2}{-.32cm}{\frac12}

  \pactop{n1}{2}\pactop{n3}{-2}\pactop{n5}{2}

  \pacidxsw{n1}{1}\pacidx{n3}{3}\pacidx{n2}{2}\pacidx{n5}{5}\pacidx{n6}{6}
\end{tikzpicture}
\captionof*{figure}{\textbf{(c)} After the second contraction. The best final vertices $\{1,5\}$ have ratio $2$, larger than the best wing ratio $5/4$, and therefore form $\mathcal I_1$.}
\end{center}

\noindent\textbf{Step 4 -- contraction of $(6,3)$.}
After peeling $\mathcal I_1$, the remaining graph contains representatives 2, 3, and 6. Vertex 2 is the best final vertex, with ratio $2/3$. The best wing is associated with $a^*=(6,3)$ and is $F_{(6,3)}=\{6,2\}$, with ratio $1$. Since $1>2/3$, PaC contracts $(6,3)$. Vertex 3 becomes the representative of $R_3=\{3,6\}$, with $\bar p_3=2$ and $\bar w_3=4$.

\begin{center}
\begin{tikzpicture}[x=1.62cm,y=1.62cm]
  \node[pacnormal] (n3) at (0,1) {}; \pacsplit{n3}{-2}{1}
  \node[pacwing]   (n2) at (1,1) {}; \pacsplit{n2}{2}{3}
  \node[pacwing]   (n6) at (0,0) {}; \pacsplit{n6}{4}{3}

  \draw[pacselarc] (n6)--(n3);
  \draw[pacinside] (n6)--(n2);

  \pacal{n6}{n3}{.30cm}{1}
  \pacal{n6}{n2}{-.32cm}{\frac12}

  \pactop{n3}{-2}\pactop{n2}{\frac23}

  \pacidx{n3}{3}\pacidx{n2}{2}\pacidx{n6}{6}
\end{tikzpicture}
\captionof*{figure}{\textbf{(d)} After peeling $\mathcal I_1$. The wing $F_{(6,3)}=\{6,2\}$ has ratio $1$, larger than the best final-vertex ratio $2/3$.}
\end{center}

\noindent\textbf{Step 5 -- the second peel.}
After contracting $(6,3)$, only representatives 2 and 3 remain, with the single arc $(3,2)$. The wing $F_{(3,2)}=\{3\}$ has ratio $1/2$, while the final representative 2 has ratio $2/3$. The final vertex therefore wins the comparison and is peeled. Since it represents $R_2=\{2,4,7\}$, the second closure layer is
\[
 \mathcal I_2=\{2,4,7\},\qquad \lambda_2=\frac23.
\]

\begin{center}
\begin{tikzpicture}[x=1.62cm,y=1.62cm]
  \node[pacfinal]  (n2) at (1,1) {}; \pacsplit{n2}{2}{3}
  \node[pacnormal] (n3) at (0,0) {}; \pacsplit{n3}{2}{4}
  \draw[pacarc] (n3)--(n2);
  \pacal{n3}{n2}{-.32cm}{\frac12}
  \pactop{n2}{\frac23}
  \pacidx{n2}{2}\pacidx{n3}{3}
\end{tikzpicture}
\captionof*{figure}{\textbf{(e)} After contracting $(6,3)$. The final representative 2 has ratio $2/3$, larger than the only wing ratio $1/2$, and therefore yields $\mathcal I_2$.}
\end{center}

\noindent\textbf{Step 6 -- the last closure layer.}
After removing representative 2, the graph consists only of representative 3. There are no arcs, so the only possible move is to peel this final vertex. It represents $R_3=\{3,6\}$ and has ratio $1/2$. Hence
\[
 \mathcal I_3=\{3,6\},\qquad \lambda_3=\frac12.
\]
The graph is now empty and PaC terminates.

\begin{center}
\begin{tikzpicture}[x=1.62cm,y=1.62cm]
  \node[pacfinal] (n3) at (0,0) {}; \pacsplit{n3}{2}{4}
  \pactop{n3}{\frac12}
  \pacidx{n3}{3}
\end{tikzpicture}
\captionof*{figure}{\textbf{(f)} Final iteration. Representative 3 is the only remaining final vertex and yields $\mathcal I_3=\{3,6\}$.}
\end{center}

\noindent\textbf{Result.}
The contractions have only aggregated vertices that are eventually assigned to the same closure layer; the peel operations identify the layers themselves. The complete optimal sequence returned by PaC is therefore
\[
 \mathcal I_1=\{1,5\},\qquad
 \mathcal I_2=\{2,4,7\},\qquad
 \mathcal I_3=\{3,6\},
\]
with strictly decreasing ratios $2$, $2/3$, and $1/2$, respectively.

\begin{center}
\begin{tikzpicture}[x=1cm,y=1cm]
  \tikzset{paclayer/.style={draw=teal,rounded corners=.10cm,minimum width=3.15cm,
                            minimum height=1.15cm,align=center,font=\small,inner sep=2pt}}
  \node[paclayer] (l1) at (0,0)
     {$\mathcal I_1=\{1,5\}$\\[.8mm]$\lambda_1=2$};
  \node[paclayer] (l2) at (4.0,0)
     {$\mathcal I_2=\{2,4,7\}$\\[.8mm]$\lambda_2=\frac23$};
  \node[paclayer] (l3) at (8.0,0)
     {$\mathcal I_3=\{3,6\}$\\[.8mm]$\lambda_3=\frac12$};
  \draw[->,line width=.30mm] (l1)--(l2);
  \draw[->,line width=.30mm] (l2)--(l3);
\end{tikzpicture}
\captionof*{figure}{\textbf{(g)} Closure layers returned by PaC, in decreasing order of their ratios.}
\end{center}

\section{Execution of \texttt{DPaC}}
\label{app:ec_dpac_execution}

\newcommand{\dpacsplit}[3]{%
  \draw[line width=.2mm] (#1.north) -- (#1.south);
  \node[font=\tiny] at ([xshift=-.15cm]#1.center) {$#2$};
  \node[font=\tiny] at ([xshift=.15cm]#1.center) {$#3$};
}
\newcommand{\dpacidx}[2]{%
  \node[font=\tiny,anchor=north west,inner sep=0pt]
    at ([xshift=.075cm,yshift=-.055cm]#1.south east) {$#2$};%
}
\newcommand{\dpacidxsw}[2]{%
  \node[font=\tiny,anchor=north east,inner sep=0pt]
    at ([xshift=-.075cm,yshift=-.055cm]#1.south west) {$#2$};%
}
\newcommand{\dpacbot}[2]{%
  \node[dpacratio,anchor=north,inner sep=0pt] at ([yshift=-.13cm]#1.south) {$#2$};%
}
\newcommand{\dpacal}[5][.5]{%
  \coordinate (dpac@m) at ($(#2)!#1!(#3)$);
  \node[dpacratio] at ($(dpac@m)!#4!90:(#3)$) {$#5$};%
}
\tikzset{
  dpacnormal/.style={circle,draw=black,fill=white,minimum size=.74cm,line width=.31mm},
  dpacwing/.style={circle,draw=darkgreen,fill=greenback!24,minimum size=.74cm,line width=.36mm},
  dpacinitial/.style={circle,draw=brown,fill=orangeback!12,minimum size=.74cm,line width=.36mm},
  dpacarc/.style={->,line width=.30mm,shorten <=2pt,shorten >=2pt},
  dpacselarc/.style={->,darkgreen,dashed,line width=.42mm,shorten <=2pt,shorten >=2pt},
  dpacinside/.style={->,darkgreen,line width=.38mm,shorten <=2pt,shorten >=2pt},
  dpacratio/.style={text=purple,font=\scriptsize,inner sep=.4pt}
}

\setlength{\abovecaptionskip}{3pt}
\setlength{\belowcaptionskip}{0pt}
\captionsetup{font=small}

The dual variant, denoted here by \texttt{DPaC}, performs the symmetric construction and returns the closure layers in increasing order of their ratios. At each iteration it compares the minimum ratio of the minimal succeeding sets $B_a$ associated with the current arcs with the minimum ratio of the current initial vertices. A strict win by the initial vertices produces the next layer from the bottom; otherwise the selected arc is contracted, so equality is again resolved by contraction.

\noindent\textbf{Step 1 -- contract $(4,2)$.}
Initially, the minimum succeeding-set ratio equals $-2$ and is attained by two arcs, $(4,2)$ and $(6,3)$; the tie is broken arbitrarily and we let $a^*=(4,2)$. The minimum ratio among the initial vertices is $4/3$, attained by vertex 6. Hence \texttt{DPaC} contracts $(4,2)$, and the representative 4 stores $R_4=\{2,4\}$, with $\bar p_4=-3$ and $\bar w_4=2$.

\begin{center}
\begin{tikzpicture}[x=1.62cm,y=1.62cm]
  \node[dpacnormal]  (n1) at (0,2) {}; \dpacsplit{n1}{2}{1}
  \node[dpacwing]    (n2) at (1,2) {}; \dpacsplit{n2}{-2}{1}
  \node[dpacnormal]  (n3) at (0,1) {}; \dpacsplit{n3}{-2}{1}
  \node[dpacwing]    (n4) at (1,1) {}; \dpacsplit{n4}{-1}{1}
  \node[dpacnormal]  (n5) at (2,1) {}; \dpacsplit{n5}{2}{1}
  \node[dpacnormal]  (n6) at (0,0) {}; \dpacsplit{n6}{4}{3}
  \node[dpacnormal]  (n7) at (1,0) {}; \dpacsplit{n7}{5}{1}

  \draw[dpacarc] (n4)--(n1); \draw[dpacselarc] (n4)--(n2);
  \draw[dpacarc] (n6)--(n3); \draw[dpacarc] (n6)--(n4);
  \draw[dpacarc] (n7)--(n4); \draw[dpacarc] (n7)--(n5);

  \dpacal[.40]{n4}{n1}{.30cm}{2}
  \dpacal[.42]{n4}{n2}{-.30cm}{-2}
  \dpacal{n6}{n3}{.30cm}{-2}
  \dpacal{n6}{n4}{-.32cm}{2}
  \dpacal[.45]{n7}{n4}{-.32cm}{\frac34}
  \dpacal{n7}{n5}{-.32cm}{2}

  \dpacbot{n6}{\frac43}\dpacbot{n7}{5}

  \dpacidxsw{n1}{1}\dpacidx{n2}{2}\dpacidx{n3}{3}\dpacidx{n4}{4}
  \dpacidx{n5}{5}\dpacidx{n6}{6}\dpacidx{n7}{7}
\end{tikzpicture}
\captionof*{figure}{\textbf{(a)} Initial dual iteration. The minimum succeeding-set ratio is $-2$, attained by $B_{(4,2)}$ and $B_{(6,3)}$; the minimum initial-vertex ratio is $4/3$.}
\end{center}

\noindent\textbf{Step 2 -- contract $(6,3)$.}
After the first contraction, the minimum succeeding-set ratio is again $-2$, now attained by the only remaining arc that achieves it, $a^*=(6,3)$, while vertex 6 is still the minimum-ratio initial vertex with ratio $4/3$. Thus $(6,3)$ is contracted and representative 6 stores $R_6=\{3,6\}$, with $\bar p_6=2$ and $\bar w_6=4$.

\begin{center}
\begin{tikzpicture}[x=1.62cm,y=1.62cm]
  \node[dpacnormal] (n1) at (0,2) {}; \dpacsplit{n1}{2}{1}
  \node[dpacwing]   (n3) at (0,1) {}; \dpacsplit{n3}{-2}{1}
  \node[dpacnormal] (n4) at (1,1) {}; \dpacsplit{n4}{-3}{2}
  \node[dpacnormal] (n5) at (2,1) {}; \dpacsplit{n5}{2}{1}
  \node[dpacwing]   (n6) at (0,0) {}; \dpacsplit{n6}{4}{3}
  \node[dpacnormal] (n7) at (1,0) {}; \dpacsplit{n7}{5}{1}

  \draw[dpacarc] (n4)--(n1); \draw[dpacselarc] (n6)--(n3);
  \draw[dpacarc] (n6)--(n4); \draw[dpacarc] (n7)--(n4); \draw[dpacarc] (n7)--(n5);

  \dpacal[.40]{n4}{n1}{.30cm}{2}
  \dpacal{n6}{n3}{.30cm}{-2}
  \dpacal{n6}{n4}{-.32cm}{\frac23}
  \dpacal[.45]{n7}{n4}{-.32cm}{\frac15}
  \dpacal{n7}{n5}{-.32cm}{2}

  \dpacbot{n6}{\frac43}\dpacbot{n7}{5}

  \dpacidxsw{n1}{1}\dpacidx{n3}{3}\dpacidx{n4}{4}\dpacidx{n5}{5}\dpacidx{n6}{6}\dpacidx{n7}{7}
\end{tikzpicture}
\captionof*{figure}{\textbf{(b)} After contracting $(4,2)$. The minimum succeeding-set ratio is attained by $(6,3)$, so the second operation is again a contraction.}
\end{center}

\noindent\textbf{Step 3 -- contract $(7,4)$.}
With $R_4=\{2,4\}$ and $R_6=\{3,6\}$, the minimum succeeding-set ratio is $-1/6$, attained by $a^*=(7,4)$, whereas the minimum initial-vertex ratio is now $1/2$, attained by representative 6. Hence $(7,4)$ is contracted and representative 7 stores $R_7=\{2,4,7\}$, with $\bar p_7=2$ and $\bar w_7=3$.

\begin{center}
\begin{tikzpicture}[x=1.62cm,y=1.62cm]
  \node[dpacnormal] (n1) at (0,2) {}; \dpacsplit{n1}{2}{1}
  \node[dpacwing]   (n4) at (1,1) {}; \dpacsplit{n4}{-3}{2}
  \node[dpacnormal] (n5) at (2,1) {}; \dpacsplit{n5}{2}{1}
  \node[dpacwing]   (n6) at (0,0) {}; \dpacsplit{n6}{2}{4}
  \node[dpacwing]   (n7) at (1,0) {}; \dpacsplit{n7}{5}{1}

  \draw[dpacarc] (n4)--(n1); \draw[dpacinside] (n6)--(n4);
  \draw[dpacselarc] (n7)--(n4); \draw[dpacarc] (n7)--(n5);

  \dpacal[.40]{n4}{n1}{.30cm}{2}
  \dpacal{n6}{n4}{-.32cm}{\frac23}
  \dpacal[.45]{n7}{n4}{-.34cm}{-\frac16}
  \dpacal{n7}{n5}{-.32cm}{2}

  \dpacbot{n6}{\frac12}\dpacbot{n7}{5}

  \dpacidxsw{n1}{1}\dpacidx{n4}{4}\dpacidx{n5}{5}\dpacidx{n6}{6}\dpacidx{n7}{7}
\end{tikzpicture}
\captionof*{figure}{\textbf{(c)} After the first two contractions. The selected succeeding set is associated with $(7,4)$ and has ratio $-1/6$.}
\end{center}

\noindent\textbf{Step 4 -- the first layer found by the dual algorithm.}
After contracting $(7,4)$, representatives 1, 5, 6, and 7 remain. The minimum initial-vertex ratio is $1/2$, attained by representative 6, while the minimum succeeding-set ratio is $2/3$. The initial vertex therefore wins strictly and is peeled. Since $R_6=\{3,6\}$, the dual algorithm first identifies the lowest-ratio closure layer
\[
 \mathcal I_3=\{3,6\},\qquad \lambda_3=\frac12.
\]

\begin{center}
\begin{tikzpicture}[x=1.62cm,y=1.62cm]
  \node[dpacnormal]  (n1) at (0,2) {}; \dpacsplit{n1}{2}{1}
  \node[dpacnormal]  (n5) at (2,2) {}; \dpacsplit{n5}{2}{1}
  \node[dpacinitial] (n6) at (0,0) {}; \dpacsplit{n6}{2}{4}
  \node[dpacnormal]  (n7) at (1,1) {}; \dpacsplit{n7}{2}{3}

  \draw[dpacarc] (n7)--(n1); \draw[dpacarc] (n7)--(n5); \draw[dpacarc] (n6)--(n7);

  \dpacal[.40]{n7}{n1}{.30cm}{2}
  \dpacal{n7}{n5}{-.32cm}{2}
  \dpacal{n6}{n7}{-.32cm}{\frac23}

  \dpacbot{n6}{\frac12}

  \dpacidxsw{n1}{1}\dpacidx{n5}{5}\dpacidx{n6}{6}\dpacidx{n7}{7}
\end{tikzpicture}
\captionof*{figure}{\textbf{(d)} The initial representative 6 has ratio $1/2$, below the best succeeding-set ratio $2/3$, and yields $\mathcal I_3$.}
\end{center}

\noindent\textbf{Step 5 -- the second layer.}
After removing representative 6, the minimum initial-vertex ratio is $2/3$, attained by representative 7, while the minimum succeeding-set ratio is $2$. Therefore representative 7 is peeled and yields
\[
 \mathcal I_2=R_7=\{2,4,7\},\qquad \lambda_2=\frac23.
\]

\begin{center}
\begin{tikzpicture}[x=1.62cm,y=1.62cm]
  \node[dpacnormal]  (n1) at (0,2) {}; \dpacsplit{n1}{2}{1}
  \node[dpacnormal]  (n5) at (2,2) {}; \dpacsplit{n5}{2}{1}
  \node[dpacinitial] (n7) at (1,1) {}; \dpacsplit{n7}{2}{3}

  \draw[dpacarc] (n7)--(n1); \draw[dpacarc] (n7)--(n5);

  \dpacal[.40]{n7}{n1}{.30cm}{2}
  \dpacal{n7}{n5}{-.32cm}{2}

  \dpacbot{n7}{\frac23}

  \dpacidxsw{n1}{1}\dpacidx{n5}{5}\dpacidx{n7}{7}
\end{tikzpicture}
\captionof*{figure}{\textbf{(e)} Representative 7 is now the minimum-ratio initial vertex and yields $\mathcal I_2$.}
\end{center}

\noindent\textbf{Step 6 -- the last layer.}
Only vertices 1 and 5 remain. Both are initial vertices, both have ratio $2$, and there are no arcs. Hence the last layer returned by the dual run is
\[
 \mathcal I_1=\{1,5\},\qquad \lambda_1=2.
\]
Thus \texttt{DPaC} has recovered the same partition as \texttt{PaC}, but in the order $\mathcal I_3,\mathcal I_2,\mathcal I_1$.

\begin{center}
\begin{tikzpicture}[x=1.62cm,y=1.62cm]
  \node[dpacinitial] (n1) at (0,0) {}; \dpacsplit{n1}{2}{1}
  \node[dpacinitial] (n5) at (1.4,0) {}; \dpacsplit{n5}{2}{1}
  \dpacbot{n1}{2}\dpacbot{n5}{2}
  \dpacidx{n1}{1}\dpacidx{n5}{5}
\end{tikzpicture}
\captionof*{figure}{\textbf{(f)} Final dual iteration. Vertices 1 and 5 yield $\mathcal I_1$.}
\end{center}

\section{Execution of \texttt{RaC}}
\label{app:ec_rac_execution}
\usetikzlibrary{positioning}

\newcommand{\racsplitpwbig}[3]{%
  \draw[line width=.2mm] (#1.north) -- (#1.south);
  \node[font=\tiny] at ([xshift=-.15cm]#1.center) {$#2$};
  \node[font=\tiny] at ([xshift=.15cm]#1.center) {$#3$};
}
\newcommand{\racidx}[2]{%
  \node[font=\tiny,anchor=north west,inner sep=0pt]
    at ([xshift=.075cm,yshift=-.055cm]#1.south east) {$#2$};%
}
\newcommand{\racidxsw}[2]{%
  \node[font=\tiny,anchor=north east,inner sep=0pt]
    at ([xshift=-.075cm,yshift=-.055cm]#1.south west) {$#2$};%
}
\newcommand{\racclabel}[2]{%
  \node[font=\tiny,anchor=north,inner sep=0.4pt]
    at ([yshift=-.4mm]#1.south east) {$#2$};%
}
\tikzset{
  racnormalbig/.style={circle,draw=black,fill=white,minimum size=.74cm,line width=.31mm},
  racselectedbig/.style={circle,draw=blue,fill=blueback!32,minimum size=.74cm,line width=.35mm},
  racabsorbedbig/.style={circle,draw=black,densely dashed,fill=white,minimum size=.74cm,line width=.31mm,text=black},
  racclusterbig/.style={draw=black,densely dashed,rounded corners=.22cm,line width=.40mm},
  racboxbig/.style={draw=violet,fill=yellowback!20,line width=.50mm,rounded corners=.18cm,align=center,inner sep=1.6mm},
  racbox2/.style={draw=orange,fill=yellowback!20,line width=.45mm,rounded corners=.10cm,align=center,inner sep=1.1mm},
  racbox1/.style={draw=darkgreen,fill=greenback!45,line width=.45mm,rounded corners=.10cm,align=center,inner sep=1.1mm},
  racarrowbig/.style={->,line width=.40mm},
  racbluearrowbig/.style={->,blue,line width=.46mm},
  racvioletarrowbig/.style={->,violet,line width=.46mm},
  racblueline/.style={blue,line width=.42mm},
  racvioletline/.style={violet,line width=.42mm}
}

\setlength{\abovecaptionskip}{3pt}
\setlength{\belowcaptionskip}{0pt}
\captionsetup{font=small}

Panels (a)--(n) show the complete mechanics of the Rake-and-Compress algorithm step by step. The original precedence graph is directed. Only when deciding which vertices are leaves, which have degree two, and which vertices can be contracted simultaneously, \texttt{RaC} works on the underlying undirected tree. The original arc orientations are retained in the feasibility conditions and in all cluster functions $f_{\mathcal C}^{\sigma}$.

\noindent\textbf{Step 0 -- original directed instance and initialization of elementary clusters.}
The running instance has the six precedence arcs
$(4,1),(4,2),(6,3),(6,4),(7,4),(7,5)$.
At initialization, every arc $(i,j)$ creates an elementary 2-cluster $({\{i,j\}},\{i,j\})$ with empty interior. Because the interior is empty, its finite cluster functions are initially zero; the boundary state that violates the directed precedence constraint is assigned value $-\infty$. These elementary 2-clusters are the leaves of the cluster tree that will be built bottom-up.

\begin{center}
\begin{tikzpicture}[x=1.55cm,y=1.43cm]
  \node[racnormalbig] (n1) at (0,2) {}; \racsplitpwbig{n1}{2}{1}
  \node[racnormalbig] (n2) at (1,2) {}; \racsplitpwbig{n2}{-2}{1}
  \node[racnormalbig] (n3) at (0,1) {}; \racsplitpwbig{n3}{-2}{1}
  \node[racnormalbig] (n4) at (1,1) {}; \racsplitpwbig{n4}{-1}{1}
  \node[racnormalbig] (n5) at (2,1) {}; \racsplitpwbig{n5}{2}{1}
  \node[racnormalbig] (n6) at (0,0) {}; \racsplitpwbig{n6}{4}{3}
  \node[racnormalbig] (n7) at (1,0) {}; \racsplitpwbig{n7}{5}{1}

  \draw[->,line width=.30mm,shorten <=2pt,shorten >=2pt] (n4)--(n1);
  \draw[->,line width=.30mm,shorten <=2pt,shorten >=2pt] (n4)--(n2);
  \draw[->,line width=.30mm,shorten <=2pt,shorten >=2pt] (n6)--(n3);
  \draw[->,line width=.30mm,shorten <=2pt,shorten >=2pt] (n6)--(n4);
  \draw[->,line width=.30mm,shorten <=2pt,shorten >=2pt] (n7)--(n4);
  \draw[->,line width=.30mm,shorten <=2pt,shorten >=2pt] (n7)--(n5);

  \racidxsw{n1}{1}\racidx{n2}{2}\racidx{n3}{3}\racidx{n4}{4}
  \racidx{n5}{5}\racidx{n6}{6}\racidx{n7}{7}

\end{tikzpicture}

\resizebox{\textwidth}{!}{%
\begin{tikzpicture}[font=\scriptsize]
  \node[racbox2] (e64) {$\partial\mathcal C=\{6,4\}$\\$\operatorname{int}(C)=\emptyset$};
  \racclabel{e64}{\mathcal C_{\{6,4\}}}
  \node[racbox2,right=5mm of e64] (e41) {$\partial\mathcal C=\{4,1\}$\\$\operatorname{int}(C)=\emptyset$};
  \racclabel{e41}{\mathcal C_{\{4,1\}}}
  \node[racbox2,right=5mm of e41] (e42) {$\partial\mathcal C=\{4,2\}$\\$\operatorname{int}(C)=\emptyset$};
  \racclabel{e42}{\mathcal C_{\{4,2\}}}
  \node[racbox2,right=5mm of e42] (e63) {$\partial\mathcal C=\{6,3\}$\\$\operatorname{int}(C)=\emptyset$};
  \racclabel{e63}{\mathcal C_{\{6,3\}}}
  \node[racbox2,right=5mm of e63] (e74) {$\partial\mathcal C=\{7,4\}$\\$\operatorname{int}(C)=\emptyset$};
  \racclabel{e74}{\mathcal C_{\{7,4\}}}
  \node[racbox2,right=5mm of e74] (e75) {$\partial\mathcal C=\{7,5\}$\\$\operatorname{int}(C)=\emptyset$};
  \racclabel{e75}{\mathcal C_{\{7,5\}}}
\end{tikzpicture}}

\captionof*{figure}{\textbf{(a)} Original directed running instance and elementary-cluster initialization: one 2-cluster per arc, with empty interior, and no arcs between them yet in the cluster graph.}
\end{center}

\noindent
For example, the elementary cluster associated with arc $(4,1)$ has boundary $\{4,1\}$ and empty interior. The state $x_4=1,x_1=0$ is infeasible because the precedence constraint requires $x_4\le x_1$. The same initialization is performed independently for all six arcs.

\noindent\textbf{Step 1 -- choose the contraction set in the first round.}
For the contraction logic, arc directions are now ignored. In the current tree $T'$, the leaves are
$\mathcal L_{T'}=\{1,2,3,5\}$ and the degree-two vertices are $\mathcal D^2_{T'}=\{6,7\}$.
The algorithm chooses a maximal independent set $I_1$ satisfying
$\mathcal L_{T'}\subseteq I_1\subseteq \mathcal L_{T'}\cup\mathcal D^2_{T'}$.
All four leaves must therefore be selected. Vertex 6 cannot be added because it is adjacent to selected leaf 3, and vertex 7 cannot be added because it is adjacent to selected leaf 5. Hence the maximal independent set is forced:
$I_1=\{1,2,3,5\}$.

\begin{center}
\begin{tikzpicture}[x=1.55cm,y=1.43cm]
  \node[racselectedbig] (n1) at (0,2) {}; \racsplitpwbig{n1}{2}{1}
  \node[racselectedbig] (n2) at (1,2) {}; \racsplitpwbig{n2}{-2}{1}
  \node[racselectedbig] (n3) at (0,1) {}; \racsplitpwbig{n3}{-2}{1}
  \node[racnormalbig]   (n4) at (1,1) {}; \racsplitpwbig{n4}{-1}{1}
  \node[racselectedbig] (n5) at (2,1) {}; \racsplitpwbig{n5}{2}{1}
  \node[racnormalbig]   (n6) at (0,0) {}; \racsplitpwbig{n6}{4}{3}
  \node[racnormalbig]   (n7) at (1,0) {}; \racsplitpwbig{n7}{5}{1}

  \draw[line width=.39mm] (n4)--(n1);
  \draw[line width=.39mm] (n4)--(n2);
  \draw[line width=.39mm] (n6)--(n3);
  \draw[line width=.39mm] (n6)--(n4);
  \draw[line width=.39mm] (n7)--(n4);
  \draw[line width=.39mm] (n7)--(n5);

  \racidxsw{n1}{1}\racidx{n2}{2}\racidx{n3}{3}\racidx{n4}{4}
  \racidx{n5}{5}\racidx{n6}{6}\racidx{n7}{7}

\end{tikzpicture}
\captionof*{figure}{\textbf{(b)} First-round maximal independent set $I_1=\{1,2,3,5\}$.}
\end{center}

\noindent\textbf{Step 2 -- apply \texttt{Compress1} to the four selected leaves.}
Every vertex in $I_1$ is a leaf, hence every contraction in this round is a \texttt{Compress1}. If a leaf $i$ belongs to a 2-cluster whose other boundary vertex is $j$, \texttt{Compress1} moves $i$ from the boundary to the interior and leaves $j$ as the unique boundary vertex. At the level of functions, the state of $i$ is eliminated by taking the maximum over $\sigma_i\in\{0,1\}$ and adding the contribution $\sigma_i(p_i-\lambda w_i)$.

In the running instance this creates four 1-clusters: one with interior $\{1\}$ and boundary $\{4\}$, one with interior $\{2\}$ and boundary $\{4\}$, one with interior $\{3\}$ and boundary $\{6\}$, and one with interior $\{5\}$ and boundary $\{7\}$. Vertices 1,2,3,5 are therefore removed from the active tree. The active tree becomes the path $6-4-7$.

\begin{center}
\resizebox{\textwidth}{!}{%
\begin{tikzpicture}[font=\scriptsize]
  \node[racbox2] (c64) {$\partial\mathcal C=\{6,4\}$\\$\operatorname{int}(C)=\emptyset$};
  \racclabel{c64}{\mathcal C_{\{6,4\}}}
  \node[racbox1,right=5mm of c64] (c41p) {$\partial\mathcal C=\{4\}$\\$\operatorname{int}(C)=\{1\}$};
  \racclabel{c41p}{\mathcal C'_{\{4,1\}}}
  \node[racbox1,right=5mm of c41p] (c42p) {$\partial\mathcal C=\{4\}$\\$\operatorname{int}(C)=\{2\}$};
  \racclabel{c42p}{\mathcal C'_{\{4,2\}}}
  \node[racbox1,right=5mm of c42p] (c63p) {$\partial\mathcal C=\{6\}$\\$\operatorname{int}(C)=\{3\}$};
  \racclabel{c63p}{\mathcal C'_{\{6,3\}}}
  \node[racbox2,right=5mm of c63p] (c74) {$\partial\mathcal C=\{7,4\}$\\$\operatorname{int}(C)=\emptyset$};
  \racclabel{c74}{\mathcal C_{\{7,4\}}}
  \node[racbox1,right=5mm of c74] (c75p) {$\partial\mathcal C=\{7\}$\\$\operatorname{int}(C)=\{5\}$};
  \racclabel{c75p}{\mathcal C'_{\{7,5\}}}

  \node[racbox2,below=1.5cm of c41p] (c41) {$\partial\mathcal C=\{4,1\}$\\$\operatorname{int}(C)=\emptyset$};
  \racclabel{c41}{\mathcal C_{\{4,1\}}}
  \node[racbox2,below=1.5cm of c42p] (c42) {$\partial\mathcal C=\{4,2\}$\\$\operatorname{int}(C)=\emptyset$};
  \racclabel{c42}{\mathcal C_{\{4,2\}}}
  \node[racbox2,below=1.5cm of c63p] (c63) {$\partial\mathcal C=\{6,3\}$\\$\operatorname{int}(C)=\emptyset$};
  \racclabel{c63}{\mathcal C_{\{6,3\}}}
  \node[racbox2,below=1.5cm of c75p] (c75) {$\partial\mathcal C=\{7,5\}$\\$\operatorname{int}(C)=\emptyset$};
  \racclabel{c75}{\mathcal C_{\{7,5\}}}

  \draw[racblueline] (c41p.south)--(c41.north);
  \draw[racblueline] (c42p.south)--(c42.north);
  \draw[racblueline] (c63p.south)--(c63.north);
  \draw[racblueline] (c75p.south)--(c75.north);

  \coordinate (topmid) at ($(c64)!0.5!(c75p)$);
  \node[racnormalbig,above=9mm of topmid] (a4) {}; \racsplitpwbig{a4}{-1}{1}
  \node[racnormalbig,left=12mm of a4] (a6) {}; \racsplitpwbig{a6}{4}{3}
  \node[racnormalbig,right=12mm of a4] (a7) {}; \racsplitpwbig{a7}{5}{1}
  \draw[line width=.30mm] (a6)--(a4)--(a7);
  \racidx{a6}{6}\racidx{a4}{4}\racidx{a7}{7}
\end{tikzpicture}}
\captionof*{figure}{\textbf{(c)} Active tree, and cluster graph after the four first-round \texttt{Compress1} operations. Each new 1-cluster (green box) is connected by a blue line to the elementary 2-cluster (yellow box) it was compressed from, with the leaf moved from boundary to interior and a prime added to the cluster's name; $\mathcal C_{\{6,4\}}$ and $\mathcal C_{\{7,4\}}$ are untouched this round and simply carry over to the next step.}
\end{center}

\noindent\textbf{Step 3 -- rake the four 1-clusters into the surviving edge clusters.}
After the compressions, the active tree is $6-4-7$, but the information associated with the removed leaves is stored in the four 1-clusters. A \texttt{Rake} operation attaches such 1-clusters to a surviving cluster sharing the same boundary vertex. At vertex 4 there are two incident surviving 2-clusters; the algorithm may choose either one when assigning the 1-clusters with boundary $\{4\}$. We follow the same legal choice as in the manuscript illustration: the clusters containing vertices 1 and 2 are raked into the $\{6,4\}$ side.

On the $\{6,4\}$ side this is actually carried out as two \texttt{Rake} operations in sequence, one for each boundary vertex that still holds pending 1-clusters. The first \texttt{Rake} is triggered by vertex $4$, which is the boundary of the two singleton-boundary clusters $\mathcal C'_{\{4,1\}}$ and $\mathcal C'_{\{4,2\}}$: it attaches both of them to $\mathcal C_{\{6,4\}}$, producing the intermediate 2-cluster $\mathcal C''_{\{6,4\}}$ with boundary $\{6,4\}$ and interior $\{1,2\}$. The second \texttt{Rake} is then triggered by vertex $6$, which is the boundary of the singleton-boundary cluster $\mathcal C'_{\{6,3\}}$: it attaches $\mathcal C'_{\{6,3\}}$ to $\mathcal C''_{\{6,4\}}$, only at this point producing the 2-cluster $\mathcal D_{\{6,4\}}$ of the next figure, with boundary $\{6,4\}$ and interior $\{1,2,3\}$. On the $\{4,7\}$ side a single \texttt{Rake}, triggered by vertex $7$, attaches $\mathcal C'_{\{7,5\}}$ directly to $\mathcal C_{\{7,4\}}$.

The result is a 2-cluster $\mathcal D_{\{6,4\}}$ with boundary $\{6,4\}$ and interior $\{1,2,3\}$, and a 2-cluster $\mathcal D_{\{4,7\}}$ with boundary $\{4,7\}$ and interior $\{5\}$. Importantly, the active tree is still the path $6-4-7$; \texttt{Rake} changes cluster information but does not remove an active vertex.

\begin{center}
\begin{tikzpicture}[x=1.1cm,y=1.1cm]
  \begin{scope}
    \node[racnormalbig]   (a6) at (0,0) {}; \racsplitpwbig{a6}{4}{3}
    \node[racnormalbig]   (a4) at (1,1) {}; \racsplitpwbig{a4}{-1}{1}
    \node[racabsorbedbig] (a3) at (0,1) {}; \racsplitpwbig{a3}{-2}{1}
    \node[racabsorbedbig] (a1) at (0,2) {}; \racsplitpwbig{a1}{2}{1}
    \node[racabsorbedbig] (a2) at (1,2) {}; \racsplitpwbig{a2}{-2}{1}
    \draw[line width=.36mm] (a6)--(a4);
    \draw[black,densely dashed,line width=.28mm] (a3)--(a6);
    \draw[black,densely dashed,line width=.28mm] (a1)--(a4);
    \draw[black,densely dashed,line width=.28mm] (a2)--(a4);
    \node[racclusterbig,fit=(a6)(a4)(a3)(a1)(a2),inner sep=4.8mm] (boxA) {};
    \racidx{a6}{6}\racidx{a4}{4}\racidx{a3}{3}\racidxsw{a1}{1}\racidx{a2}{2}
    \node[font=\scriptsize,text=black,fill=white,inner sep=.7pt] at ([yshift=2.0mm]boxA.north) {$\mathcal D_{\{6,4\}}$};
  \end{scope}

  \begin{scope}[shift={(4.3,0)}]
    \node[racnormalbig]   (b4) at (1,1) {}; \racsplitpwbig{b4}{-1}{1}
    \node[racnormalbig]   (b7) at (1,0) {}; \racsplitpwbig{b7}{5}{1}
    \node[racabsorbedbig] (b5) at (2,1) {}; \racsplitpwbig{b5}{2}{1}
    \draw[line width=.36mm] (b4)--(b7);
    \draw[black,densely dashed,line width=.28mm] (b5)--(b7);
    \node[racclusterbig,fit=(b4)(b7)(b5),inner sep=4.8mm] (boxB) {};
    \racidx{b4}{4}\racidx{b7}{7}\racidx{b5}{5}
    \node[font=\scriptsize,text=black,fill=white,inner sep=.7pt] at ([yshift=2.0mm]boxB.north) {$\mathcal D_{\{4,7\}}$};
  \end{scope}
\end{tikzpicture}
\captionof*{figure}{\textbf{(d)} The two 2-clusters obtained after the first \texttt{Rake}, drawn as sub-trees of the original instance.}
\end{center}

\begin{center}
\resizebox{\textwidth}{!}{%
\begin{tikzpicture}[font=\scriptsize]
  \node[racbox2] (e41) {$\partial\mathcal C=\{4,1\}$\\$\operatorname{int}(C)=\emptyset$};
  \racclabel{e41}{\mathcal C_{\{4,1\}}}
  \node[racbox2,right=5mm of e41] (e42) {$\partial\mathcal C=\{4,2\}$\\$\operatorname{int}(C)=\emptyset$};
  \racclabel{e42}{\mathcal C_{\{4,2\}}}

  \coordinate (c41) at (e41);
  \node[racbox1,above=1.45cm of c41] (e41p) {$\partial\mathcal C=\{4\}$\\$\operatorname{int}(C)=\{1\}$};
  \racclabel{e41p}{\mathcal C'_{\{4,1\}}}
  \coordinate (c42) at (e42);
  \node[racbox1,above=1.45cm of c42] (e42p) {$\partial\mathcal C=\{4\}$\\$\operatorname{int}(C)=\{2\}$};
  \racclabel{e42p}{\mathcal C'_{\{4,2\}}}
  \node[racbox2,right=5mm of e42p] (e64) {$\partial\mathcal C=\{6,4\}$\\$\operatorname{int}(C)=\emptyset$};
  \racclabel{e64}{\mathcal C_{\{6,4\}}}
  \node[racbox2,right=5mm of e64] (e63a) {$\partial\mathcal C=\{6,3\}$\\$\operatorname{int}(C)=\emptyset$};
  \racclabel{e63a}{\mathcal C_{\{6,3\}}}
  \node[racbox2,right=3.2cm of e63a] (e75a) {$\partial\mathcal C=\{7,5\}$\\$\operatorname{int}(C)=\emptyset$};
  \racclabel{e75a}{\mathcal C_{\{7,5\}}}

  \coordinate (m64pp) at ($(e41p)!0.5!(e64)$);
  \node[racbox2,above=1.45cm of m64pp] (e64pp) {$\partial\mathcal C=\{6,4\}$\\$\operatorname{int}(C)=\{1,2\}$};
  \racclabel{e64pp}{\mathcal C''_{\{6,4\}}}
  \coordinate (c63a) at (e63a);
  \node[racbox1,above=1.45cm of c63a] (e63p) {$\partial\mathcal C=\{6\}$\\$\operatorname{int}(C)=\{3\}$};
  \racclabel{e63p}{\mathcal C'_{\{6,3\}}}
  \node[racbox2,right=5mm of e63p] (e74b) {$\partial\mathcal C=\{7,4\}$\\$\operatorname{int}(C)=\emptyset$};
  \racclabel{e74b}{\mathcal C_{\{7,4\}}}
  \coordinate (c75a) at (e75a);
  \node[racbox1,above=1.45cm of c75a] (e75p) {$\partial\mathcal C=\{7\}$\\$\operatorname{int}(C)=\{5\}$};
  \racclabel{e75p}{\mathcal C'_{\{7,5\}}}

  \coordinate (mA) at ($(e64pp)!0.5!(e63p)$);
  \node[racbox2,above=1.45cm of mA] (eA) {$\partial\mathcal C=\{6,4\}$\\$\operatorname{int}(C)=\{1,2,3\}$};
  \racclabel{eA}{\mathcal D_{\{6,4\}}}
  \coordinate (mB) at ($(e74b)!0.5!(e75p)$);
  \node[racbox2,above=1.45cm of mB] (eB) {$\partial\mathcal C=\{4,7\}$\\$\operatorname{int}(C)=\{5\}$};
  \racclabel{eB}{\mathcal D_{\{4,7\}}}

  \draw[racblueline] (e41p.south)--(e41.north);
  \draw[racblueline] (e42p.south)--(e42.north);
  \draw[racblueline] (e63p.south)--(e63a.north);
  \draw[racblueline] (e75p.south)--(e75a.north);

  \draw[racvioletline] (e64pp.south)--(e41p.north);
  \draw[racvioletline] (e64pp.south)--(e42p.north);
  \draw[racvioletline] (e64pp.south)--(e64.north);
  \draw[racvioletline] (eA.south)--(e64pp.north);
  \draw[racvioletline] (eA.south)--(e63p.north);
  \draw[racvioletline] (eB.south)--(e74b.north);
  \draw[racvioletline] (eB.south)--(e75p.north);
\end{tikzpicture}}
\captionof*{figure}{\textbf{(e)} Cluster graph completing panel (c) with the first-round \texttt{Rake} operations (violet lines), still with \texttt{Compress1} links (blue lines) from panel (c).}
\end{center}

This concludes the first iteration of the main cycle in lines \ref{ln:startwhile_bottomup}--\ref{ln:endwhile_bottomup} of \cref{alg:bottomup}. The next iteration begins with the active tree $6-4-7$ and the two 2-clusters $\mathcal D_{\{6,4\}}$ and $\mathcal D_{\{4,7\}}$ at the top of Figure (e).

\noindent\textbf{Step 4 -- choose the contraction set in the second round.}
The active tree is now the three-vertex path $6-4-7$. Its leaves are 6 and 7, while vertex 4 has degree two. Since the maximal independent set must contain all leaves, both 6 and 7 are selected. Vertex 4 cannot then be added because it is adjacent to both selected leaves. Consequently
$I_2=\{6,7\}$.

This also explains why \texttt{Compress2} is never executed in this particular running example. In a general round, a degree-two vertex belonging to the maximal independent set would be processed by \texttt{Compress2}, combining its two incident 2-clusters and moving that degree-two vertex to the interior. Here, no degree-two vertex is ever selected.

\begin{center}
\begin{tikzpicture}[x=2.00cm,y=1.50cm]
  \node[racselectedbig] (s6) at (0,0) {}; \racsplitpwbig{s6}{4}{3}
  \node[racnormalbig]   (s4) at (1.55,0) {}; \racsplitpwbig{s4}{-1}{1}
  \node[racselectedbig] (s7) at (3.10,0) {}; \racsplitpwbig{s7}{5}{1}
  \draw[line width=.40mm] (s6)--(s4)--(s7);
  \racidx{s6}{6}\racidx{s4}{4}\racidx{s7}{7}
\end{tikzpicture}
\captionof*{figure}{\textbf{(f)} Second-round maximal independent set $I_2=\{6,7\}$.}
\end{center}

\noindent\textbf{Step 5 -- apply \texttt{Compress1} to vertices 6 and 7.}
The two selected vertices are leaves of the active path, so each is processed by \texttt{Compress1}. Compressing 6 in $\mathcal D_{\{6,4\}}$ produces the 1-cluster $\mathcal D'_{\{6,4\}}$ with boundary $\{4\}$ and interior $\{1,2,3,6\}$. Compressing 7 in $\mathcal D_{\{4,7\}}$ produces the 1-cluster $\mathcal D'_{\{4,7\}}$ with boundary $\{4\}$ and interior $\{5,7\}$. Only vertex 4 remains active.

\begin{center}
\begin{tikzpicture}[font=\scriptsize]
  \node[racbox2] (D64) {$\partial\mathcal C=\{6,4\}$\\$\operatorname{int}(C)=\{1,2,3\}$};
  \racclabel{D64}{\mathcal D_{\{6,4\}}}
  \node[racbox2,right=2.6cm of D64] (D47) {$\partial\mathcal C=\{4,7\}$\\$\operatorname{int}(C)=\{5\}$};
  \racclabel{D47}{\mathcal D_{\{4,7\}}}
  \coordinate (cD64) at (D64);
  \node[racbox1,above=1.45cm of cD64] (D64p) {$\partial\mathcal C=\{4\}$\\$\operatorname{int}(C)=\{1,2,3,6\}$};
  \racclabel{D64p}{\mathcal D'_{\{6,4\}}}
  \coordinate (cD47) at (D47);
  \node[racbox1,above=1.45cm of cD47] (D47p) {$\partial\mathcal C=\{4\}$\\$\operatorname{int}(C)=\{5,7\}$};
  \racclabel{D47p}{\mathcal D'_{\{4,7\}}}
  \draw[racblueline] (D64p.south)--(D64.north);
  \draw[racblueline] (D47p.south)--(D47.north);

  \coordinate (mv4) at ($(D64p)!0.5!(D47p)$);
  \node[racnormalbig,above=9mm of mv4] (v4) {}; \racsplitpwbig{v4}{-1}{1}
  \racidx{v4}{4}
\end{tikzpicture}
\captionof*{figure}{\textbf{(g)} The two second-round \texttt{Compress1} operations result in the formation of the 1-clusters $\mathcal D'_{\{6,4\}}$ and $\mathcal D'_{\{4,7\}}$ each connected to the 2-cluster it was compressed from. The active tree consists now of the single vertex 4.}
\end{center}

\noindent\textbf{Step 6 -- form the root of the cluster tree and then compute $\varphi$.}
Both $\mathcal D'_{\{6,4\}}$ and $\mathcal D'_{\{4,7\}}$ are 1-clusters with the same singleton boundary $\{4\}$, so this is precisely the special case at the end of \cref{def:rake}. Applying this special \texttt{Rake} merges $\mathcal D'_{\{6,4\}}$ and $\mathcal D'_{\{4,7\}}$ into the 1-cluster $\mathcal E_4=(V,\{4\})$, the root of the cluster tree $\mathcal G=({\mathcal V},{\mathcal A})$ stored by the bottom-up phase.

This concludes the second and last iteration of the main cycle in lines \ref{ln:startwhile_bottomup}--\ref{ln:endwhile_bottomup} of \cref{alg:bottomup}: now $|V'|=1$, the {\tt while} loop terminates, and execution moves to line~\ref{ln:finalcompress_bottomup}, where one final \texttt{Compress1} eliminates the last boundary vertex 4 and takes $\mathcal E_4$ to
$\mathcal C^\star=(V,\emptyset)$, with
$\varphi(\lambda)=f_{\mathcal C^\star}^{\emptyset}(\lambda)$. It is important to distinguish the root $\mathcal E_4$, which is stored as part of $\mathcal G$, from the boundary-free cluster $\mathcal C^\star$, which only serves to obtain $\varphi$.

\begin{center}
\begin{tikzpicture}[font=\scriptsize]
  \node[racbox2] (D64) {$\partial\mathcal C=\{6,4\}$\\$\operatorname{int}(C)=\{1,2,3\}$};
  \racclabel{D64}{\mathcal D_{\{6,4\}}}
  \node[racbox2,right=2.6cm of D64] (D47) {$\partial\mathcal C=\{4,7\}$\\$\operatorname{int}(C)=\{5\}$};
  \racclabel{D47}{\mathcal D_{\{4,7\}}}
  \coordinate (cD64) at (D64);
  \node[racbox1,above=1.45cm of cD64] (Dp64) {$\partial\mathcal C=\{4\}$\\$\operatorname{int}(C)=\{1,2,3,6\}$};
  \racclabel{Dp64}{\mathcal D'_{\{6,4\}}}
  \coordinate (cD47) at (D47);
  \node[racbox1,above=1.45cm of cD47] (Dp47) {$\partial\mathcal C=\{4\}$\\$\operatorname{int}(C)=\{5,7\}$};
  \racclabel{Dp47}{\mathcal D'_{\{4,7\}}}
  \draw[racblueline] (Dp64.south)--(D64.north);
  \draw[racblueline] (Dp47.south)--(D47.north);
  \coordinate (mE4) at ($(Dp64)!0.5!(Dp47)$);
  \node[racbox1,above=1.45cm of mE4] (E4) {$\partial\mathcal C=\{4\}$\\$\operatorname{int}(C)=\{1,2,3,5,6,7\}$};
  \racclabel{E4}{\mathcal E_4}
  \coordinate (cE4) at (E4);
  \node[racbox2,above=1.45cm of cE4] (Cstar) {$\partial\mathcal C=\emptyset$\\$\operatorname{int}(C)=\{1,2,\ldots,7\}$};
  \racclabel{Cstar}{\mathcal C^\star}
  \draw[racvioletline] (E4.south)--(Dp64.north);
  \draw[racvioletline] (E4.south)--(Dp47.north);
  \draw[racblueline] (Cstar.south)--(E4.north);
\end{tikzpicture}
\captionof*{figure}{\textbf{(h)} Panel (g) completed with the special \texttt{Rake} of \cref{def:rake} merging $\mathcal D'_{\{6,4\}}$ and $\mathcal D'_{\{4,7\}}$ into the root $\mathcal E_4=(V,\{4\})$, immediately followed by the final \texttt{Compress1} that eliminates boundary vertex $4$ and yields $\mathcal C^\star=(V,\emptyset)$, ending the bottom-up phase.}
\end{center}

\noindent\textbf{Step 7 -- start the top-down phase and recover the threshold of vertex 4.}
The top-down phase starts from $\mathcal E_4=(V,\{4\})$, not from $\mathcal C^\star$. The two possible states of the root boundary vertex define
\[
 g_0^4(\lambda)=f_{\mathcal E_4}^{0}(\lambda),\qquad
 g_1^4(\lambda)=f_{\mathcal E_4}^{1}(\lambda)+(p_4-\lambda w_4).
\]
The threshold $\theta_4$ is the unique value at which $g_1^4(\lambda)-g_0^4(\lambda)$ changes sign. For this instance,
$\theta_4=2/3$.

\noindent\textbf{Step 8 -- traverse the special-rake arcs.}
The root $\mathcal E_4$ was created by a \texttt{Rake}. Rake arcs do not correspond to a vertex being moved into the interior of a cluster, so the top-down phase does not compute a new threshold while traversing them. It simply descends from $\mathcal E_4$ to its two children $\mathcal D'_{\{6,4\}}$ and $\mathcal D'_{\{4,7\}}$. The already known threshold $\theta_4=2/3$ remains available as a boundary condition for the next computations.

\noindent\textbf{Step 9 -- reverse the second-round \texttt{Compress1} operations.}
Cluster $\mathcal D'_{\{6,4\}}$ was obtained by compressing vertex 6 from $\mathcal D_{\{6,4\}}$, while $\mathcal D'_{\{4,7\}}$ was obtained by compressing vertex 7 from $\mathcal D_{\{4,7\}}$. When a \texttt{Compress1} arc is traversed top-down, the threshold of the surviving boundary vertex is already known. For vertex 6 this known boundary is vertex 4, so the state $\sigma_4(\lambda)$ is fixed and the algorithm compares
\[
 g_x^6(\lambda)=f_{\mathcal D_{\{6,4\}}}^{\sigma_4(\lambda),x}(\lambda)
                 +x(p_6-\lambda w_6),\qquad x\in\{0,1\}.
\]
The sign change gives $\theta_6=1/2$. The analogous computation for vertex 7 gives $\theta_7=2/3$.

\noindent\textbf{Step 10 -- continue downward to the first-round clusters: vertices 1 and 2.}
Inside $\mathcal D_{\{6,4\}}$, rake arcs are first traversed without computation until the two 1-clusters created by \texttt{Compress1}(1) and \texttt{Compress1}(2) are reached. In both cases the surviving boundary vertex is 4, whose threshold is already known. Reversing the corresponding \texttt{Compress1} arcs and comparing the two state functions yields
$\theta_1=2$ and $\theta_2=2/3$.

\noindent\textbf{Step 11 -- recover the thresholds of vertices 3 and 5.}
The remaining first-round 1-clusters have boundaries $\{6\}$ and $\{7\}$. Their boundary thresholds have already been recovered in Step 9. Reversing \texttt{Compress1}(3) with the state of vertex 6 fixed by $\theta_6=1/2$ gives $\theta_3=1/2$. Reversing \texttt{Compress1}(5) with the state of vertex 7 fixed by $\theta_7=2/3$ gives $\theta_5=2$.

\noindent\textbf{Step 12 -- group equal thresholds to obtain the closure layers.}
At this point a threshold has been recovered for every vertex. The values are
\[
\theta_1=2,\quad
\theta_2=\frac23,\quad
\theta_3=\frac12,\quad
\theta_4=\frac23,\quad
\theta_5=2,\quad
\theta_6=\frac12,\quad
\theta_7=\frac23.
\]
Only now are vertices grouped into closure layers. Equal thresholds define the same layer, and the distinct threshold values are ordered decreasingly. Hence
\[
\mathcal I_1=\{1,5\},\qquad
\mathcal I_2=\{2,4,7\},\qquad
\mathcal I_3=\{3,6\},
\]
with layer thresholds $2$, $2/3$, and $1/2$, respectively. This last step emphasizes the conceptual distinction between the two phases: the bottom-up phase constructs cluster functions, whereas the top-down phase recovers vertex thresholds; the closure-layer partition is obtained from those thresholds only at the end.

\begin{center}
\renewcommand{\arraystretch}{1.35}
\setlength{\tabcolsep}{7pt}
\begin{tabular}{c|ccccccc}
$v$ & $1$ & $2$ & $3$ & $4$ & $5$ & $6$ & $7$ \\
\hline
$\theta_v$ & $2$ & $\frac23$ & $\frac12$ & $\frac23$ & $2$ & $\frac12$ & $\frac23$
\end{tabular}

\begin{tikzpicture}[x=.85cm,y=.85cm]
  \tikzset{raclayer/.style={rounded corners=.18cm,line width=.39mm,minimum width=2.95cm,
                            minimum height=1.25cm,align=center,font=\small,inner sep=2pt}}
  \node[raclayer,draw=darkgreen,fill=greenback!45] (I1) at (1.8,0)
     {$\mathcal I_1=\{1,5\}$\\[.45mm]$\lambda=2$};
  \node[font=\small] at (3.90,0) {$>$};
  \node[raclayer,draw=orange,fill=yellowback!35] (I2) at (6.0,0)
     {$\mathcal I_2=\{2,4,7\}$\\[.45mm]$\lambda=\frac23$};
  \node[font=\small] at (8.125,0) {$>$};
  \node[raclayer,draw=red,fill=redback!35] (I3) at (10.25,0)
     {$\mathcal I_3=\{3,6\}$\\[.45mm]$\lambda=\frac12$};
\end{tikzpicture}
\captionof*{figure}{\textbf{(n)} Recovered thresholds and induced closure-layer partition.}
\end{center}

\end{document}